\documentclass{dpreprint}
\begin{document}

\dtitle{Stabilizers of Actions of Lattices in Products of Groups}
\dauthor[Darren Creutz]{Darren Creutz}{darren.creutz@vanderbilt.edu}{Vanderbilt University}{}
\datewritten{9 September 2013}

\dabstract{%
\large
We prove that any ergodic nonatomic probability-preserving action of an irreducible lattice in a semisimple group, at least one factor being connected and higher-rank, is essentially free.  This generalizes the result of Stuck and Zimmer \cite{SZ94} that the same statement holds when the ambient group is a semisimple real Lie group and every simple factor is higher-rank.

We also prove a generalization of a result of Bader and Shalom \cite{BS04} by showing that any probability-preserving action of a product of simple groups, at least one with property $(T)$, which is ergodic for each simple subgroup is either essentially free or essentially transitive.

Our method involves the study of relatively contractive maps and the Howe-Moore property, rather than the relaying on algebraic properties of semisimple groups and Poisson boundaries, and introduces a generalization of the ergodic decomposition to invariant random subgroups of independent interest.
}

\makepreprint

\section{Introduction}

Generalizing a particular case of Margulis' breakthrough work \cite{Ma79} showing that irreducible lattices in higher-rank semisimple groups have no nontrivial infinite index normal subgroups, Nevo and Stuck and Zimmer \cite{SZ94},\cite{nevozimmer} showed that irreducible lattices in semisimple real Lie groups, each simple factor having higher-rank, admit no nonatomic actions that are not essentially free (if one takes the Bernoulli shift action of a lattice modulo a normal subgroup and treats it as an action of the lattice, the Nevo-Stuck-Zimmer result then recovers Margulis' result).  However, the question of actions of lattices in semisimple groups in general (allowing $p$-adic and removing the higher-rank assumption) remained open.

Bader and Shalom \cite{BS04}, more recently, proved a normal subgroup theorem for irreducible lattices in products of simple nondiscrete groups: as with lattices in semisimple groups, the only nontrivial normal subgroups of an irreducible lattice in a product of simple nondiscrete groups are all of finite index.  While the methods of Bader and Shalom do provide information about the actions of products of groups (specifically, they obtain that if a product of two groups, both with property $(T)$, acts on a probability space in such a way that each simple factor acts ergodically then the action is either essentially free or essentially transitive), their methods do not yield information about the actions of lattices in products, leaving open the question about lattices in general semisimple groups.

Addressing this issue, the author and Peterson \cite{CP12}, introduced a new method for studying lattices in semisimple groups, based on the commensurator approach developed by the author and Shalom \cite{CS11a},\cite{creutzphd}, and showed that actions of irreducible lattices in products of groups with the Howe-Moore property (in particular, semisimple groups), at least one with property $(T)$, at least one totally disconnected and such that every connected (real) factor has property $(T)$, also only admit essentially free actions on nonatomic probability spaces.  However, the requirement of higher-rank (property $(T)$) remained.

Our purpose here is to present a new proof of the results of Nevo and Stuck and Zimmer and of Bader and Shalom (in particular, without making use of their factor theorems), and to make substantial progress on removing the higher-rank (property $(T)$) requirement.  Unlike the methods in \cite{SZ94} and \cite{BS04}, which focus on the Poisson boundary, and, in the case of Stuck and Zimmer, on algebraic properties of semisimple groups, we follow an approach much more in the spirit of that of \cite{CP12} focusing on contractive spaces and the Howe-Moore property.  Our work here, combined with the work of the author and Peterson in \cite{CP12}, yields the following:
\begin{theorem*}[Corollary \ref{C:85prime}]
Let $G = G_{1}, \times \cdots \times G_{k}$ be a product of at least two simple nondiscrete noncompact locally compact second countable groups with the Howe-Moore property, at least one with property $(T)$ and such that if any of the $G_{j}$ are connected then at least one connected $G_{j}$ has property $(T)$.  Let $\Gamma < G$ be an irreducible lattice and let $(X,\nu)$ be a nonatomic ergodic probability-preserving $\Gamma$-space.  Then $\Gamma \actson (X,\nu)$ has finite stabilizers.
\end{theorem*}

In particular, for semisimple groups we obtain:
\begin{corollary*}[Corollary \ref{C:latticesfinalstate}]
Let $G$ be a semisimple group with trivial center and no compact factors with at least one simple factor being a connected (real) Lie group with property $(T)$ (of higher-rank for example).  Let $\Gamma < G$ be an irreducible lattice and $(X,\nu)$ be a nonatomic ergodic measure-preserving $\Gamma$-space.  Then $\Gamma \actson (X,\nu)$ is essentially free.
\end{corollary*}

Unlike the methods of \cite{CP12}, we do obtain information about the actions of the ambient groups.
In particular, we sharpen the result of Bader and Shalom on actions of products of groups by removing the requirement that all groups have property $(T)$ and instead only requiring one group to have it (we also remove the requirement that the groups without property $(T)$ be simple):
\begin{theorem*}[Corollary \ref{C:bsnew2}]
Let $G$ be a product of a simple locally compact second countable group with property $(T)$ and an arbitrary locally compact second countable group and let $(X,\nu)$ be a faithful measure-preserving $G$-space that is ergodic for both groups.  Then $G \actson (X,\nu)$ is either essentially free or essentially transitive.
\end{theorem*}

Moreover, when both groups do have property $(T)$, our methods allow us to relax the requirement that each factor act ergodically:
\begin{corollary*}[Corollary \ref{C:actprodTreal}]
Let $G_{j}$ be locally compact second countable groups for $j = 1, \ldots, k$ with $k \geq 2$ each with property $(T)$.  Set $G = G_{1} \times \cdots \times G_{k}$ and let $(X,\nu)$ be an ergodic measure-preserving $G$-space.
Assume that there exist closed subgroups $H_{j} < G_{j}$ such that the spaces of $\prod_{\ell \ne j}G_{\ell}$-ergodic components is isomorphic to $(G_{j}/H_{j},\mathrm{Haar})$ for each $j$  and such that any nontrivial normal subgroup of $H_{j}$ has finite index in $H_{j}$.  Then either at least one $G_{j} \actson (X,\nu)$ essentially free or $G \actson (X,\nu)$ is essentially transitive.
\end{corollary*}

Even when the ambient group does not have any property $(T)$ (for example, lattices in products of at least two rank-one groups), one can obtain some information about the action.  If $G$ is a product of at least two simple locally compact second countable groups or be a semisimple group with at least two factors and $\Gamma$ is an irreducible lattice in $G$ then any ergodic nonatomic probability-preserving action of $\Gamma$ is either essentially free or is weakly amenable and any ergodic probability-preserving action of $G$ that is ergodic for each simple factor of $G$ is either essentially free or weakly amenable (Corollary \ref{C:bsnew2} and Theorem \ref{T:83prime} combined with the results in \cite{CP12}).  Note that weak amenability of the action implies that almost every stabilizer subgroup is coamenable (see Remark \ref{R:coamen}).  We mention that the previous statement for semisimple real Lie groups (without any higher-rank assumption) and irreducible lattices in semisimple real Lie groups is implicit already in \cite{SZ94}.

\subsection{Stabilizers of Actions and Random Subgroups}

Invariant random subgroups are the natural setting for the study of stabilizers of probability-preserving actions.  The study of stabilizers goes back at least to Moore, \cite{moore66} Chapter 2, and Ramsay, \cite{ramsay} Section 9 (see also Adams and Stuck \cite{adamsstuck} Section 4).
Bergeron and Gaboriau \cite{gaboriau} noticed the similarities between normal subgroups and invariant random subgroups and recently this has been the focus of much attention: \cite{seven}, \cite{AGV12}, \cite{bowen12}, \cite{CP12}, \cite{kostya}, \cite{grig1}, \cite{grig2}, \cite{td1}, \cite{td2}, \cite{vershik11}, \cite{vershiktotallynonfree}.
Random subgroups play a key role in our technique, and we prove several results involving invariant random subgroups in general that are of independent interest.

Given a measure-preserving action of a group on a probability space, the pushforward of the measure to the space of closed subgroups obtained by mapping each point to its stabilizer subgroup gives rise to a conjugation-invariant probability measure on the space of subgroups, that is, gives rise to an invariant random subgroup.
As shown in \cite{AGV12} (for discrete groups) and \cite{CP12} (for the locally compact case), every invariant random subgroup arises in this way.  We generalize this result in two directions.  Firstly, we show that every quasi-invariant random subgroup (probability measure on the space of closed subgroups that is quasi-invariant under conjugation) arises as the stabilizers of a quasi-invariant action of the group.  Secondly, we study the notion of subgroups of random subgroups (introduced in \cite{CP12}, we say that a random subgroup $\alpha$ is a subgroup of a random subgroup $\beta$ when there exists a joining $\rho$ of $\alpha$ and $\beta$ such that for $\rho$-almost every $(H,K)$ it holds that $H$ is a subgroup of $K$).  Given an equivariant map $(X,\nu) \to (Y,\eta)$ of $G$-spaces, the pushforward measure $\stab_{*}\nu$ is evidently a subgroup of $\stab_{*}\eta$.  We show that such pairs of random subgroups always arise as the stabilizers of actions of quotient maps of $G$-spaces (Theorem \ref{T:extensionsrg}).

Building on this, we introduce the notion of a free extension of an action by a random subgroup: given an action of a group $G$ on a probability space $(X,\nu)$ and given a $G$-equivariant map $\varphi$ from $X$ to random subgroups of $G$ such that $\varphi(x)$ is supported on subgroups of $\stab(x)$ almost everywhere, we show there exists an a $G$-space, the free extension of $(X,\nu)$ by $\varphi$, having stabilizers equal to $\varphi_{*}\nu$ such that $(X,\nu)$ is a quotient of this space in a canonical way.  In particular, given a subgroup of $\stab_{*}\nu$, there is always an extension of $(X,\nu)$ having stabilizers given by that subgroup.

Crucial to our work here, we also introduce the notion of the quotient of a space by a random subgroup, a generalization of the ergodic decomposition for a normal subgroup.  Recall that if $N$ is a normal subgroup of $G$ and $(X,\nu)$ is a $G$-space, one can define the space of $N$-ergodic components by considering the algebra of $N$-invariant functions.  We generalize this to considering the ``invariant functions for a random subgroup below $(X,\nu)$" (defined precisely in section \ref{s:quotrs}) and prove various universality properties of the quotient space.  We then apply this quotienting procedure to actions of products of groups by considering the random subgroup obtained by taking the projections of the stabilizer groups to each factor.  This functor, the product random subgroups functor (see section \ref{s:prg}), allows us to study action of products of groups at a much finer level of detail than the ergodic decomposition functor used by Bader and Shalom.

\subsection{The Contractive Factor Theorem for Products}

The second major ingredient in our work is a factor theorem for actions of products of groups based on the notion of relatively contractive maps introduced in \cite{CP12}.  Our factor theorem generalizes the Bader-Shalom factor theorem and allows us to study actions of products of groups which are not necessarily ergodic when restricted to each factor:
\begin{theorem*}[Theorem \ref{T:IFT}]
Let $G = G_{1} \times G_{2}$ be a product of two locally compact second countable groups and let $\mu_{j} \in P(G_{j})$ be admissible probability measures for $j=1,2$.  Set $\mu = \mu_{1} \times \mu_{2}$.

Let $(B,\beta)$ be the Poisson boundary for $(G,\mu)$ and let $(X,\nu)$ be a measure-preserving $G$-space.  Let $(W,\rho)$ be a $G$-space such that there exist $G$-maps $\pi : (B \times X, \beta\times\nu) \to (W,\rho)$ and $\varphi : (W,\rho) \to (X,\nu)$ with $\varphi \circ \pi$ being the natural projection to $X$.

Let $(W_{1},\rho_{1})$ be the space of $G_{2}$-ergodic components of $(W,\rho)$ and let $(W_{2},\rho_{2})$ be the space of $G_{1}$-ergodic components.  Likewise, let $(X_{1},\nu_{1})$ and $(X_{2},\nu_{2})$ be the ergodic components of $(X,\nu)$ for $G_{2}$ and $G_{1}$, respectively.

Then $(W,\rho)$ is $G$-isomorphic to the independent relative joining of $(W_{1},\rho_{1}) \times (W_{2},\rho_{2})$ and $(X,\nu)$ over $(X_{1},\nu_{1}) \times (X_{2},\nu_{2})$.
\end{theorem*}

The factor theorem of Bader and Shalom requires that $X_{1}$ and $X_{2}$ be trivial and can be phrased as saying that in that case $W$ is always isomorphic to the independent joining of $W_{1} \times W_{2}$ and $X$.  Their theorem follows from a careful study of properties of the Poisson boundary.  Our theorem, on the other hand, only makes use of two (easy) properties of the Poisson boundary: that it contractive and that it an amenable space.  Replacing the study of boundary dynamics, we make use of a result about uniqueness of relatively contractive joinings which may be of independent interest:
\begin{theorem*}[Corollary \ref{C:reljoinunique}]
Let $G$ be a locally compact second countable group and let $(X,\nu)$, $(Y,\eta)$, $(Z,\zeta)$ and $(W,\rho)$ be $G$-spaces such that the following diagram of $G$-maps commutes:
\begin{diagram}
(W,\rho)		&\rTo	&(X,\nu)\\
\dTo	&			&\dTo\\
(Y,\eta)		&\rTo	&(Z,\zeta)
\end{diagram}
If the vertical maps are relatively measure-preserving and the horizontal maps are relatively contractive then $(W,\rho)$ is isomorphic to the independent relative joining of $(X,\nu)$ and $(Y,\eta)$ over $(Z,\zeta)$.
\end{theorem*}

This contractive factor theorem allows us to study actions of lattices in products of groups by inducing the action to the ambient group and considering intermediate factors.  Since in general the induced action will not be ergodic for each factor, our theorem allows us to study lattices where the Bader-Shalom theorem does not.  Combining the induced action with the ``projected action" (see section \ref{S:projact}) obtained by considering the stabilizers of the original action of the lattice and projecting them to each factor and then taking free extensions by the corresponding random subgroups, we obtain enough information about the stabilizers of the induced action to study the stabilizers of the original action.

\subsection[Relaxing the Property (T) Requirement]{Relaxing the Property $(T)$ Requirement}

The third major ingredient in our work is the use of a type of relative property $(T)$, in the form of resolutions (introduced by de Cornulier \cite{yves}), to relax the requirement (present in all previous work on the subject of actions of semisimple groups and lattices) that every simple factor have property $(T)$.  In the work of Stuck and Zimmer, Bader and Shalom, and the author and Peterson, the requirement of the ambient group having property $(T)$, and therefore the lattice also having property $(T)$, was a necessary step in moving from knowing the equivalence relation of an action is amenable to knowing the action is essentially transitive.

We develop a new approach to the study of actions generating an amenable orbit equivalence relation when the group involved does not have property $(T)$ but has ``some" property $(T)$ (in the case of products, one factor having property $(T)$ and in the case of lattices, admitting a resolution which in turn comes from one factor having property $(T)$).  As an example, we obtain the following statement:
\begin{theorem*}[Theorem \ref{T:resolutions}]
Let $G = G_{1} \times G_{2}$ be a product of two locally compact second countable groups such that $G_{2}$ has property $(T)$.  Let $(X,\nu)$ be an ergodic measure-preserving $G$-space such that $G \actson (X,\nu)$ weakly amenably and not essentially transitively.  Let $\mathcal{H}$ be the subspace of $L^{2}(X,\nu)$ consisting of the $G_{2}$-invariant functions that are not $G$-invariant.  Then there exists a sequence of almost invariant vectors in $\mathcal{H}$.
\end{theorem*}

The previous theorem immediately implies that if both $G_{1}$ and $G_{2}$ act ergodically and $G$ acts weakly amenably then the action is essentially transitive.  The same ideas allow us to conclude a similar result for actions of lattices in such products.

\section{Preliminaries}

\subsection[G-Spaces and G-Maps]{$G$-Spaces and $G$-Maps}

\begin{definition}
Let $G$ be a locally compact second countable group.  A \textbf{$G$-space} is a probability space $(X,\nu)$ equipped with an action of $G$ such that $\nu$ is quasi-invariant under the action (the class of null sets is preserved by the action).  This will be written $G \actson (X,\nu)$.
\end{definition}

\begin{definition}
Let $G$ be a locally compact group and $G \actson (X,\nu)$ a $G$-space.  The \textbf{translate of $\nu$ by $g \in G$} is the probability measure $g\nu$ defined by $g\nu(E) = \nu(g^{-1}E)$ for all measurable sets $E$.  If $\mu \in P(G)$ is a probability measure on $G$, the \textbf{convolution of $\nu$ by $\mu$} is the probability measure $\mu * \nu \in P(X)$ given by
\[
\mu * \nu(E) = \int_{G} g\nu(E)~d\mu(g) = \int_{G} \nu(g^{-1}E)~d\mu(g).
\]
\end{definition}

\begin{definition}
Let $G$ be a locally compact second countable group and $(X,\nu)$ a $G$-space.  Then $(X,\nu)$ is \textbf{measure-preserving} when $g\nu = \nu$ for all $g \in G$.  If $\mu \in P(G)$ is a probability measure on $G$ such that $\mu * \nu = \nu$ then $(X,\nu)$ is \textbf{$\mu$-stationary}.
\end{definition}

\begin{definition}
Let $G$ be a locally compact second countable group and let $(X,\nu)$ and $(Y,\eta)$ be $G$-spaces.  A measurable map $\pi : X \to Y$ such that $\pi_{*}\nu = \eta$ is a \textbf{$G$-map} when $\pi$ is $G$-equivariant: $\pi(gx) = g\pi(x)$ for all $g \in G$ and almost every $x \in X$ (here $\pi_{*}\nu$ is the \textbf{pushforward measure} defined by, for $E$ a measurable subset of $Y$, $\pi_{*}\nu(E) = \nu(\pi^{-1}(E))$).
\end{definition}

\begin{definition}
Let $G$ be a locally compact second countable group and $\pi : (X,\nu) \to (Y,\eta)$ a $G$-map of $G$-spaces.  The \textbf{disintegration} of $\nu$ over $\eta$ is the almost everywhere unique map $D_{\pi} : Y \to P(X)$ such that the support of $D_{\pi}(y)$ is contained in $\pi^{-1}(y)$ and that $\int_{Y} D_{\pi}(y)~d\eta(y) = \nu$.
\end{definition}

\begin{definition}
Let $G$ be a locally compact second countable group and $\pi : (X,\nu) \to (Y,\eta)$ a $G$-map of $G$-spaces.  Then $\pi$ is \textbf{relatively measure-preserving} when the disintegration of $\nu$ over $\eta$ via $\pi$, $D_{\pi} : Y \to P(X)$, is $G$-equivariant: $D_{\pi}(gy) = gD_{\pi}(y)$ for all $g \in G$ and almost every $y \in Y$.
\end{definition}

We also need the following well-known characterization of relatively measure-preserving:
\begin{theorem}\label{T:relmpRN}
Let $G$ be a locally compact second countable group and $\pi : (X,\nu) \to (Y,\eta)$ a $G$-map of $G$-spaces.  Then $\pi$ is relatively measure-preserving if and only if the Radon-Nikodym derivatives satisfy $\frac{dg\nu}{d\nu}(x) = \frac{dg\eta}{d\eta}(\pi(x))$ almost everywhere.
\end{theorem}
\begin{proof}
By uniqueness of the Radon-Nikodym derivative,
\[
\frac{dg\nu}{d\nu}(x) = \frac{d gD_{\pi}(g^{-1}\pi(x))}{d D_{\pi}(\pi(x))}(x)\frac{dg\eta}{d\eta}(\pi(x)).
\]
Therefore $\frac{dg\nu}{d\nu}(x) = \frac{dg\eta}{d\eta}(\pi(x))$ if and only if $\frac{d gD_{\pi}(g^{-1}\pi(x))}{d D_{\pi}(\pi(x))}(x) = 1$ almost surely which says precisely that $\pi$ is relatively measure-preserving.
\end{proof}

\begin{definition}
Let $G$ be a locally compact second countable group and $(X,\nu)$ a $G$-space.  Let $\mathcal{F} \subseteq L^{\infty}(X,\nu)$ be a closed $G$-invariant subalgebra.  A \textbf{point realization} or \textbf{Mackey point realization} of $\mathcal{F}$ is a $G$-space $(Y,\eta)$ where there exists a $G$-map $\pi : (X,\nu) \to (Y,\eta)$ such that $\mathcal{F} = \{ f \circ \pi : f \in L^{\infty}(Y,\eta) \}$.
\end{definition}

\begin{theorem}[Mackey \cite{Ma62}, \cite{mackey}]\label{T:mackey}
Let $G$ be a locally compact second countable group and $(X,\nu)$ a $G$-space.  Let $\mathcal{F} \subseteq L^{\infty}(X,\nu)$ be a closed $G$-invariant subalgebra.  Then there exists a point realization of $\mathcal{F}$.
\end{theorem}

\subsection{Stabilizers of Actions}

\begin{definition}
Let $G$ be a locally compact second countable group and $(X,\nu)$ a $G$-space.  The \textbf{stabilizer subgroups} are $\stab(x) = \{ g \in G : gx = x \}$ for each $x \in X$.
\end{definition}

\begin{definition}
Let $G$ be a locally compact second countable group and $(X,\nu)$ a $G$-space.  Then $G \actson (X,\nu)$ is \textbf{essentially free} when  $\stab(x) = \{ e \}$ for almost every $x \in X$.
\end{definition}

\begin{definition}
Let $G$ be a locally compact second countable group and $(X,\nu)$ a $G$-space.  Then $G \actson (X,\nu)$ is \textbf{faithful} when the kernel of the action is trivial.
\end{definition}

\begin{definition}
Let $G$ be a locally compact second countable group and $(X,\nu)$ a $G$-space.  Then $G \actson (X,\nu)$ is \textbf{essentially transitive} when there exists $x_{0} \in X$ such that $\nu(G \cdot x_{0}) = 1$ (a full measure orbit).
\end{definition}

\begin{definition}
Let $G$ be a locally compact second countable group and $\pi : (X,\nu) \to (Y,\eta)$ a $G$-map of $G$-spaces.  Then $\pi$ is \textbf{orbital} when $\stab(x) = \stab(\pi(x))$ almost everywhere.
\end{definition}

\subsection{Weakly Amenable Actions}

Introduced by Zimmer \cite{Zi77}, the notion of weakly amenable actions will play a crucial role in our study of the stabilizers of actions of groups.

\begin{definition}
Let $G \actson (X,\nu)$ be a quasi-invariant action of a locally compact second countable group.  Let $E$ be a separable Banach space and write $E_{1}^{*}$ for the unit ball in the dual of $E$.  Let $\alpha : G \times X \to \mathrm{Iso}(E)$ be a cocycle.  Denote the dual cocycle $\alpha^{*}$ by $\alpha^{*}(g,x) = (\alpha(g,x)^{-1})^{*}$.  Let $A_{x} \subseteq E_{1}^{*}$ be a closed convex nonempty set for almost every $x$ such that $\alpha^{*}(g,x)A_{gx} = A_{x}$.  Consider the space
\[
A = \bigsqcup_{x} \{ x \} \times A_{x} \subseteq X \times_{\alpha^{*}} E_{1}^{*}
\]
endowed with the $\alpha^{*}$-twisted action.
This is a closed compact space which is $G$-invariant under the $\alpha^{*}$-twisted action.
Such a space $A$ is called an \textbf{affine $G$-space over $(X,\nu)$}.
\end{definition}

\begin{definition}
The cocycle $\alpha$ is called \textbf{orbital} when $\alpha(g,x) = e$ for all $g \in \mathrm{stab}_{G}(x)$ for almost every $x$.  The affine $G$-space $A$ is called an \textbf{orbital affine $G$-space} when $\alpha$ is orbital.
\end{definition}

\begin{definition}
Let $G$ be a locally compact second countable group and $(X,\nu)$ a $G$-space.
$G \actson (X,\nu)$ is \textbf{amenable} when for every affine $G$-space $A$ over $(X,\nu)$ there exists an $\alpha^{*}$-invariant function $f : X \to E_{1}^{*}$ such that $f(x) \in A_{x}$ for almost every $x$: $\alpha^{*}$-invariant means $f(x) = \alpha^{*}(g,x)f(gx)$.  $G \actson (X,\nu)$ is \textbf{weakly amenable} when that condition holds for all orbital affine $G$-spaces over $(X,\nu)$.
\end{definition}


\begin{proposition}[Stuck-Zimmer \cite{SZ94}]\label{P:weakamenmaps}
Let $G$ be a locally compact second countable group, $(X,\nu)$ a $G$-space and $(B,\beta)$ an amenable $G$-space.  Let $A$ be an affine $G$-space over $X$.  Then there exists $G$-maps
\[
(B \times X, \beta \times \nu) \to (A,\alpha) \to (X,\nu)
\]
such that the composition is the natural projection to $X$ and $\alpha$ is the pushforward of $\beta \times \nu$.
\end{proposition}

The proof of the previous statement is implicit in \cite{SZ94}; the reader is referred to \cite{CP12} for a concrete proof.

\subsection{Amenable Equivalence Relations}

The notion of amenability for equivalence relations was introduced by Zimmer in \cite{Zi77}.  The reader is also referred to Kechris and Miller \cite{kechris} for more detailed information.

\begin{definition}
Let $G$ be a locally compact second countable group and $(X,\nu)$ be a measure-preserving $G$-space.  The \textbf{orbit equivalence relation} generated by the action is given by $R$ where $xRy$ when there exists $g \in G$ such that $gx = y$.
\end{definition}

\begin{definition}
Let $(X,\nu)$ be a measure space.  An equivalence relation $R \subseteq X \times X$ is \textbf{measurable} when there exists a $\sigma$-finite measure $\rho$ on $R$ such that the projection $R \to X$ sends $\rho$ to $\nu$.
\end{definition}

When $G$ is a locally compact second countable group and $(X,\nu)$ a measure-preserving $G$-space, the orbit equivalence relation generated by the action is measurable: let $R = \{ (x, gx) : x \in X, g \in G \}$ be the equivalence relation generated by the action and let $m$ be a Haar measure on $G$ and consider the map $p : X \times G \to R$ given by $p(x,g) = (x,gx)$; then $\rho = p_{*}(\nu \times m)$ is $\sigma$-finite and makes $R$ measurable.

\begin{definition}
Let $(X,\nu)$ be a $\sigma$-finite measure space.  A \textbf{mean} on $(X,\nu)$ is a linear functional $m \in L^{\infty}(X,\nu)^{*}$ that is positive and has $m(\bbone) = 1$.
\end{definition}

\begin{definition}
Let $(X,\nu)$ be a measure space and $R \subseteq X \times X$ be a measurable equivalence relation.  A map $m : x \mapsto m_{x}$ is a \textbf{mean} on $R$ when for almost every $x \in X$, $m_{x}$ is a mean on $[x]$, the equivalence class of $x$, and the map $x \mapsto m_{x}$ is measurable in the sense that for any $F \in L^{\infty}(R)$, writing $F_{x} : [x] \to \mathbb{R}$ by $F_{x}(y) = F(x,y)$, it holds that $x \mapsto m_{x}(F_{x})$ is a measurable map.
\end{definition}

\begin{definition}
Let $(X,\nu)$ be a measure space and $R \subseteq X \times X$ be a measurable equivalence relation.  A map $m : x \mapsto m_{x}$ is an \textbf{invariant mean} when it is a mean such that $m_{x} = m_{y}$ for almost every $x \in X$ and all $y \in [x]$.
\end{definition}

\begin{definition}
The equivalence relation $R_{G \actson (X,\nu)}$ is \textbf{amenable} when there exists an invariant mean for $R_{G \actson (X,\nu)}$.
\end{definition}

\begin{remark}\label{R:coamen}
Recall that a subgroup $H < G$ is said to be \textbf{coamenable} in $G$ when there is a $G$-invariant mean on $G / H$.  If $G \actson (X,\nu)$ gives rise to an amenable equivalence relation then for almost every $x \in X$, the stabilizer subgroup $\stab(x)$ is coamenable in $G$ (since the orbit $[x]$ is isomorphic to $G / \stab(x)$).
\end{remark}

\begin{theorem}[Zimmer \cite{Zi77}]\label{T:amener}
Let $G$ be a locally compact second countable group and $(X,\nu)$ be a measure-preserving $G$-space.  Then the orbit equivalence relation of $G \actson (X,\nu)$ is amenable if and only if $G \actson (X,\nu)$ is weakly amenable.
\end{theorem}

\begin{theorem}[Connes-Feldman-Weiss \cite{cfw}]\label{T:cfw}
Let $G$ be a locally compact second countable group and $(X,\nu)$ be an ergodic measure-preserving $G$-space.  Then the orbit equivalence relation of $G \actson (X,\nu)$ is amenable if and only if $G \actson (X,\nu)$ is orbit equivalent to a free ergodic action of $\mathbb{R}$ or $\mathbb{Z}$ depending on whether $G$ is discrete (two actions $G \actson (X,\nu)$ and $H \actson (Y,\eta)$ are orbit equivalent when there exists a measure-space isomorphism $\theta : (X,\nu) \to (Y,\eta)$ such that for all $g \in G$ and almost every $x \in X$ there exists $h \in H$ such that $\theta(gx) = h\theta(x)$).
\end{theorem}

\begin{proposition}\label{P:weakamen}
Let $G$ be a locally compact second countable group and $\pi : (Y,\eta) \to (X,\nu)$ be a $G$-map of $G$-spaces such that $G \actson (X,\nu)$ is weakly amenable and $\pi$ is orbital.  Then $G \actson (Y,\eta)$ weakly amenably.
\end{proposition}
\begin{proof}
Let $R_{Y} = \{ (y, gy) : y \in Y, g \in G \}$ and $R_{X} = \{ (x,gx) : x \in X, g \in G \}$ be the equivalence relations of the actions.  Define the set $S = \{ (y, g\pi(y)) : y \in Y, g \in G \} \subseteq Y \times X$.  For each $y \in Y$, the map $[y] \to [\pi(y)]$ given by $gy \mapsto g\pi(y)$ is one-one since $\pi$ is orbital.  Then the map $\psi : S \to R_{Y}$ by $\psi(y,g\pi(y)) = (y,gy)$ is a well-defined measurable map.  Let $m$ be an invariant mean on $R_{X}$.  Define a map $M : y \mapsto M_{y}$ as follows: for $F \in L^{\infty}(R_{Y})$ consider $F \circ \psi : S \to \mathbb{R}$ and write $(F \circ \psi)_{y} : [\pi(y)] \to \mathbb{R}$ as $(F \circ \psi)_{y}(g\pi(y)) = F(\psi(y,g\pi(y))) = F(y,gy) = F_{y}(gy)$.  Define $M_{y}(F_{y}) = m_{\pi(y)}((F \circ \psi)_{y})$.  Then $y \mapsto M_{y}(F_{y})$ is measurable since $y \mapsto \pi(y) \mapsto m_{\pi}(y)$ is measurable and $\psi$ is measurable.  Then $M$ is a mean on $R_{Y}$.  Also, $F_{gy} = F_{y}$ and $m_{\pi(gy)} = m_{g\pi(y)} = m_{\pi(y)}$ so $M$ is invariant.  Hence $G \actson (Y,\eta)$ is weakly amenable.
\end{proof}

\begin{proposition}\label{P:weakamenprod}
Let $G_{1}$ and $G_{2}$ be locally compact second countable groups and let $(X_{1},\nu_{1})$ be a weakly amenable $G_{1}$-space and $(X_{2},\nu_{2})$ a weakly amenable $G_{2}$-space.  Then $(X_{1} \times X_{2}, \nu_{1} \times \nu_{2})$ is a weakly amenable $G_{1} \times G_{2}$-space (with the product action).
\end{proposition}
\begin{proof}
By Theorem \ref{T:amener}, there exist measurable maps on $X_{1}$ and $X_{2}$, written $x_{1} \mapsto m_{x_{1}}$ and $x_{2} \mapsto m_{x_{2}}$, such that $m_{x_{j}}$ is a mean on $L^{\infty}([x_{j}])$ where $[x_{j}]$ is the $G_{j}$-orbit of $x_{j}$ and such that $m_{y_{j}} = m_{x_{j}}$ for all $y_{j} \in [x_{j}]$.  Define the map $(x_{1}, x_{2}) \mapsto m_{x_{1},x_{2}}$ by defining $m_{x_{1},x_{2}}(f_{1} \times f_{2}) = m_{x_{1}}(f_{1}) m_{x_{2}}(f_{2})$ and extending by continuity and linearity to all of $L^{\infty}(X_{1} \times X_{2}, \nu_{1} \times \nu_{2})$.  Then $m_{x_{1},x_{2}}$ are means on $[x_{1},x_{2}]$ and the map is measurable and invariant under the orbit of $G_{1} \times G_{2}$.  So by Theorem \ref{T:amener} and Proposition \ref{P:weakamenprod}, the claim follows.
\end{proof}


\subsection{Irreducible Lattices}

\begin{definition}
Let $G$ be a locally compact second countable group.  A subgroup $\Gamma < G$ is a \textbf{lattice} when it is discrete and has finite covolume (there exists an open set $F \subseteq G$ such that $F\Gamma = G$, $F \cap \Gamma = \{ e \}$ and $\Haar(F) < \infty$).
\end{definition}

\begin{definition}
A lattice $\Gamma$ in a locally compact second countable group $G$ is \textbf{irreducible} when for any noncentral closed normal subgroup $M \normal G$ that is not cocompact, $\Gamma / (\Gamma \cap M)$ is dense in $G / M$.
\end{definition}

Central and cocompact normal subgroups are excepted in the definition to allow for cases such as $\SL_{n}(\mathbb{Z}) < \SL_{n}(\mathbb{R})$ with $M$ being the center and cases such as $G = H \times K$ where $K$ is compact and $\Gamma = \Gamma_{0} \times \{ e \}$ where $\Gamma_{0}$ is an irreducible lattice in $H$.

\begin{proposition}\label{P:irrlattint}
Let $\Gamma < G \times H$ be an irreducible lattice in a product of noncompact nondiscrete locally compact second countable groups.  Then $\Gamma \cap (\{ e \} \times H)$ is contained in $\{ e \} \times Z(H)$ where $Z(H)$ is the center of $H$.
\end{proposition}
\begin{proof}
Let $N = \Gamma \cap (\{ e \} \times H)$.  Then $N \normal \Gamma$ since $\{ e \} \times H \normal G \times H$.  Therefore $\overline{\mathrm{proj}_{H}~N} \normal \overline{\mathrm{proj}_{H}~\Gamma}$.  Since $\Gamma$ is irreducible and $G$ and $H$ are not compact, $\overline{\mathrm{proj}_{H}~\Gamma} = H$.  On the other hand, $N \subseteq \{ e \} \times H$ so write $N = \{ e \} \times M$ for some $M \normal H$ and observe that $\mathrm{proj}_{H}~N = M$.  Now $M$ is discrete in $H$ since $\Gamma$ is discrete in $G \times H$.  Hence $N \normal G \times H$ is a discrete (hence closed) normal subgroup.

If $N$ is central then $N < Z(G \times H) \cap (\{ e \} \times H) = \{ e \} \times Z(H)$.  So we may assume $N$ is noncentral.  Note that $N$ is not cocompact since $G$ is noncompact.  Therefore the projection of $\Gamma$ to $(G \times H) / N$ is dense: $\Gamma / N$ is dense in $(G \times H) / N$.

On the other hand, the quotient map $G \times H \to (G \times H) / N$ is an open map so if $U$ is an open neighborhood of $e$ in $G \times H$ such that $U \cap \Gamma = \{ e \}$ then the image of $U$ in $(G \times H) / N$ is an open neighborhood of the identity intersecting $\Gamma / N$ only at the identity.  Hence $\Gamma / N$ is discrete in $(G \times H) / N$.

But then $\Gamma / N$ is both dense and discrete in $(G \times H) / N$ hence $(G \times H) / N$ is discrete.  As $N$ is also discrete, this would mean that $G \times H$ is discrete, contradicting our hypotheses.
\end{proof}

\subsection{Induced Actions}

Let $G$ be a locally compact second countable group and $\Gamma < G$ a lattice.  Given a $\Gamma$-space $(X,\nu)$, take $F$ to be a fundamental domain for $G / \Gamma$ with the normalized Haar measure $m$ and define the $G$-space $G \times_{\Gamma} X$ to be $(F \times X, m \times \nu)$ with the action $g \cdot (f,x) = (gf\alpha(g,f), \alpha(g,f)^{-1}x)$ where $\alpha : G \times F \to \Gamma$ is the cocycle such that $gf\alpha(g,f) \in F$ for all $g \in G$ and $f \in F$.  This construction is independent (up to $G$-isomorphism) of the fundamental domain chosen.  The space $G \times_{\Gamma} X$ is the \textbf{induced action}.  Note that it is measure-preserving when $(X,\nu)$ is measure-preserving.

Let $(f,x) \in F \times X$.  Observe that $g \cdot (f,x) = (f,x)$ if and only if $gf\alpha(g,f) = f$ and $\alpha(g,f)^{-1}x = x$.  Therefore
\[
\stab_{G}(f,x) = f \stab_{\Gamma}(x) f^{-1}
\]
for all $(f,x) \in F \times X$.

\begin{proposition}\label{P:weakamenlattice}
Let $\Gamma$ be a lattice in a locally compact second countable group $G$ and let $(X,\nu)$ be a $\Gamma$-space.  Then $\Gamma \actson (X,\nu)$ weakly amenably if and only if $G \actson G \times_{\Gamma} X$ weakly amenably.
\end{proposition}
\begin{proof}
Assume that $\Gamma \actson (X,\nu)$ weakly amenably.  Then by Theorem \ref{T:amener}, the orbit equivalence relation is amenable so there exists a measurable map $x \mapsto m_{x}$ such that $m_{x}$ is a mean on $\ell^{\infty}[x]$ where $[x]$ is the $\Gamma$-orbit of a point $x$ and such that $m_{y} = m_{x}$ for all $y \in [x]$.

Let $F$ be a fundamental domain for $G / \Gamma$ with cocycle $\alpha : G \times F \to \Gamma$ such that $gf\alpha(g,f) \in F$ and let $\rho$ be the normalized Haar measure on $F$.  Observe that the $G$-orbit of a point $(f,x)$ is $G \cdot (f,x) = F \times [x]$.  Given $q \in L^{\infty}(F \times [x], \rho \times \mathrm{count})$ (where $\mathrm{count}$ is the counting measure on $[x]$), write $q_{f}(x) = q(f,x)$ to be the fiber of $q$ over $f \in F$.  Then $q_{f} \in \ell^{\infty}[x]$ for almost every $f \in F$.  Define a mean $M_{f,x}$ on $L^{\infty}(F \times [x])$ by
\[
M_{f,x}(q) = \int_{F} m_{x}(q_{f_{0}})~d\rho(f_{0}).
\]
One easily checks that $M_{f,x}(\bbone) = 1$ and that $M_{f,x} \geq 0$ since $m_{x}$ is a mean.  Observe that
\[
M_{g \cdot (f,x)}(q) = M_{gf\alpha(g,f), \alpha(g,f)^{-1}x}(q)
= \int_{F} m_{\alpha(g,f)^{-1}x}(q_{f_{0}})~d\rho(f_{0}) = \int_{F} m_{x}(q_{f_{0}})~d\rho(f_{0}) = M_{f,x}(q)
\]
so $M_{f,x}$ is invariant.  The map $(f,x) \mapsto M_{f,x}$ is measurable since $x \mapsto m_{x}$ is (and $M_{f,x}$ does not depend on $f$).  Therefore the orbit equivalence relation of $G \actson G \times_{\Gamma} X$ is amenable hence the action is weakly amenable by Theorem \ref{T:amener}.

Conversely, assume that $G \actson G \times_{\Gamma} X$ is weakly amenable.  Let $(f,x) \mapsto M_{f,x}$ be an invariant mean.  Define $m_{x}$ by, for $q \in \ell^{\infty}[x]$, set $\widetilde{q}(f,x) = q(x) \in L^{\infty}(F \times [x])$ and set
\[
m_{x}(q) = \int_{F} M_{f,x}(\widetilde{q})~d\rho(f).
\]
Then $x \mapsto m_{x}$ is measurable since $(f,x) \mapsto M_{f,x}$ is and $m_{x}$ is a mean.  Clearly for $\gamma \in \Gamma$,
\[
m_{\gamma x}(q) = \int_{F} M_{f,\gamma x}(\widetilde{q})~d\rho(f)
= \int_{F} M_{(f \gamma f^{-1}) \cdot (f,x)}(\widetilde{q})~d\rho(f)
= \int_{F} M_{f,x}(\widetilde{q})~d\rho(f) = m_{x}(q)
\]
using that $M_{f,x}$ is invariant under the $G$-action.  Therefore $\Gamma \actson (X,\nu)$ has an amenable orbit equivalence relation hence acts weakly amenably.
\end{proof}

\subsection{The Howe-Moore Property}

\begin{definition}[Howe-Moore \cite{HM79}]
A locally compact second countable group $G$ has the \textbf{Howe-Moore property} when every irreducible unitary representation of $\pi : G \to \mathcal{U}(\mathcal{H})$ without nontrivial invariant vectors has matrix coefficients vanishing at infinity: $\lim_{g \to \infty} \langle \pi(g)x, y \rangle = 0$ as $g$ leaves compact sets for any $x,y \in \mathcal{H}$.
\end{definition}

\begin{theorem}[Schmidt \cite{schmidtmixing}]\label{T:HMmixing}
Let $G$ be a locally compact second countable group.  Then $G$ has the Howe-Moore property if and only if every ergodic $G$-space is mixing: if $(X,\nu)$ is a $G$-space then $\lim_{g \to \infty} \nu(gE \cap F) = \nu(E)\nu(F)$ as $g$ leaves compact sets for all measurable $E,F \subseteq X$.
\end{theorem}

The main result of \cite{CP12} is:
\begin{theorem}[Creutz-Peterson \cite{CP12}]\label{T:CP12}
Let $G$ be a product of at least two simple nondiscrete noncompact locally compact second countable groups with the Howe-Moore property, at least one of which has property $(T)$, at least one of which is totally disconnected and such that every connected simple factor has propertyy $(T)$.  Let $\Gamma < G$ be an irreducible lattice.  Then any ergodic measure-preserving action of $\Gamma$ has finite stabilizers almost surely or finite index stabilizers almost surely.
\end{theorem}

\subsection{Semisimple Groups}

The main class of groups having the Howe-Moore property are the simple real and $p$-adic Lie groups.  Semisimple groups are almost direct products of such groups and as such serve as a main example of our results.  We remark that automorphism groups of regular trees also have the Howe-Moore property and so serve as another example.

\begin{definition}
A \textbf{semisimple group} is an almost direct product of simple real and $p$-adic Lie groups.
\end{definition}

\begin{theorem}[Rothman \cite{rothman}]\label{T:connHMLie}
Let $G$ be a simple connected locally compact second countable group with the Howe-Moore property.  Then $G$ is a simple real Lie group.
\end{theorem}

\begin{theorem}[Howe-Moore \cite{HM79}]
Every simple real and $p$-adic Lie group has the Howe-Moore property.
\end{theorem}

\begin{theorem}[Zimmer \cite{zimmer2}]\label{T:countableessfree}
Let $G$ be a noncompact nondiscrete simple real Lie group and let $(X,\nu)$ be a nontrivial ergodic measure-preserving $G$-space.  Let $\Lambda$ be any countable subgroup of $G$.  Then the restriction of the action to $\Lambda$ on $(X,\nu)$ is essentially free.
\end{theorem}

\noindent A direct, easy proof of the previous statement appears in \cite{CP12} though it follows from the work of Zimmer in \cite{zimmer2}.

\subsection{Ergodic Decomposition}

Given a $G$-space $(X,\nu)$, consider the $G$-invariant subalgebra of invariant functions $\mathcal{F} = \{ f \in L^{\infty}(X,\nu) : g \cdot f = f \text{ for all $g \in G$} \}$ and let $(\rpf{X}{G},\overline{\nu})$ be the Mackey point realization (Theorem \ref{T:mackey}) of this algebra.  The space $\rpf{X}{G}$ is referred to as the \textbf{ergodic components} of $G \actson (X,\nu)$.  Let $\pi : (X,\nu) \to (\rpf{X}{G},\overline{\nu})$ be the quotient map.  Then $(\pi^{-1}(y),D_{\pi}(y))$ is an ergodic $G$-space for each $y \in \rpf{X}{G}$ and the disintegration decomposition $\nu = \int_{\rpf{X}{G}} D_{\pi}(y)~d\overline{\nu}(y)$ is the \textbf{ergodic decomposition}.

We remark that $G$ acts trivially on $\rpf{X}{G}$ (an easy consequence of the construction: if it did not act trivially there would be some bounded Borel function on $\rpf{X}{G}$ that is not $G$-invariant but the algebra of bounded Borel functions on $\rpf{X}{G}$ consist only of invariant functions).  From this, it is easy to see that each component $(\pi^{-1}(y),D_{\pi}(y))$ is an ergodic $G$-space.

We will need the following fact about ergodic decomposition in what follows:
\begin{proposition}\label{P:ergdecomptrivial}
Let $(X,\nu)$ be a $G$-space and $(Y,\eta)$ be a $G$-space where $G$ acts trivially.  Then $\rpf{(X \times Y)}{G} = (\rpf{X}{G}) \times Y$.
\end{proposition}
\begin{proof}
Write $(Z,\zeta)$ for the ergodic components of $(X,\nu)$ and $\pi : (X,\nu) \to (Z,\zeta)$ for the decomposition map.
Let $\varphi : (X \times Y, \nu\times\eta) \to (Z\times Y,\zeta\times\eta)$ be given by $\varphi(x,y) = (\pi(x),y)$.

Let $E \subseteq X \times Y$ be a positive measure $G$-invariant set.  For each $(z,y) \in Z \times Y$, define the set
\[
E_{z,y} = E \cap (\pi^{-1}(z) \times \{ y \}).
\]
Then $E_{z,y}$ is a $D_{\pi}(z) \times \delta_{y}$-measurable set.  Since $G$ acts trivially on $Y$ and $Z$, and $gE = E$ for all $g \in G$, we have that $gE_{z,y} = E_{z,y}$ for all $g \in G$.  Now $D_{\pi}(z)$ is ergodic and $\delta_{y}$ is a point mass, hence $D_{\pi}(z) \times \delta_{y}$ is ergodic.  Therefore for almost every $(z,y) \in Z \times Y$, either $E_{z,y}$ is null or has full measure.

Let
\[
A = \{ (z,y) \in Z \times Y : D_{\pi}(z) \times \delta_{y}(E_{z,y}) = 1 \}
\]
and observe that $D_{\pi}(z) \times \delta_{y}(E) = D_{\pi}(z) \times \delta_{y}(E_{z,y})$ since $D_{\pi}(z) \times \delta_{y}$ is supported on $\pi^{-1}(z) \times \{ y \}$.  Therefore, for $(z,y) \in A$ we have that $D_{\pi}(z) \times \delta_{y}(E \symdiff \varphi^{-1}(A)) = 0$ since both $E$ and $\varphi^{-1}(A)$ have full $D_{\pi}(z)\times\delta_{y}$-measure.  On the other hand, for $(z,y) \notin A$ we also have that $D_{\pi}(z)\times\delta_{y}(E \symdiff \varphi^{-1}(A)) = 0$ since both sets are null.

Therefore
\[
\nu\times\eta(E \symdiff \varphi^{-1}(A))
= \int_{Z \times Y} D_{\pi}(z)\times\delta_{y}(E \symdiff \varphi^{-1}(A))~d\zeta\times\eta(z,y) = 0
\]
meaning that any $G$-invariant positive measure set in $(X\times Y,\nu\times\eta)$ belongs to the algebra of measurable sets of $(Z\times Y,\eta\times\eta)$ as claimed.
\end{proof}

\subsection{The Poisson Boundary}

The Poisson boundary of a group will play a relatively minor role in our work compared to its presence in the work of Bader and Shalom \cite{BS04} and in the work of Nevo and Stuck and Zimmer \cite{SZ94},\cite{nevozimmer}.  The main interest we will have in the Poisson boundary is that it is a contractive action and therefore gives rise to relatively contractive maps.  The reader is referred to \cite{BS04} and \cite{creutzphd} for a detailed account of Poisson boundaries in the abstract setting and to \cite{Fu63}, \cite{Fu67}, \cite{Fu71}, \cite{Ka89} and \cite{Ka92} for information on Poisson boundaries of semisimple groups and lattices in semisimple groups.

\begin{definition}[Furstenberg \cite{Fu63}]
Let $G$ be a locally compact second countable group and $\mu \in P(G)$ a probability measure on $G$.  Consider the map $T : G^{\mathbb{N}} \to G^{\mathbb{N}}$ given by $T(w_{1},w_{2},w_{3},\cdots) = (w_{1}w_{2},w_{3},\cdots)$.  The space of $T$-ergodic components of $(G^{\mathbb{N}},\mu^{\mathbb{N}})$ is the Poisson boundary of $(G,\mu)$.
\end{definition}

\begin{theorem}[Furstenberg \cite{Fu63}]\label{T:PBcont}
Let $G$ be a locally compact second countable group and $\mu \in P(G)$.  The action of $G$ on the Poisson boundary is a $\mu$-stationary contractive action.
\end{theorem}

\begin{theorem}[Zimmer \cite{Zi84}]\label{T:PBamen}
Let $G$ be a locally compact second countable group and $\mu \in P(G)$.  The action of $G$ on the Poisson boundary is amenable.
\end{theorem}

\subsection{Ergodic Decomposition and Poisson Boundaries}

We will need a basic fact about the ergodic decomposition of the product of the Poisson boundary and a measure-preserving space due to Bader and Shalom \cite{BS04}.

\begin{proposition}[Bader-Shalom, \cite{BS04} Corollary 2.18]
Let $(B,\beta)$ be the Poisson boundary for $(G,\mu)$ where $\mu$ is an admissible measure on $G$ and let $(X,\nu)$ be an ergodic measure-preserving $G$-space.  Then $(B \times X, \beta\times\nu)$ is an ergodic $\mu$-stationary $G$-space.
\end{proposition}

The only difficulty in the proof of the above statement is the ergodicity.  We will need the following extension of their result:
\begin{proposition}\label{P:ergdecomperg}
Let $(B,\beta)$ be the Poisson boundary for $(G,\mu)$ where $\mu$ is an admissible measure on $G$, let $(C,\eta)$ be any $G$-quotient of $(B,\beta)$ and let $(X,\nu)$ be a measure-preserving $G$-space.  Then $\rpf{(C \times X)}{G} = \rpf{X}{G}$.
\end{proposition}
\begin{proof}
Write $(Z,\zeta)$ for the ergodic components of $(X,\nu)$ and let $\pi : (X,\nu) \to (Z,\zeta)$ be the decomposition map.  Let $\varphi : (B \times X, \beta\times\nu) \to (Z,\zeta)$ be given by $\varphi(b,x) = \pi(x)$.  Let $\tau : B \times X \to C \times X$ be given by $\tau(b,x) = (\tau_{0}(b),x)$ where $\tau_{0} : B \to C$ is the $G$-map making $(C,\eta)$ a $G$-quotient of $(B,\beta)$.

Let $E \subseteq C \times X$ be a positive measure $G$-invariant set.  For each $z \in Z$, let $E_{z} = \tau^{-1}(E) \cap (B \times \pi^{-1}(z))$.  Since $G$ acts trivially on $Z$, $E_{z}$ is a $G$-invariant set.  Consider $D_{\varphi}(z) = \beta\times D_{\pi}(z)$.  Since $\nu$ is measure-preserving, so is $D_{\pi}(z)$ for each $z$.  By the above Proposition, $D_{\varphi}(z)$ is then ergodic for each $z$.

Therefore, $D_{\varphi}(z)(E_{z})$ is either null or conull for each $z$.  Let $A \subseteq Z$ be the set of $z$ such that $D_{\varphi}(z)(\tau^{-1}(E)) = D_{\varphi}(z)(E_{z}) = 1$.  Then $D_{\varphi}(z)(\tau^{-1}(E) \symdiff \varphi^{-1}(A)) = 0$ for almost every $z$ since either both $\tau^{-1}(E)$ and $\varphi^{-1}(A)$ are full (when $z \in A$) or both null (when $z \notin A$).  Hence
\[
\beta\times\nu(\tau^{-1}(E) \symdiff \varphi^{-1}(A)) = \int_{Z} D_{\varphi}(z)(\tau^{-1}(E) \symdiff \varphi^{-1}(A))~d\zeta(z) = 0.
\]
Therefore $\eta \times \nu(E \symdiff \tau(\varphi^{-1}(A))) = \beta\times\nu(\tau^{-1}(E) \symdiff \varphi^{-1}(A)) = 0$
meaning that every $G$-invariant measurable set in $C \times X$ belongs to the algebra of measurable sets of $\rpf{X}{G}$ as claimed.
\end{proof}

\subsection{The Invariants Product Functor}

We recall now the invariants product functor of Bader and Shalom \cite{BS04}.
Let $G = G_{1} \times G_{2}$ be a product of groups and let $(X,\nu)$ be an ergodic $G$-space.  Write $\rpf{X}{G_{j}}$ for the space of $G_{j}$-ergodic components of $X$, for $j=1,2$.  Then $G_{j}$ acts trivially on $\rpf{X}{G_{j}}$ and $G_{3-j}$ acts ergodically (since $G$ acts ergodically on $X$).  We will write $(X_{1},\nu_{1})$ to be the space of $G_{2}$-ergodic components with the push-forward of $\nu$ and likewise write $(X_{2},\nu_{2})$ for the space of $G_{1}$-ergodic components.

The \textbf{invariants product functor} is the functor $F^{G}$ that assigns $F^{G}(X,\nu) = (X_{1},\nu_{1}) \times (X_{2},\nu_{2})$, which we treat as a $G$-space with the diagonal $G$-action.  That this is indeed a functor is shown in Bader-Shalom in the sense that given a $G$-map $\pi : (X,\nu) \to (Y,\eta)$ of ergodic $G$-spaces, define $F^{G}(\pi) = \pi_{1} \times \pi_{2}$ where $\pi_{j} : (X_{j},\nu_{j}) \to (Y_{j},\eta_{j})$ is the Mackey point realization (Theorem \ref{T:mackey}) of the inclusion at the level of $\sigma$-algebras, and the following diagram commutes:
\begin{diagram}
(X,\nu)		&\rTo^{\pi}		&(Y,\eta)\\
\dTo_{F^{G}}	&			&\dTo_{F^{G}}\\
(X_{1},\nu_{1}) \times (X_{2},\nu_{2})	&\rTo^{\pi_{1} \times \pi_{2}}	&(Y_{1},\eta_{1}) \times (Y_{2},\eta_{2})
\end{diagram}
In general, the mapping $(X,\nu) \to F^{G}(X,\nu)$ need not be a $G$-map (though of course $\pi_{j}$ is a $G_{j}$-map so $F^{G}(\pi)$ is always a product of $G_{1}$- and $G_{2}$-maps).  However, in the case of ergodic stationary $G$-spaces the map is a $G$-map:
\begin{proposition}[Bader-Shalom \cite{BS04} Proposition 1.10]\label{P:1.10}
Let $G = G_{1} \times G_{2}$ be a product of two locally compact second countable groups and let $\mu_{j} \in P(G_{j})$ be admissible probability measures for $j=1,2$.  Set $\mu = \mu_{1} \times \mu_{2}$.  If $(X,\nu)$ is a $\mu$-stationary ergodic $G$-space then $(X,\nu) \to (X_{1},\nu_{1}) \times (X_{2},\nu_{2})$ is a relatively measure-preserving $G$-map.
\end{proposition}

\subsection{Relatively Contractive Maps}

Relatively contractive maps were introduced in \cite{CP12} as a generalization of both the contractive spaces studied by Jaworski \cite{Ja94}, \cite{Ja95} (under the name SAT) and the notion of proximal maps for stationary actions (see e.g.~\cite{FG10}).  In \cite{CP12}, strong uniqueness properties of such maps is proved and we generalize a result in \cite{CP12} regarding joinings of contractive spaces.  This generalization will be the key ingredient in our Intermediate Contractive Factor Theorem.

\begin{definition}[Jaworski \cite{Ja94}]
Let $G$ be a locally compact second countable group and $(X,\nu)$ a $G$-space.  Then $(X,\nu)$ is \textbf{contractive} when for any measurable set $E \subseteq X$ with $\nu(E) > 0$ there exists a sequence $\{g_{n}\}$ in $G$ such that $\nu(g_{n}E) \to 1$.
\end{definition}

\begin{definition}[Creutz-Peterson \cite{CP12} Definition 4.4]
Let $G$ be a locally compact second countable group and $\pi : (X,\nu) \to (Y,\eta)$ a $G$-map of $G$-spaces.  Then $\pi$ is \textbf{relatively contractive} when for any measurable set $E \subseteq X$ and almost every $y \in Y$ such that $D_{\pi}(y)(E) > 0$there exists a sequence $\{ g_{n} \}$ in $G$ such that $g_{n}^{-1}D_{\pi}(g_{n}y)(E) \to 1$.
\end{definition}

\begin{theorem}[Creutz-Peterson \cite{CP12} Theorem 4.15]\label{T:relcontcomp}
Let $G$ be a locally compact second countable group and let $\pi : (X,\nu) \to (Y,\eta)$ and $\psi : (Y,\eta) \to (Z,\zeta)$ be $G$-maps of $G$-spaces.  If $\psi \circ \pi$ is relatively contractive then $\pi$ and $\psi$ are relatively contractive.
\end{theorem}

\begin{theorem}[Creutz-Peterson \cite{CP12} Theorem 4.13]\label{T:relcontB}
Let $G$ be a locally compact second countable group, $(X,\nu)$ a $G$-space and $(B,\beta)$ a contractive $G$-space.  Then the natural projection map $p : (B \times X, \beta \times \nu) \to (X,\nu)$ is relatively contractive.
\end{theorem}

\begin{theorem}\label{T:relmprelcon}
Let $G$ be a locally compact second countable group and $\pi : (X,\nu) \to (Y,\eta)$ a $G$-map of $G$-spaces.  If $\pi$ is both relatively measure-preserving and relatively contractive then it is an isomorphism.
\end{theorem}
\begin{proof}
Let $E$ be a measurable set in $X$.  Since $\pi$ is relatively contractive, for almost every $y \in Y$ such that $D_{\pi}(y)(E) > 0$ there is a sequence $g_{n} \in G$ such that $D_{\pi}(g_{n}y)(g_{n}E) \to 1$.  Since $\pi$ is relatively measure-preserving, $D_{\pi}(g_{n}y)(g_{n}E) = D_{\pi}(y)(E)$.  Therefore $D_{\pi}(y)(E) = 1$ for almost every $y$ such that $D_{\pi}(y)(E) > 0$.  As this holds for all measurable sets $E$ this means $\pi$ is an isomorphism.
\end{proof}

\begin{corollary}\label{C:mpquotrelcon}
Let $G$ be a locally compact second countable group and $(X,\nu)$ a contractive $G$-space.  If $\pi : (X,\nu) \to (Y,\eta)$ is a relatively measure-preserving $G$-map of $G$-spaces then it is an isomorphism.
\end{corollary}
\begin{proof}
This follows from the previous theorem and the observation that any map from a contractive space is relatively contractive (the map from $(X,\nu)$ to the trivial one-point system is relatively contractive and so by Theorem \ref{T:relcontcomp} then so is $\pi$).
\end{proof}

\subsection{Joinings}

Joinings will play a key role in both our contractive factor theorem and in the study of random subgroups.  The reader is referred to \cite{glasner} for more information on joinings.

\begin{definition}
Let $G$ be a locally compact second countable group and let $(X,\nu)$ and $(Y,\eta)$ be $G$-spaces.  A \textbf{joining} of $(X,\nu)$ and $(Y,\eta)$ is a probability measure $\alpha \in P(X \times Y)$ such that $(p_{X})_{*}\alpha = \nu$ and $(p_{Y})_{*}\alpha = \eta$ where $p_{X}$ and $p_{Y}$ are the natural projections from $X \times Y$ to $X$ and $Y$.  The space $(X \times Y,\alpha)$ is then a $G$-space with  the diagonal action.
\end{definition}

\begin{definition}
Let $(X,\nu)$ and $(Y,\eta)$ be $G$-spaces with a common $G$-quotient $(Z,\zeta)$, that is a diagram of $G$-maps and $G$-spaces as follows:
\begin{diagram}
		&			&(X,\nu)\\
		&			&\dTo^{\pi} \\
(Y,\eta)	&\rTo^{\varphi}	&(Z,\zeta)
\end{diagram}
Treat $X \times Y$ as a $G$-space with the diagonal action.
A $G$-quasi-invariant Borel probability measure $\rho \in P(X \times Y)$ is a \textbf{relative joining} of $(X,\nu)$ and $(Y,\eta)$ over $(Z,\zeta)$ when the following diagram of $G$-maps commutes:
\begin{diagram}
(X \times Y,\rho)		&\rTo^{p_{X}}	&(X,\nu)\\
\dTo^{p_{Y}}	&			&\dTo^{\pi} \\
(Y,\eta)		&\rTo^{\varphi}	&(Z,\zeta)
\end{diagram}
where $p_{X}$ and $p_{Y}$ are the natural projections from $X \times Y$ to $X$ and $Y$, respectively.
\end{definition}

In general, the product $\nu \times \eta$ is not a relative joining of $(X,\nu)$ and $(Y,\eta)$ over $(Z,\zeta)$ unless $(Z,\zeta)$ is trivial since we require that $\pi \circ p_{X} = \varphi \circ p_{Y}$ almost everywhere.  However, there is a notion of independent joining in the relative case:
\begin{definition}
Let $(X,\nu)$ and $(Y,\eta)$ be $G$-spaces with common $G$-quotient $(Z,\zeta)$.  Let $\pi : (X,\nu) \to (Z,\zeta)$ and $\varphi : (Y,
\eta) \to (Z,\zeta)$ be the quotient maps.  The probability measure $\rho \in P(X \times Y)$ given by
\[
\rho = \int_{Z} D_{\pi}(z) \times D_{\varphi}(z)~d\zeta(z)
\]
is the \textbf{independent relative joining} of $(X,\nu)$ and $(Y,\eta)$ over $(Z,\zeta)$.
\end{definition}

Of course, the independent relative joining is a relative joining.  We also note that the independent joining $\nu \times \eta$ is the independent relative joining over the trivial system.

\begin{proposition}\label{P:irjfactor}
Let $\pi : (X,\nu) \to (Y,\eta)$ be a $G$-map of $G$-spaces.  Then the independent relative joining of $(X,\nu)$ and $(Y,\eta)$ over $(Y,\eta)$ is $G$-isomorphic to $(X,\nu)$.
\end{proposition}
\begin{proof}
The independent relative joining is $(X\times Y,\alpha)$ where
\[
\alpha = \int_{Y} D_{\pi}(y) \times \delta_{y}~d\eta(y).
\]
Let $p : X \times Y \to X$ be the projection to $X$.  Let $\alpha_{x} \in P(X \times Y)$ by $\alpha_{x} = \delta_{x} \times \delta_{\pi(x)}$.  Then
\begin{align*}
\int_{X} \alpha_{x}~d\nu(x) &= \int_{Y} \int_{X} \delta_{x} \times \delta_{\pi(x)}~dD_{\pi}(y)(x)~d\eta(y) \\
&= \int_{Y} \int_{X} \delta_{x} \times \delta_{y}~dD_{\pi}(y)(x)~d\eta(y)\\
&= \int_{Y} D_{\pi}(y) \times \delta_{y}~d\eta(y) = \alpha
\end{align*}
and $\alpha_{x}$ is supported on $p^{-1}(x) = \{ x \} \times Y$.  Therefore $D_{p}(x) = \alpha_{x}$ by uniqueness of disintegration.
Since $\alpha_{x}$ is a point mass, then $p$ is an isomorphism so $(X \times Y, \alpha)$ is isomorphic to $(X,\nu)$.
\end{proof}

\subsection{Resolutions}

The notion of resolution, due to de Cornulier \cite{yves}, is intimately connected with notion of relative property $(T)$.  We will make use of resolutions in the easy case when considering a product of two groups, one of which has property $(T)$, to show that weakly amenable actions are in fact essentially transitive in many cases.  The reader is referred to \cite{yves} for a systematic description and proofs.

\begin{definition}[de Cornulier \cite{yves}]
Let $G$ and $Q$ be locally compact second countable groups and let $p : G \to Q$ be a homomorphism with dense image.  Let $f : G \to X$ be any map to a topological space.  Then $f$ \textbf{factors through} $p$ when for every net $\{ g_{i} \}$ in $G$, if $p(g_{i})$ converges in $Q$ then $f(g_{i})$ converges in $X$.  Given an action $G \actson X$ on a topological space, the \textbf{$(Q,f)$-points} of $X$ are $X^{Q} = \{ x \in X : g \mapsto gx \text{ factors through $f$} \}$.
\end{definition}

\begin{proposition}[de Cornulier \cite{yves}]\label{P:closedQpoints}
Let $f : G \to Q$ be a homomorphism of locally compact second countable groups with dense image and let $G \actson X$ be any action on a topological space.  Then the space of $(Q,f)$-points $X^{Q}$ is a closed $G$-invariant set in $X$.
\end{proposition}

\begin{definition}
Let $f : G \to Q$ be a homomorphism of locally compact second countable groups with dense image and let $\pi : G \to \mathcal{U}(\mathcal{H})$ be a (strongly continuous) unitary representation of $G$ on a Hilbert space.  Let $\mathcal{H}^{Q}$ be the space of $(Q,f)$-points in $\mathcal{H}$ and let $\pi^{Q} : Q \to \mathcal{H}^{Q}$ be the restriction of $\pi$ to $Q$ on $\mathcal{H}^{Q}$.
\end{definition}

\begin{definition}
Let $G$ be a locally compact second countable group and $\pi : G \to \mathcal{U}(\mathcal{H})$ be a unitary representation of $G$ on a Hilbert space.  Then $\pi$ has \textbf{almost invariant vectors} when there exists a sequence $\{ v_{n} \}$ in $\mathcal{H}$ such that $\| v_{n} \| = 1$ for all $n$ and such that for each fixed $g \in G$ it holds that $\lim_{n} \| \pi(g)v_{n} - v_{n} \| \to 0$.
\end{definition}

\begin{definition}
Let $f : G \to Q$ be a homomorphism of locally compact second countable groups with dense image.  Then $f$ is a \textbf{resolution} when for every unitary representation $\pi : G \to \mathcal{U}(\mathcal{H})$ of $G$ on a Hilbert space that has almost invariant vectors, the representation $\pi^{Q} : Q \to \mathcal{U}(\mathcal{H}^{Q})$ also has almost invariant vectors.
\end{definition}

\begin{proposition}[de Cornulier \cite{yves}]\label{P:resolutions}
Let $G = G_{1} \times G_{2}$ be a product of two locally compact second countable groups.  If $G_{2}$ has property $(T)$ then the projection map $\mathrm{proj}_{1} : G \to G_{1}$ is a resolution.
\end{proposition}

\begin{proposition}[de Cornulier \cite{yves}]
Let $G$ and $Q$ be locally compact second countable groups and let $p : G \to Q$ be a resolution.  Let $\Gamma < G$ be a lattice.  Then $p : \Gamma \to \overline{p(\Gamma)}$ is a resolution.
\end{proposition}

Combining the previous two propositions:
\begin{proposition}\label{P:reslatts}
Let $G$ and $H$ be locally compact second countable groups such that $H$ has property $(T)$ and let $\Gamma < G \times H$ be an irreducible lattice.  Then the projection map $\mathrm{proj}_{G} : \Gamma \to G$ is a resolution.
\end{proposition}

\section{Random Subgroups}

Invariant random subgroups are an active area of research and are the natural setting for the study of stabilizers of actions of groups.  We present here a systematic approach to treating random subgroups as subgroups of one another and how this interacts with the possible stabilizers of actions of the group.

\begin{definition}
Let $G$ be a locally compact second countable group.  Denote by $S(G)$ the space of closed subgroups of $G$ endowed with the Chabauty topology.  Let $G$ act on $S(G)$ by conjugation.  A Borel probability measure $\eta \in P(S(G))$ that is invariant under the conjugation action is an \textbf{invariant random subgroup}.
\end{definition}

We generalize the notion of invariant random subgroup to quasi-invariant actions:
\begin{definition}
Let $G$ be a locally compact second countable group and denote by $S(G)$ the space of closed subgroups of $G$ endowed with the Chabauty topology and the action of $G$ by conjugation.  A Borel probability measure $\eta \in P(S(G))$ is a \textbf{random subgroup} (or more precisely, a \textbf{quasi-invariant random subgroup}) when it is quasi-invariant under the conjugation action.
\end{definition}

The following generalizes the equivalent statement for measure-preserving actions and invariant random subgroups due to Abert-Glasner-Vir\'{a}g \cite{AGV12}:
\begin{theorem}
Let $G \actson (X,\nu)$ be a quasi-invariant action of a locally compact second countable group.  Then the map $\mathrm{stab} : X \to S(G)$ by $\mathrm{stab}(x) = \{ g \in G : gx = x \}$ gives rise to a random subgroup $\mathrm{stab}_{*}\nu$.  This will be an invariant random subgroup precisely when the action is measure-preserving.

Conversely, given a random subgroup $\eta \in P(S(G))$ there exists a quasi-invariant action $G \actson (X,\nu)$ such that $\mathrm{stab}_{*}\nu = \eta$.  Moreover, this action will be a measure-preserving extension of $(S(G),\eta)$.
\end{theorem}

The previous theorem is actually a special case of Theorem \ref{T:extensionsrg} and will be proved as Corollary \ref{C:easy} below.

\subsection{Subgroups of Random Subgroups}

Subgroups of invariant random subgroups were introduced in \cite{CP12}.  We generalize this idea to quasi-invariant random subgroups.

\begin{definition}
Let $\rho, \zeta \in P(S(G))$ be random subgroups of a locally compact second countable group $G$.  Then $\rho$ is a \textbf{subgroup} of a $\zeta$ when there exists a joining $\alpha \in P(S(G) \times S(G))$ of $\rho$ and $\zeta$ such that for $\alpha$-almost every $(H,L) \in S(G) \times S(G)$, it holds that $H$ is a subgroup of $L$.  This will be written $\rho < \zeta$.
\end{definition}

\begin{proposition}
The property of being a subgroup is a transitive relation on random subgroups.
\end{proposition}
\begin{proof}
Let $\alpha, \beta, \rho \in P(S(G))$ be random subgroups of a group $G$ such that $\alpha < \beta$ and $\beta < \rho$.  Let $\psi \in P(S(G) \times S(G))$ be a joining of $\alpha$ and $\beta$ such that $H < L$ for $\psi$-almost every $(H,L)$ and let $\varphi \in P(S(G) \times S(G))$ be a joining of $\beta$ and $\rho$ such that $L < K$ for $\varphi$-almost every $(L,K)$.  Let $p_{A} : S(G) \times S(G) \to S(G)$ be the projection to the first coordinate and $p_{B} : S(G) \times S(G) \to S(G)$ the projection to the second.

Observe that $D_{p_{A}}(L) = \delta_{L} \times \varphi_{L}$ for $\varphi_{L} \in P(S(G))$ such that $\int_{S(G)} \delta_{L} \times \varphi_{L}~d\beta(L) = \varphi$.  Likewise, $D_{p_{B}}(L) = \psi_{L} \times \delta_{L}$ for $\psi_{L} \in P(S(G))$ such that $\int_{S(G)} \psi_{L} \times \delta_{L}~d\beta(L) = \psi$.

Define $\tau \in P(S(G) \times S(G) \times S(G))$ by
\[
\tau = \int_{S(G)} \psi_{L} \times \delta_{L} \times \varphi_{L}~d\beta(L).
\]
Then, letting $p_{j} : S(G) \times S(G) \times S(G) \to S(G)$ be the projections,
\[
(p_{1})_{*}\tau = \int_{S(G)} \psi_{L}~d\beta(L) = (p_{A})_{*} \int_{S(G)} \psi_{L} \times \delta_{L}~d\beta(L) = (p_{A})_{*}\psi = \alpha
\]
and likewise
\[
(p_{2})_{*}\tau = \beta \quad\quad\text{and}\quad\quad (p_{3})_{*}\tau = \rho.
\]
Therefore $\tau$ is a joining of $\alpha$ and $\beta$ and $\rho$.

Note that $(p_{1} \times p_{2})_{*}\tau = \psi$ and that $(p_{2} \times p_{3})_{*}\tau = \varphi$.
For $\tau$-almost every $(H,L,K)$ we then have that $H < L$ and $L < K$.  Hence $H < K$ for $(p_{1} \times p_{3})_{*}\tau$-almost every $(H,K)$.  As $(p_{1} \times p_{3})_{*}\tau$ is a joining of $\alpha$ and $\rho$, this shows that $\alpha < \rho$.
\end{proof}

\begin{definition}
Let $\rho, \zeta \in P(S(G))$ be random subgroups of a locally compact second countable group $G$.  Then $\rho$ is a \textbf{normal subgroup} of a $\zeta$ when there exists a joining $\alpha \in P(S(G) \times S(G))$ of $\rho$ and $\zeta$ such that for $\alpha$-almost every $(H,L) \in S(G) \times S(G)$, it holds that $H$ is a normal subgroup of $L$.  This will be written $\rho \normal \zeta$.
\end{definition}

\begin{definition}
A random subgroup $\rho \in P(S(G))$ of a locally compact second countable group is \textbf{simple} when the only normal subgroups of it are trivial: if $\eta \normal \rho$ then for any joining $\alpha$ witnessing that $\eta < \rho$, for $\alpha$-almost every $(H,L) \in S(G) \times S(G)$, either $H = e$ or $H = L$. 
\end{definition}

Note that if $\rho$ is a simple ergodic random subgroup and $\eta \normal \rho$ is also ergodic then $\eta = \rho$ or $\eta = \delta_{e}$.

The main reason for introducing the notion of subgroups of random subgroups is the following relativization of the fact that stabilizers of quasi-invariant actions give rise to random subgroups:
\begin{theorem}
Let $G$ be a locally compact second countable group and let $\pi : (X,\nu) \to (Y,\eta)$ be a $G$-map of $G$-spaces.  Then $\stab_{*}\nu$ is a subgroup of $\stab_{*}\eta$.
\end{theorem}
\begin{proof}
Define $\alpha \in P(S(G) \times S(G))$ by
\[
\alpha = \int_{X} \delta_{\stab(x)} \times \delta_{\stab(\pi(x))}~d\nu(x).
\]
Then the projection to the first coordinate $\mathrm{pr}_{1} : S(G) \times S(G) \to S(G)$ has the property that
\[
(\mathrm{pr}_{1})_{*}\alpha = \int_{X} \delta_{\stab(x)}~d\nu(x) = \stab_{*}\nu
\]
and the projection to the second coordinate has the property that
\[
(\mathrm{pr}_{2})_{*}\alpha = \int_{X} \delta_{\stab(\pi(x))}~d\nu(x) = \stab_{*}\pi_{*}\nu = \stab_{*}\eta.
\]
Therefore $\alpha$ is a joining of $\stab_{*}\nu$ and $\stab_{*}\eta$.  Now $\stab(x) < \stab(\pi(x))$ for all $x \in X$ since $\pi$ is a $G$-map and therefore for $\alpha$-almost every $(H,L) \in S(G) \times S(G)$ it holds that $H < L$.
\end{proof}

\begin{theorem}\label{T:crazy}
Let $G$ be a locally compact second countable group and $\pi : (X,\nu) \to (Y,\eta)$ a $G$-map of $G$-spaces such that $\stab(x)$ is constant on each fiber: for $\eta$-almost every $y \in Y$, it holds that $\stab(x)$ is constant for $D_{\pi}(y)$-almost every $x \in \pi^{-1}(y)$.  Then $\stab_{*}\nu$ is a normal random subgroup of $\stab_{*}\eta$.
\end{theorem}
\begin{proof}
Since $\stab$ is constant on fibers, it descends to a measurable map $s : Y \to S(G)$ such that $\stab(x) = s(\pi(x))$ for almost every $x \in X$.  For such an $x \in X$ and for $g \in \stab(\pi(x))$,
\[
g\stab(x)g^{-1} = \stab(gx) = s(\pi(gx)) = s(g\pi(x)) = s(\pi(x)) = \stab(x)
\]
meaning that $\stab(x) \normal \stab(\pi(x))$.  The joining $\alpha \in P(S(G) \times S(G))$ given by $\alpha = (\stab \circ (\mathrm{id} \times \pi))_{*}\nu$ where $\stab \circ (\mathrm{id} \times \pi) : X \to X \times Y$ by $(\stab \circ (\mathrm{id}\times\pi))(x) = (\stab(x),\stab(\pi(x)))$ then shows that $\stab_{*}\nu$ is a normal subgroup of $\stab_{*}\eta$.
\end{proof}

\subsection{Subgroups of Random Subgroups Correspond to Quotient Maps}\label{s:subrsquot}

\begin{theorem}\label{T:extensionsrg}
Let $G$ be a locally compact second countable group and $(X,\nu)$ a $G$-space.  Let $\varphi : X \to P(S(G))$ be a $G$-equivariant map such that for $\nu$-almost every $x \in X$ and $\varphi(x)$-almost every $H \in S(G)$ it holds that $H < \stab(x)$.  Then there exists a $G$-space $(Y,\eta)$ and a $G$-map $\pi : (Y,\eta) \to (X,\nu)$ such that $\stab_{*}\eta$ is the barycenter of $\varphi_{*}\nu$.
Moreover, $\pi$ is relatively measure-preserving.
\end{theorem}
\begin{proof}
Observe that for almost every $x \in X$, $\varphi(x)$ is an invariant random subgroup of $\stab(x)$ since $H < \stab(x)$ for $\varphi(x)$-almost every $H$ and since the $G$-equivariance of $\varphi$ gives that for $g \in \stab(x)$, $g \cdot \varphi(x) = g \varphi(x) g^{-1} = \varphi(gx) = \varphi(x)$.

Fix a probability measure $\rho \in P(G)$ in the class of the Haar measure.
For each $x \in X$ and $H \in S(\stab(x))$, let $(Q_{x,H},\rho_{x,H})$ be the Gaussian probability space corresponding to an infinite direct sum of $L^{2}(G / H)$ where $\rho_{x,H}$ is the pushforward of $\rho$ under the quotient map $q_{H} : G \to G/H$.  Let $Q = ((Q_{x,H},\rho_{x,H}))_{x \in X, H \in S(\stab(x))}$ be the field of probability spaces just constructed (this is a measurable field following the same reasoning as in \cite{CP12} Theorem 3.3).

Define the cocycle $\alpha : G \times X \times S(G) \to Q$ such that $\alpha(g,x,H) \in \mathrm{Aut}(Q_{x,H}, Q_{gx, gHg^{-1}})$ is the induced automorphism from the operator $T_{g,x,H}$ from the infinte direct sum of $L^{2}(Q_{x,H},\rho_{x,H})$ to the infinite direct sum of $L^{2}(Q_{gx,gHg^{-1}},\rho_{gx,gHg^{-1}})$ given by
\[
(T_{g,x,H}f)(kgHg^{-1}) = f(kgH)\sqrt{\frac{d(q_{H})_{*}(\rho g^{-1})}{d\rho_{x,H}}(kgH)}.
\]
Define the probability space $(Q,\rho)$ by
\[
\big{(}Q = \bigsqcup_{x} \bigsqcup_{H} Q_{x,H}, \rho = \int_{X} \int_{S(G)} \rho_{x,H}~d\varphi(x)(H)~d\nu(x)\big{)}
\]
equipped with the $G$-action coming from the cocycle $\alpha$.  The cocycle identity holds almost everywhere so by Mackey's point realization \cite{Ma62}, as $G$ is locally compact and second countable, after removing a null set we may assume the cocycle identity holds everywhere.

Note that $g \cdot \rho_{x,H} = \rho_{gx, gHg^{-1}}$ and therefore, using the equivariance of $\varphi$,
\begin{align*}
g \cdot \rho &= \int_{X} \int_{S(G)} \rho_{gx,gHg^{-1}}~d\varphi(x)(H)~d\nu(x)
= \int_{X} \int_{S(G)} \rho_{gx,H}~dg\varphi(x)(H)~d\nu(x) \\
&= \int_{X} \int_{S(G)} \rho_{gx,H}~d\varphi(gx)(H)~d\nu(x)
= \int_{X} \int_{S(G)} \rho_{x,H}~d\varphi(x)(H)~dg\nu(x)
\end{align*}
meaning that $\rho$ is quasi-invariant under the $G$-action since $\nu$ is.

For $x \in X$ and $H \in S(\stab(x))$, the map $g \mapsto \alpha(g,x,H)$ defines an action of $N_{\stab(x)}(H) / H$ on $Q_{x,H}$ which is essentially free (Proposition 1.2 in \cite{AEG}).  Now for $q \in Q_{x,H}$ and $g \in G$ we have that $g \cdot (x,H,q) = (gx, gHg^{-1}, \alpha(g,x,H)q)$ meaning that $g \cdot (x,H,q) = (x,H,q)$ if and only if $gx = x$ and $gHg^{-1} = H$ and $\alpha(g,x,H)q = q$ so if and only if $g \in \stab(x)$ and $g \in N_{\stab(x)}(H)$ and $\alpha(g,x,H)q = q$ hence if and only if $g \in H$.
Therefore
\[
\stab_{*}\rho = \int_{X} \int_{H} \stab_{*}\rho_{x,H}~d\varphi(x)(H)~d\nu(x) = \int_{X} \int_{H} \delta_{H}~d\varphi(x)(H)~d\nu(x) = \int_{X} \varphi(x)~d\nu(x)
\]
as required.

Define $\pi : (Q,\rho) \to (X,\nu)$ by $\pi(x,H,q) = x$.  Then
\[
\pi(g \cdot (x,H,q)) = \pi(gx, gHg^{-1}, \alpha(g,x,H)q) = gx
\]
so $\pi$ is a $G$-map.

To see that $\pi$ is relatively measure-preserving, observe that for $f \in L^{\infty}(Q,\rho)$,
\begin{align*}
\int_{Q} f(x,H,q) &\frac{dg\nu}{d\nu}(x)~d\rho(x,H,q) \\
&= \int_{X} \int_{S(G)} \int_{Q_{x,H}} f(x,H,q) \frac{dg\nu}{d\nu}(x)~d\rho_{x,H}(q)~d\varphi(x)(H)~d\nu(x) \\
&= \int_{X} \int_{S(G)} \int_{Q_{x,H}} f(x,H,q)~d\rho_{x,H}(q)~d\varphi(x)(H)~dg\nu(x) \\
&= \int_{Q} f(x,H,q)~d(g \cdot \rho)(x,H,q)
\end{align*}
and therefore
$\frac{dg\cdot\rho}{d\rho}(x,H) = \frac{dg\nu}{d\nu}(x)$ meaning that $\pi$ is relatively measure-preserving (Theorem \ref{T:relmpRN}).
\end{proof}

\begin{theorem}\label{T:extensionsrgerg}
Let $G$ be a locally compact second countable group and $(X,\nu)$ an ergodic $G$-space.  Let $\varphi : X \to P(S(G))$ be a $G$-equivariant map such that for $\nu$-almost every $x \in X$ and $\varphi(x)$-almost every $H \in S(G)$ it holds that $H < \stab(x)$.  Then there exists an ergodic $G$-space $(Y,\eta)$ and a $G$-map $\pi : (Y,\eta) \to (X,\nu)$ such that $\stab_{*}\eta$ is the barycenter of $\varphi_{*}\nu$.
\end{theorem}
\begin{proof}
Let $(Q,\rho)$ be the construction from Theorem \ref{T:extensionsrg} such that there exists a $G$-map $\tau : (Q,\rho) \to (X,\nu)$ with the barycenter of $\varphi_{*}\nu$ being $\stab_{*}\rho$.  Consider the ergodic decomposition $\psi : (Q,\rho) \to (R,\kappa)$.  Then $G$ acts trivially on $(R,\kappa)$ and almost every fiber $(\psi^{-1}(r),D_{\psi}(r))$ is an ergodic $G$-space.  Observe that
\[
\int_{R} \stab_{*}D_{\psi}(r)~d\kappa(r) = \stab_{*} \int_{R} D_{\psi}(r)~d\kappa(r) = \stab_{*}\rho = \mathrm{bar}~\varphi_{*}\nu.
\]
Since $\nu$ is ergodic, so is $\varphi_{*}\nu$ (treating $(P(S(G)),\varphi_{*}\nu)$ as $G$-space).  Therefore $(S(G),\mathrm{bar}~\varphi_{*}\nu)$ is also an ergodic $G$-space.  Now ergodic random subgroups are extremal, and therefore, since $\int_{R} \stab_{*}D_{\psi}(r)~d\kappa(r)$ is a convex combination of random subgroups, $D_{\psi}(r)$ must be constant and equal to $\mathrm{bar}~\varphi_{*}\nu$ for almost every $r \in R$.  Therefore for almost every fiber, the map $\pi : (\psi^{-1}(r),D_{\psi}(r)) \to (\pi(\psi^{-1}(x)),\pi_{*}D_{\psi}(r))$ has the required properties.  Observe that since $(X,\nu)$ is ergodic, $\pi_{*}D_{\psi}(r) = \nu$ almost everywhere and therefore $\pi : (\psi^{-1}(r),D_{\psi}(r)) \to (X,\nu)$ has the required properties.
\end{proof}

\begin{corollary}\label{C:easy}
Let $\rho \in P(S(G))$ be a random subgroup of a locally compact second countable group $G$.  Then there exists a $G$-space $(X,\nu)$ such that $\stab_{*}\nu = \rho$.  Moreover, if $\rho$ is an invariant random subgroup then $(X,\nu)$ is measure-preserving.
\end{corollary}
\begin{proof}
Apply Theorem \ref{T:extensionsrg} to the trivial one-point space and the map $\varphi : 0 \to P(S(G))$ given by $\varphi(0) = \rho$.
\end{proof}

\subsection{Free Extensions}

\begin{definition}
Let $G$ be a locally compact second countable group, $(X,\nu)$ a $G$-space and $\varphi: X \to P(S(G))$ a $G$-equivariant map such that $\varphi(x)$ is an invariant random subgroup of $\stab(x)$ almost everywhere.  The construction from Theorem \ref{T:extensionsrg} is the \textbf{free extension of $(X,\nu)$ by $\varphi$}.
\end{definition}

\begin{theorem}\label{T:freeextcomm}
Let $G$ be a locally compact second countable group and let $\pi : (Y,\eta) \to (X,\nu)$ be a $G$-map of $G$-spaces.  Let $\varphi : X \to P(S(G))$ be a $G$-equivariant measurable map such that $\varphi(x) \in P(\mathrm{stab}(x))$ almost everywhere.  Let $(Z,\zeta)$ be the free extension of $(X,\nu)$ over $\varphi$ and let $(W,\rho)$ be the free extension of $(Y,\eta)$ over $\varphi \circ \pi$.  Then there exists a $G$-map $\tau :  (W,\rho) \to (Z,\zeta)$ such that the resulting diagram commutes.
\end{theorem}
\begin{proof}
For $y \in Y$ and $H \in S(\stab(y))$, let $W_{y,H}$ be the fiber in $W$ over $(y,H)$.  For $x \in X$ and $H \in S(\stab(x))$, let $Z_{x,H}$ be the fiber in $Z$ over $(x,H)$.  Define the map $\tau_{y,H} : W_{y,H} \to Z_{\pi(y),H}$ by $\tau(y,H)(y,H,q) = (\pi(y),H,q)$ which is well-defined since $H < \stab(y) < \stab(\pi(y))$.  Observe that the cocycle $\alpha$ defining the actions on $W$ and $Z$ are identical on each fiber.  Define the map $\tau : W \to Z$ by $\tau(y,H,q) = \tau_{y,H}(y,H,q) = (\pi(y),H,q)$.  Then $\tau_{*}\zeta = \rho$ since $\pi_{*}\eta = \nu$.  Also, $\tau(g \cdot (y,H,q)) = \tau(gy, gHg^{-1}, \alpha(g,y,H)q) = (\pi(gy),gHg^{-1},\alpha(g,\pi(y),H)q) = g \cdot \tau(y,H,q)$.  We also see that, letting $p_{X} : Z \to X$ and $p_{Y} : W \to Y$ be the free extension quotient maps,
\[
p_{X}(\tau(y,H,q)) = p_{X}(\pi(y),H,q) = \pi(y)
\]
and that
\[
\pi(p_{Y}(y,H,q)) = \pi(y)
\]
meaning that the diagram commutes.
\end{proof}

\subsection{Invariant Random Subgroups and Quotient Maps}

\begin{theorem}\label{T:moreuseful}
Let $G$ be a locally compact second countable group and $(X,\nu)$ a measure-preserving $G$-space.  Let $\rho \in P(S(G))$ be an invariant random subgroup such that $\rho$ is a subgroup of $\stab_{*}\nu$.  Then there exists a measure-preserving $G$-space $(Y,\eta)$ and a $G$-map $\pi : (Y,\eta) \to (X,\nu)$ such that $\stab_{*}\eta = \rho$.
\end{theorem}
\begin{proof}
Let $\alpha \in P(S(G) \times S(G))$ be a joining of $\rho$ and $\stab_{*}\nu$ such that for $\alpha$-almost every $(H,L)$ it holds that $H < L$.  Let $p : (S(G) \times S(G),\alpha) \to (S(G),\stab_{*}\nu)$ be the projection to the second coordinate.
Define the map $\varphi : X \to P(S(G))$ by $\varphi(x) = D_{p}(\stab(x))$, the disintegration of $\alpha$ over $\stab_{*}\nu$ at $\stab(x)$.  Then $\varphi(gx) = D_{p}(\stab(gx)) = D_{p}(g\stab(x)g^{-1}) = g D_{p}(\stab(x)) g^{-1}$ since $D_{p}$ is $G$-equivariant because $p$ is relatively measure-preserving (since $\rho$ and $\stab_{*}\nu$ are invariant random subgroups).

Theorem \ref{T:extensionsrg} then yields a measure-preserving $G$-space $(Y,\eta)$ (a measure-preserving extension of a measure-preserving action is measure-preserving) and a $G$-map $(Y,\eta) \to (X,\nu)$.
\end{proof}

\begin{corollary}
Let $G$ be a locally compact second countable group and let $\rho,\zeta \in P(S(G))$ be invariant random subgroups of $G$ such that $\rho$ is a subgroup of $\zeta$.  Then there exists a $G$-map of measure-preserving $G$-spaces $\pi : (X,\nu) \to (Y,\eta)$ such that $\stab_{*}\nu = \zeta$ and $\stab_{*}\eta = \rho$.
\end{corollary}
\begin{proof}
Let $(Y,\eta)$ be the space corresponding to $\zeta$ from Corollary \ref{C:easy} and let $(X,\nu)$ be the construction from Theorem \ref{T:moreuseful}.
\end{proof}

\subsection{Free Extensions by Random Subgroups}

\begin{definition}
Let $G$ be a locally compact second countable group, $(X,\nu)$ a measure-preserving $G$-space and $\rho \in P(S(G))$ an invariant random subgroup of $G$ such that $\rho$ is a subgroup of $\stab_{*}\nu$.  The construction from Theorem \ref{T:moreuseful} is the \textbf{free extension of $(X,\nu)$ by $\rho$}.
\end{definition}

\begin{definition}
Let $G$ be a locally compact second countable group and $\rho \in P(S(G))$ be an invariant random subgroup of $G$.  The \textbf{$\rho$-nonfree action} of $G$ is the free extension of the trivial one-point action by $\rho$.  The (an) \textbf{ergodic $\rho$-nonfree action} of $G$ is any of the ergodic components of the $\rho$-nonfree action.
\end{definition}

\section{Quotienting Out By Random Subgroups}\label{s:quotrs}

We now introduce a generalization of the ergodic decomposition by normal subgroups.  Recall that given a locally compact second countable group $G$ and a closed normal subgroup $N \normal G$, for any $G$-space $(X,\nu)$ one defines the ergodic decomposition of $(X,\nu)$ by $N$ to be $(\rpf{X}{N}, \overline{\nu})$ as the Mackey point realization (Theorem \ref{T:mackey}) of the algebra of $N$-invariant functions in $L^{\infty}(X,\nu)$.  Our goal here is to define a similar decomposition over random subgroups.

\subsection{The Quotient Space By a Random Subgroup}

We now explore the class of random subgroups that live below a $G$-space $(X,\nu)$ and introduce the method of quotienting out by such random subgroups.

\begin{proposition}
Let $G \actson (X,\nu)$ be a quasi-invariant action of a locally compact second countable group and let $\varphi : X \to S(G)$ be a $G$-equivariant measurable map such that $\stab(x) \subseteq \varphi(x)$ for $\nu$-almost every $x \in X$.  Then $\varphi_{*}\nu$ is a random subgroup of $G$ such that $\stab_{*}\nu$ is a subgroup of $\varphi_{*}\nu$.
\end{proposition}
\begin{proof}
That $\varphi_{*}\nu$ is a random subgroup is immediate from the $G$-equivariance of $\varphi$.  The joining $\int \delta_{\stab(x)} \times \delta_{\varphi(x)}~d\nu(x)$ shows that $\stab_{*}\nu$ is a subgroup of $\varphi_{*}\nu$.
\end{proof}

\begin{definition}
Let $G \actson (X,\nu)$ be a quasi-invariant action of a locally compact second countable group.  A random subgroup $\eta$ of $G$ is a \textbf{random subgroup of $G$ below $(X,\nu)$} when there exists a measurable $G$-equivariant map $\varphi : X \to S(G)$ such that $\stab(x) \subseteq \varphi(x)$ almost everywhere and $\varphi_{*}\nu = \eta$.
\end{definition}

Having established the class of random subgroups that live below a $G$-space, we now generalize the ergodic decomposition to such random subgroups.

\begin{definition}
Let $G \actson (X,\nu)$ be a quasi-invariant action of a locally compact second countable group and let $\varphi : X \to S(G)$ be a $G$-equivariant measurable map such that $\stab(x) \subseteq \varphi(x)$ for $\nu$-almost every $x \in X$.  Let $f \in L^{\infty}(X,\nu)$.  Then $f$ is \textbf{$\varphi$-invariant} when for almost every $x \in X$ and all $g \in \varphi(x)$, it holds that $f(gx) = f(x)$.  The space of $\varphi$-invariant functions will be written $L^{\infty}(X,\nu)^{\varphi}$.
\end{definition}

\begin{proposition}
Let $G \actson (X,\nu)$ be a quasi-invariant action of a locally compact second countable group and let $\varphi : X \to S(G)$ be a $G$-equivariant measurable map such that $\stab(x) \subseteq \varphi(x)$ for $\nu$-almost every $x \in X$.  The space of $\varphi$-invariant functions $L^{\infty}(X,\nu)^{\varphi}$ is a closed $G$-invariant subalgebra of $L^{\infty}(X,\nu)$.
\end{proposition}
\begin{proof}
That it is a subalgebra is clear.  Let $f_{n} \in L^{\infty}(X,\nu)^{\varphi}$ such that $f_{n} \to f \in L^{\infty}(X,\nu)$.  For almost every $x \in X$ and any $g \in \varphi(x)$ then $f_{n}(gx) = f_{n}(x) \to f(x)$ and $f_{n}(gx) \to f(gx)$ so $f$ is also $\varphi$-invariant.  Given $f \in L^{\infty}(X,\nu)^{\varphi}$ and $g \in G$ define $q(x) = f(gx)$.  For $h \in \varphi(x)$, observe that $ghg^{-1} \in g\varphi(x)g^{-1} = \varphi(gx)$ so
\[
q(hx) = f(ghx) = f((ghg^{-1})gx) = f(gx) = q(x)
\]
meaning that $L^{\infty}(X,\nu)^{\varphi}$ is $G$-invariant.
\end{proof}

\begin{definition}
Let $G \actson (X,\nu)$ be a quasi-invariant action of a locally compact second countable group and let $\varphi : X \to S(G)$ be a $G$-equivariant measurable map such that $\stab(x) \subseteq \varphi(x)$ for $\nu$-almost every $x \in X$.  The \textbf{quotient space of $(X,\nu)$ over $\varphi$} is the Mackey point realization (Theorem \ref{T:mackey}) of the $G$-algebra of $\varphi$-invariant functions.
\end{definition}

\begin{definition}
Let $G \actson (X,\nu)$ be a quasi-invariant action of a locally compact second countable group and let $\eta$ be a random subgroup of $G$ below $(X,\nu)$.  The \textbf{quotient space of $(X,\nu)$ by $\eta$} is the quotient space of $(X,\nu)$ over the map $\varphi$ witnessing that $\eta$ is below $(X,\nu)$.
\end{definition}

\subsection{Examples of Quotienting By Random Subgroups}

\begin{proposition}
Let $G \actson (X,\nu)$ be a quasi-invariant action of a locally compact second countable group and let $N \normal G$ be a closed normal subgroup.  The quotient space of $(X,\nu)$ over $\varphi(x) = \overline{N \cdot \stab(x)}$ is the ergodic decomposition of $(X,\nu)$ over $N$.
\end{proposition}
\begin{proof}
Let $f \in L^{\infty}(X,\nu)$ be a $\varphi$-invariant function.  Then for all $g \in N$ and all $x \in X$ we have that $g \in \varphi(x)$ hence $f(gx) = f(x)$ almost everywhere so $f$ is $N$-invariant.  Now let $f \in L^{\infty}(X,\nu)$ be an $N$-invariant function.  Then for almost every $x \in X$ and all $g \in N$ we have that $f(gx) = f(x)$.  For $g \in stab(x)$ of course $f(gx) = f(x)$ so $f$ is in fact $\varphi$-invariant.  Hence the space of $\varphi$-invariant functions agrees with the space of $N$-invariant functions.
\end{proof}

\begin{proposition}
Let $G \actson (X,\nu)$ be a quasi-invariant action of a locally compact second countable group and let $H < G$ be a closed subgroup such that $\stab(x) \subseteq H$ for almost every $x \in X$.  The quotient space of $(X,\nu)$ over $\varphi(x) = H$ is the ergodic decomposition of $(X,\nu)$ over the normal closure of $H$.
\end{proposition}
\begin{proof}
Let $f \in L^{\infty}(X,\nu)$ be $\varphi$-invariant.  Then $f(hx) = f(x)$ for all $h \in H$ and almost every $x \in X$.  Therefore $H$ acts trivially on the quotient space.  As the kernel of that action must be normal, it contains the normal closure of $H$.  Conversely, any function invariant under the normal closure of $H$ is invariant under $H$.
\end{proof}

\subsection{The Universal Property of the Quotient Space}

We now state and prove several results that jointly amount to a universal property for the quotient space by a random subgroup.

\begin{theorem}\label{T:disintergodic}
Let $G$ be a locally compact second countable group and let $\pi : (Y,\eta) \to (X,\nu)$ be a $G$-map of $G$-spaces.
Let $\varphi : X \to S(G)$ be a $G$-equivariant measurable map such that $\stab(x) \subseteq \varphi(x)$ for $\nu$-almost every $x \in X$.  Let $(Z,\zeta)$ be the quotient space of $(X,\nu)$ over $\varphi$ and let $\psi : (X,\nu) \to (Z,\zeta)$ be the corresponding map.  Then for almost every $z \in Z$, the disintegration measure $D_{\psi}(z)$ is $\stab(z)$-ergodic.  More precisely, if $f \in L^{\infty}(X,\nu)$ such that for almost every $x \in X$ and all $g \in \stab(\psi(x))$ it holds that $f(gx) = f(x)$ then $f$ descends to $L^{\infty}(Z,\zeta)$: there exists $F \in L^{\infty}(Z,\zeta)$ such that $f = F \circ \psi$.
\end{theorem}
\begin{proof}
Let $f \in L^{\infty}(X,\nu)$ such that for almost every $x \in X$ and all $g \in \stab(z)$ it holds that $f(gx) = f(x)$.  Since $\varphi(x) \subseteq \stab(\psi(x))$, then $f$ is a $\varphi$-invariant function.  Since $(Z,\zeta)$ is the point realization of all $\varphi$-invariant functions, then $f$ descends to a function on $Z$.  
\end{proof}

\begin{theorem}\label{T:irsfactor}
Let $G \actson (X,\nu)$ be a quasi-invariant action and let $\varphi : X \to S(G)$ be a $G$-equivariant measurable map such that $\stab(x) \subseteq \varphi(x)$ for $\nu$-almost every $x \in X$.  Let $(Z,\zeta)$ be the quotient space of $(X,\nu)$ over $\varphi$ and let $\psi : (X,\nu) \to (Z,\zeta)$ be the map.  Then $\varphi(x) \subseteq \stab(\psi(x))$ for $\nu$-almost every $x \in X$.

Moreover, the quotient space has the following universal property: if $\pi : (X,\nu) \to (Y,\eta)$ is a $G$-map such that $\varphi(x) \subseteq \stab(\pi(x))$ for $\nu$-almost every $x \in X$ then there exist $G$-maps $\psi : (X,\nu) \to (Z,\zeta)$ and $\tau : (Z,\zeta) \to (Y,\eta)$ such that $\tau \circ \psi = \pi$.
\end{theorem}
\begin{proof}
Consider the disintegration $D_{\psi} : Z \to P(X)$ of $\nu$ over $\zeta$.  For $f \in L^{\infty}(X,\nu)$ the function $Ef(x) = D_{\psi}(\psi(x))(f)$ is the conditional expectation to the space of $\varphi$-invariant functions.  
Therefore, for almost every $x \in X$ and all $g \in \varphi(x)$ it holds that $Ef(gx) = Ef(x)$.  Then $D_{\psi}(\psi(gx)) = D_{\psi}(\psi(x))$ as this holds for all $f$.  Hence the supports agree meaning that $\psi^{-1}(\psi(gx)) = \psi^{-1}(\psi(x))$ and so $g\psi(x) = \psi(gx) = \psi(x)$.  Therefore $\varphi(x) \subseteq \stab(\psi(x))$.

Let $f \in L^{\infty}(Y,\eta)$.  Then $f \circ \pi \in L^{\infty}(X,\nu)$.  For almost every $x \in X$ and all $g \in \varphi(x)$ we have that $f \in \stab(\pi(x))$ so $f \circ \pi (gx) = f \circ \pi(x)$ so $f$ is $\varphi$-invariant.  Hence $\pi^{*}(L^{\infty}(Y,\eta))$ is a closed $G$-invariant subalgebra of the $\varphi$-invariant functions so the maps $\psi$ and $\tau$ follow from Mackey's point realization (Theorem \ref{T:mackey}).
\end{proof}

\begin{theorem}\label{T:univpropquotnono}
Let $G$ be a locally compact second countable group and let $\pi : (Y,\eta) \to (X,\nu)$ be a $G$-map of $G$-spaces.
Let $\varphi : X \to S(G)$ be a $G$-equivariant measurable map such that $\stab(x) \subseteq \varphi(x)$ for $\nu$-almost every $x \in X$.  Let $(Z,\zeta)$ be the quotient space of $(X,\nu)$ over $\varphi$ and let $(W,\rho)$ be the quotient space of $(Y,\eta)$ over $\varphi \circ \pi$.  Let $\psi : (X,\nu) \to (Z,\zeta)$ and $\xi : (Y,\eta) \to (W,\rho)$ be the corresponding $G$-maps.  Then there exists a $G$-map $\tau : (W,\rho) \to (Z,\zeta)$ such that the following diagram commutes:
\begin{diagram}
(Y,\eta)		&\rTo^{\pi}	&(X,\nu)\\
\dTo^{\xi}	&			&\dTo^{\psi}\\
(W,\rho)		&\rTo^{\tau}	&(Z,\zeta)
\end{diagram}

Moreover, if $\pi$ is an orbital $G$-map and $\psi$ has the property that $\stab(\psi(x)) = \varphi(x)$ almost everywhere then $\tau$ is orbital (and, in particular, $\xi$ has the property that $\stab(\xi(y)) = \varphi(\pi(y))$ almost everywhere).
\end{theorem}
\begin{proof}
Let $f \in L^{\infty}(X,\nu)$ be $\varphi$-invariant.  Then $f \circ \pi \in L^{\infty}(Y,\eta)$ is $\varphi \circ \pi$-invariant.  Therefore at the level of algebras, $L^{\infty}(Z,\zeta)$ is a closed $G$-invariant subalgebra of $L^{\infty}(W,\rho)$ and the required map exists by restricting $\pi$ to the $\varphi \circ \pi$-invariant functions.

Assume now that $\pi$ is orbital and that $\stab(\psi(x)) = \varphi(x)$ almost everywhere.  Since the diagram commutes and $\stab(\psi(x)) = \varphi(x)$ almost everywhere, for almost every $y \in Y$,
\[
\stab(\xi(y)) \subseteq \stab(\tau(\xi(y))) = \stab(\psi(\pi(y))) = \varphi(\pi(y)).
\]
On the other hand, since $(W,\rho)$ is the quotient of $(Y,\eta)$ by $\varphi \circ \pi$, for almost every $y \in Y$,
\[
\varphi(\pi(y)) \subseteq \stab(\xi(y)).
\]
Therefore $\stab(\xi(y)) = \varphi(\pi(y))$ almost everywhere.  Hence, for almost every $y \in Y$,
\[
\stab(\xi(y)) = \varphi(\pi(y)) = \stab(\psi(\pi(y))) = \stab(\tau(\xi(y)))
\]
and so for almost every $w \in W$, then $\stab(w) = \stab(\tau(w))$ so $\tau$ is orbital.
\end{proof}

\begin{corollary}\label{C:univproporbital}
Let $G$ be a locally compact second countable group and let $\pi : (Y,\eta) \to (X,\nu)$ be a $G$-map of $G$-spaces.
Let $\varphi : X \to S(G)$ be a $G$-equivariant measurable map such that $\stab(x) \subseteq \varphi(x)$ for $\nu$-almost every $x \in X$.  Let $(Z,\zeta)$ be the quotient space of $(X,\nu)$ over $\varphi$ and let $\psi : (X,\nu) \to (Z,\zeta)$ be the corresponding map.  

Define the map $\phi : X \to S(G)$ by $\phi(x) = \stab(\psi(x))$.  Then $\phi$ is a $G$-equivariant measurable map and the quotient of $(X,\nu)$ by $\phi$ is isomorphic to $(Z,\zeta)$.

Let $(W,\rho)$ be the quotient space of $(Y,\eta)$ over $\varphi \circ \pi$ and let $(W^{\prime},\rho^{\prime})$ be the quotient of $(Y,\eta)$ over $\phi \circ \pi$.  Then there exists a commuting diagram of $G$-maps:
\begin{diagram}
(Y,\eta)		&\rTo^{\pi}	&(X,\nu) \\
\dTo			&			&\dTo_{\psi} \\
(W,\rho)		&\rTo		&(Z,\zeta) \\
\dTo			&			&\dTo_{\simeq} \\
(W^{\prime},\rho^{\prime})	&\rTo	&(Z,\zeta)
\end{diagram}
Moreover, if $\psi$ is orbital then the map $(W^{\prime},\rho^{\prime}) \to (Z,\zeta)$ is orbital.
\end{corollary}
\begin{proof}
Let $(Z^{\prime},\zeta^{\prime})$ be the quotient of $(X,\nu)$ by $\phi$.  Then there is a $G$-map $(Z,\zeta) \to (Z^{\prime},\zeta^{\prime})$ since $\varphi(x) \subseteq \phi(x)$.  On the other hand, by Theorem \ref{T:irsfactor}, since $\phi(x) = \stab(\psi(x))$ there is a $G$-map $(Z^{\prime},\zeta^{\prime}) \to (Z,\zeta)$ and therefore they are isomorphic.  

When $\psi$ is orbital, that the map $W^{\prime} \to Z$ is orbital follows from the previous theorem since by construction, $\stab(\psi(x)) = \phi(x)$.
\end{proof}

\subsection{The Quotient Space as a Functor}

\begin{definition}
Let $G$ be a locally compact second countable group and let $\Phi : S(G) \to S(G)$ be a conjugation-equivariant map such that $H \subseteq \Phi(H)$ for all $H \in S(G)$ and such that for $H,L \in S(G)$, if $H \subseteq L$ then $\Phi(H) \subseteq \Phi(L)$.

For a $G$-space $(X,\nu)$, define $F^{\Phi}(X,\nu)$ to be the quotient space of $(X,\nu)$ by the map $\varphi = \Phi \circ \mathrm{stab}$ where $\mathrm{stab}(x) = \{ g \in G : gx = x \}$.
\end{definition}

\begin{theorem}\label{T:functor}
Let $G$ be a locally compact second countable group and let $\Phi : S(G) \to S(G)$ be a conjugation-equivariant map such that $H \subseteq \Phi(H)$ for all $H \in S(G)$.  Then $F^{\Phi}$ is a functor on $G$-spaces and $G$-maps.
\end{theorem}
\begin{proof}
Let $\pi : (X,\nu) \to (Y,\eta)$ be a $G$-map of $G$-spaces.  Let $\varphi : Y \to S(G)$ by $\varphi(y) = \Phi(\stab(y))$ and $\psi : X \to S(G)$ by $\psi(x) = \Phi(\stab(x))$.  Then $\varphi \circ \pi (x) = \Phi(\stab(\pi(x))) \supseteq \Phi(\stab(x)) = \psi(x)$.  Writing $\rpf{X}{\psi}$ for the quotient of $(X,\nu)$ by $\psi$ (and likewise writing $\rpf{X}{\varphi \circ \pi}$ and $\rpf{Y}{\varphi}$), by Theorem \ref{T:irsfactor} there exists a $G$-map $\tau : \rpf{X}{\psi} \to \rpf{X}{\varphi\circ\pi}$ and, combining this map with the diagram obtained from Theorem \ref{T:univpropquotnono}, the following diagram of $G$-maps exists and commutes (omitting measures for clarity):
\begin{diagram}
X			&\rTo^{=}		&X					&\rTo^{\pi}		&Y \\
\dTo			&			&\dTo				&			&\dTo \\
\rpf{X}{\psi}	&\rTo^{\tau}		&\rpf{X}{\varphi\circ\pi}	&\rTo		&\rpf{Y}{\varphi}.
\end{diagram}
Ignoring the middle column, this says precisely that our construction defines a functor.
\end{proof}

\subsection{Quotients of Affine Spaces}\label{ss:affinequot}

Let $(X,\nu)$ be an ergodic $G$-space.  Let $E$ be a Banach space and $\alpha : G \times X \to \mathrm{Iso}(E)$ be a cocycle.  Let $A_{x} \subseteq E_{1}^{*}$ be a closed convex nonempty subset for each $x \in X$ such that $\alpha^{*}(g,x)A_{gx} = A_{x}$ for all $g \in G$ and $x \in X$ (where $\alpha^{*}(g,x) = (\alpha(g,x)^{-1})^{*}$ is the adjoint cocycle).  Let $\pi : (X,\nu) \to (Z,\zeta)$ be a $G$-map.

Define the closed subspace
\[
E_{x}^{\pi} = \{ e \in E : \alpha(gh,x)e = \alpha(g,x)e \text{ for all $g \in G$ and all $h \in \stab(\pi(x))$ } \}.
\]
Observe that in particular $\alpha(h,x)e = e$ for all $h \in \stab(\pi(x))$ for every $x \in X$ and $e \in E_{x}^{\pi}$.  Also observe that for any $e \in E_{x}^{\pi}$, any $g \in G$, any $h \in \stab(\pi(x))$ and any $k \in \stab(\pi(hx))$, writing $k = h\ell h^{-1}$ for some $\ell \in \stab(\pi(x))$,
\begin{align*}
\alpha(gk,hx)e &= \alpha(gh\ell h^{-1},hx)e = \alpha(gh\ell,x)\alpha(h^{-1},hx)e \\
&= \alpha(gh\ell,x)(\alpha(h,x)^{-1})e = \alpha(gh\ell,x)e \\
&= \alpha(gh,x)e = \alpha(g,hx)\alpha(h,x)e = \alpha(g,hx)e
\end{align*}
meaning that $E_{hx}^{\pi} = E_{x}^{\pi}$ for all $h \in \stab(\pi(x))$.  By Theorem \ref{T:disintergodic} then $E_{x}^{\pi}$ is constant on fibers over $Z$.  So write $E_{z}^{\pi} = E_{x}^{\pi}$ for $D_{\pi}(z)$-almost every $x \in \pi^{-1}(z)$.

Now given $g \in G$ and $e \in E_{z}^{\pi}$, for any $x \in \pi^{-1}(z)$ and $h \in \stab(\pi(x))$,
\[
\alpha(g,hx)e = \alpha(gh,x)\alpha(h,x)^{-1}e = \alpha(g,hx)e = \alpha(g,x)e
\]
so $\alpha(g,\cdot)e$ is $\stab(\pi(x))$-invariant meaning that (again by Theorem \ref{T:disintergodic}) it descends to $\beta(g,z) = \alpha(g,x)$ for almost every $x \in \pi^{-1}(z)$.
Note that for all $g \in $G and $z \in Z$, it holds that $\beta(g,z) \in \mathrm{Iso}(E_{z}^{\pi} \to E_{gz}^{\pi})$ and that $\beta$ is a cocycle.

Define $E^{\pi} = \bigcap_{z} E_{z}^{\pi}$ which is a closed subspace of $E$.  Observe now that for any $q,k \in G$, any $z \in Z$ and any $e \in E^{\pi}$ it holds that $e \in E_{q^{-1}kqz}^{\pi}$ and so $\beta(k^{-1},kqz)\beta(q,z)e \in \beta(k^{-1},kqz)E_{kqz}^{\pi} = E_{qz}^{\pi}$.  Therefore for any $g \in G$ and $h \in \stab(kqz)$, writing $h = k\ell k^{-1}$ for $\ell \in \stab(qz)$,
\begin{align*}
\beta(gh,kqz)\beta(q,z)e &= \beta(ghk,qz)\beta(k,qz)^{-1}\beta(q,z)e \\
&= \beta(gk\ell,qz)\beta(k^{-1},kqz)\beta(q,z)e \\
&= \beta(gk,qz)\beta(k^{-1},kqz)\beta(q,z)e
= \beta(g,kqz)\beta(q,z)e
\end{align*}
using that $\ell \in \stab(qz)$ and that $\beta(k^{-1},kqz)\beta(q,z)e \in E_{qz}^{\pi}$ to move from the second line to the third.  This means that $\beta(q,z)e \in E_{kqz}^{\pi}$ for all $k \in G$.  By ergodicity (and Theorem \ref{T:disintergodic}) then $\beta(q,z)e \in E^{\pi}$.  As this holds for all $g \in G$ and $z \in Z$, this means that $\beta : G \times Z \to \mathrm{Iso}(E^{\pi})$ is a well-defined cocycle.

Consider now $a \in A_{x}$ and $g \in \stab(\pi(x))$.  For $e \in E^{\pi}$,
\[
(\alpha^{*}(g,x)a)(e) = a(\alpha(g,x)^{-1}e) = a(e)
\]
so the map $r : E_{1}^{*} \to (E^{\pi})_{1}^{*}$ given by restricting to $E^{\pi}$ has the property that if $a \in A_{x}$ and $g \in \stab(\pi(x))$ then $r(\alpha^{*}(g,x)a) = r(a)$.  By Theorem \ref{T:disintergodic}, since $(X,\nu)$ is ergodic, $\stab(z) \actson (\pi^{-1}(z), D_{\pi}(z))$ is ergodic almost surely.  Therefore, for each $z \in Z$, the set $B_{z} \subseteq (E^{\pi})_{1}^{*}$ given by
\[
B_{z} = \{ a\big{|}_{E^{\pi}} : a \in A_{x} \text{ for some $x \in X$ such that $\pi(x) = z$} \}
\]
is well-defined and $\beta^{*}(g,z)B_{gz} = B_{z}$.

The affine space $A \subseteq X \times_{\alpha^{*}} E_{1}^{*}$ then maps to the affine space $B \subseteq Z \times_{\beta^{*}} (E^{\pi})_{1}^{*}$ by $(x,q) \mapsto (\pi(x),r(q))$.  The space $B$ is the \textbf{quotient affine space of $A$ by $\pi$}.  By construction, $B$ is orbital over $(Z,\zeta)$ since $\beta(g,z) = e$ for $g \in \stab(z)$.

Note that the quotient affine space requires that $\pi$ be a $G$-map to a $G$-space (and is not well-defined for an arbitrary $G$-equivariant $\varphi : X \to S(G)$).  However, if one first takes the quotient of $(X,\nu)$ by $\varphi$ and then applies the above construction, one still obtains an orbital affine space over the quotient of $(X,\nu)$ by $\varphi$.  Therefore, given a $G$-equivariant measurable map $\varphi : X \to S(G)$, we define the \textbf{quotient affine space of $A$ by $\varphi$} to be the quotient of $A$ by $\pi$ where $\pi$ is the map $(X,\nu) \to (Z,\zeta)$ such that $(Z,\zeta)$ is the quotient of $(X,\nu)$ by $\varphi$.

Consider now a Borel function $f : A \to \mathbb{R}$ such that $f(g \cdot (x,q)) = f(x,q)$ for all $g \in \stab(\pi(x))$.  Then $f$ descends to a function on $B$ by the construction of $B$.  Therefore, if $\alpha \in P(A)$ is any probability measure such that $(\mathrm{proj}_{X})_{*}\alpha = \nu$, the $\stab(\pi(\mathrm{proj}_{X}(a)))$-invariant functions in $L^{\infty}(A,\alpha)$ are in fact in $L^{\infty}(B,\beta)$ (where $\beta$ is the pushforward of $\alpha$ to $B$).  Likewise, any Borel function on $B$ extends to a Borel function on $A$ that is $\stab(\pi(\mathrm{proj}_{X}(a)))$-invariant.  Therefore, the quotient space of $(A,\alpha)$ by $\stab(\pi(\mathrm{proj}_{X}(a)))$ is isomorphic to $(B,\beta)$.

\subsection{The Product Random Subgroups Functor}\label{s:prg}

One reason for introducing the notion of quotienting by random subgroups is to construct the product random subgroups functor.  The product random subgroups functor will play the role in our work that the ergodic decomposition does in the work of Bader and Shalom \cite{BS04} and in this sense is the key to our study of actions.

\begin{definition}
Let $G_{1}$ and $G_{2}$ be locally compact second countable groups and set $G = G_{1} \times G_{2}$.  Let $\Phi : S(G) \to S(G)$ be given by $\Phi(H) = \overline{\mathrm{proj}_{G_{1}}~H} \times \overline{\mathrm{proj}_{G_{2}}~H}$.

The \textbf{product random subgroups functor}, denoted by $PRG$, is the quotient functor $F^{\Phi}$: for a $G$-space $(X,\nu)$, the quotient space of $(X,\nu)$ by $\Phi \circ \mathrm{stab}_{G}$ is written $PRG(X,\nu)$ and for a $G$-map $\pi : (X,\nu) \to (Y,\eta)$, the map between quotient spaces is written $PRG(\pi) : PRG(X,\nu) \to PRG(Y,\eta)$.
\end{definition}

\begin{proposition}
Let $G = G_{1} \times G_{2}$ be a product of locally compact second countable groups.  Let $\pi : (X,\nu) \to (Y,\eta)$ be a $G$-map of $G$-spaces.  Then there exists a $G$-map $PRG(\pi) : PRG(X,\nu) \to PRG(Y,\eta)$ such that $PRG \circ \pi = PRG(\pi) \circ PRG$.  That is, $PRG$ is a functor on $G$-spaces.
\end{proposition}
\begin{proof}
Clearly, for $H \in S(G)$, $H \subseteq \overline{\mathrm{proj}_{G_{1}}~H} \times \overline{\mathrm{proj}_{G_{2}}~H}$ and for $H, L \in S(G)$ with $H \subseteq L$, $\overline{\mathrm{proj}_{G_{1}}~H} \times \overline{\mathrm{proj}_{G_{2}}~H} \subseteq \overline{\mathrm{proj}_{G_{1}}~L} \times \overline{\mathrm{proj}_{G_{2}}~L}$.
Then the result is Theorem \ref{T:functor}.
\end{proof}

\begin{proposition}\label{P:something}
Let $G = G_{1} \times G_{2}$ be a product of locally compact second countable groups and let $(X,\nu)$ be a $G$-space.  Denote by $(X_{1},\nu)$ and $(X_{2},\nu_{2})$ the spaces of $G_{1}$- and $G_{2}$-ergodic components (the invariant products functor applied to $(X,\nu)$) and by $PRG(X,\nu)$ the quotient space of $(X,\nu)$ over the map $\Phi \circ \mathrm{stab}$ where $\Phi(H) = \overline{\mathrm{proj}_{G_{1}}~H} \times \overline{\mathrm{proj}_{G_{2}}~H}$.  Then there exist $G$-maps
\[
(X,\nu) \to PRG(X,\nu) \to (X_{1} \times X_{2}, \nu_{1} \times \nu_{2}).
\]
\end{proposition}
\begin{proof}
Let $\pi_{1} : (X,\nu) \to (X_{1},\nu_{1})$ be the ergodic decomposition map.  Let $g \in \mathrm{stab}_{G}(x)$.  Write $g = g_{1}g_{2}$ for $g_{1} \in G_{1}$ and $g_{2} \in G_{2}$.  Then $gx = x$ so $\pi_{1}(x) = \pi_{1}(gx) = g\pi_{1}(x) = g_{1}\pi(x)$ since $G_{2}$ acts trivially on $X_{1}$.  Therefore $\mathrm{proj}_{G_{1}}~\mathrm{stab}_{G}(x) \subseteq \mathrm{stab}_{G_{1}}(\pi_{1}(x))$.  Since the stabilizer subgroups are always closed, then $\overline{\mathrm{proj}_{G_{1}}~\mathrm{stab}_{G}(x)} \subseteq \mathrm{stab}_{G_{1}}(\pi_{1}(x))$.  Of course the same holds for $G_{2}$.

Let $\pi : (X,\nu) \to (X_{1},\nu_{1}) \times (X_{2},\nu_{2})$ by $\pi(x) = (\pi_{1}(x),\pi_{2}(x))$.  Then
\[
\mathrm{stab}_{G}(\pi(x)) = \mathrm{stab}_{G_{1}}(\pi_{1}(x)) \times G_{2} \cap G_{1} \times \mathrm{stab}_{G_{2}}(\pi_{2}(x)) \supseteq s(x).
\]
By the universal property of the quotient space then there exists $\tau : PRG(X,\nu) \to F(X,\nu)$ such that $\tau \circ \psi = \pi$ and the conclusion follows.
\end{proof}

\begin{theorem}\label{T:prgmp}
Let $G = G_{1} \times G_{2}$ be a product of locally compact second countable groups.  Let $\mu_{1} \in P(G_{1})$ and $\mu_{2} \in P(G_{2})$ be admissible probability measures and set $\mu = \mu_{1} \times \mu_{2}$.  Let $PRG(X,\nu)$ be the quotient space of $(X,\nu)$ by $\varphi(x) = \overline{\mathrm{proj}_{G_{1}}~\stab(x)} \times \overline{\mathrm{proj}_{G_{2}}~\stab(x)}$.  If $(X,\nu)$ is an ergodic $\mu$-stationary $G$-space then the $G$-map $(X,\nu) \to PRG(X,\nu)$ is relatively measure-preserving.
\end{theorem}
\begin{proof}
By Proposition \ref{P:1.10}, the $G$-map $(X,\nu) \to (X_{1} \times X_{2}, \nu_{1} \times \nu_{2})$ is relatively measure-preserving.  The previous proposition shows that $PRG(X,\nu)$ is an intermediate quotient of these spaces hence the $G$-maps $(X,\nu) \to PRG(X,\nu)$ and $PRG(X,\nu) \to (X_{1} \times X_{2}, \nu_{1} \times \nu_{2})$ are relatively measure-preserving.
\end{proof}

\section{Relative Joinings Over Relatively Contractive Maps}

Relatively contractive maps were introduced in \cite{CP12} and can be used to show that any joining between a contractive space and a measure-preserving space such that the projection to the contractive space is relatively measure-preserving is necessarily the independent joining.  We generalize this fact to the case of relative joinings and obtain an analogous result.

\begin{theorem}
Let $(X,\nu)$ and $(Y,\eta)$ be $G$-spaces with a common $G$-quotient $(Z,\zeta)$ such that $\varphi : (Y,\eta) \to (Z,\zeta)$ is relatively contractive and $\pi : (X,\nu) \to (Z,\zeta)$ is a $G$-map.  Then there exists at most one relative joining of $(X,\nu)$ and $(Y,\eta)$ over $(Z,\zeta)$ such that the projection to $(Y,\eta)$ is relatively measure-preserving.
\end{theorem}
\begin{proof}
For convenience, write $W = X \times Y$.
Let $\rho$ be a relative joining of $(X,\nu)$ and $(Y,\eta)$ over $(Z,\zeta)$ such that $\varphi : (Y,\eta) \to (Z,\zeta)$ is relatively contractive, $p_{Y} : (W,\rho) \to (Y,\eta)$ is relatively measure-preserving and $p_{X} : (W,\rho) \to (X,\nu)$ and $\pi : (X,\nu) \to (Z,\zeta)$ are $G$-maps such that $\pi \circ p_{X} = \varphi \circ p_{Y}$ almost everywhere.  Denote by $\psi : (W,\rho) \to (Z,\zeta)$ the composition: $\psi = \pi \circ p_{X} = \varphi \circ p_{Y}$.

Let $z \in Z$ and let $f \in L^{\infty}(\pi^{-1}(z), D_{\pi}(z))$ be arbitrary.  Then $f \circ p_{X} \in L^{\infty}(\psi^{-1}(z), D_{\psi}(z))$ since $D_{\psi}(z) = \int D_{p_{X}}(x)~dD_{\pi}(z)(x)$.  Define 
\[
F(y) = D_{p_{Y}}(y)(f \circ p_{X})
\]
and observe that $F \in L^{\infty}(\varphi^{-1}(z), D_{\varphi}(z))$.

For an arbitrary $g \in G$, using that $p_{Y}$ is relatively measure-preserving,
\begin{align*}
D_{\varphi}^{(g)}(z)(F) &= \int_{\varphi^{-1}(z)} F(y)~dg^{-1}D_{\varphi}(gz) \\
&= \int_{\varphi^{-1}(gz)} F(g^{-1}y)~dD_{\varphi}(gz) \\
&= \int_{\varphi^{-1}(gz)} \int_{p_{Y}^{-1}(g^{-1}y)} f(p_{X}(w))~dD_{p_{Y}}(g^{-1}y)(w)~dD_{\varphi}(gz)(y) \\
&= \int_{\varphi^{-1}(gz)} \int_{p_{Y}^{-1}(g^{-1}y)} f(p_{X}(w))~dg^{-1}D_{p_{Y}}(y)(w)~dD_{\varphi}(gz)(y) \\
&= \int_{\varphi^{-1}(gz)} \int_{p_{Y}^{-1}(y)} f(p_{X}(g^{-1}w))~dD_{p_{Y}}(y)(w)~dD_{\varphi}(gz)(y) \\
&= \int_{\varphi^{-1}(gz)} \int_{p_{Y}^{-1}(y)} f(g^{-1} p_{X}(w))~dD_{p_{Y}}(y)(w)~dD_{\varphi}(gz)(y)
\end{align*}
Now $\int_{\varphi^{-1}(gz)} D_{p_{Y}}(y)~dD_{\varphi}(gz)(y) = D_{\psi}(gz)$ and therefore
\begin{align*}
D_{\varphi}^{(g)}(z)(F) &= \int_{\psi^{-1}(gz)} f(g^{-1}p_{X}(w))~dD_{\psi}(gz)(w) \\
&= \int_{p_{X}(\psi^{-1}(gz))} f(g^{-1}x)~d((p_{X})_{*}D_{\psi}(gz))(x) \\
&= \int_{\pi^{-1}(gz)} f(g^{-1}x)~dD_{\pi}(gz)(x) \\
&= D_{\pi}^{(g)}(z)(f).
\end{align*}

Now let $\rho_{1}$ and $\rho_{2}$ both be relative joinings over $(Z,\zeta)$.  Since $\varphi$ is relatively contractive, there is a measure one set of $z \in Z$ such that for all $F \in L^{\infty}(\varphi^{-1}(z),D_{\varphi}(z))$, we have that $\sup_{g \in G} |D_{\varphi}^{(g)}(F)| = \| F \|$.  Fix $z$ in this measure one set.

Let $f \in L^{\infty}(\pi^{-1}(z),D_{\pi}(z))$ be arbitrary.  Let $D_{p_{Y}}^{j}$ and $D_{\psi}^{j}$ for $j = 1,2$ denote the disintegrations of $\rho_{1}$ and $\rho_{2}$ over $\eta$ and $\zeta$, respectively.
Define, for $j = 1,2$,
\[
F_{j}(y) = D_{p_{Y}}^{j}(f \circ p_{X})
\]
and set $F(y) = F_{1}(y) - F_{2}(y)$.  As above, $F \in L^{\infty}(\varphi^{-1}(z),D_{\varphi}(z))$.  Now, by the above, for any $g \in G$,
\[
D_{\varphi}^{(g)}(z)(F_{1}) = D_{\pi}^{(g)}(z)(f) = D_{\varphi}^{(g)}(z)(F_{2})
\]
and therefore $D_{\varphi}^{(g)}(z)(F) = 0$.

Since $z$ is in the measure one set where that map is an isometry, $\| F \| = \sup_{g} |D_{\varphi}^{(g)}(z)(F)| = 0$.  Therefore $F = 0$ almost everywhere.  As this holds for all $f \in L^{\infty}(\pi^{-1}(z),D_{\pi}(z))$, we conclude that $D_{\varphi}^{1}(y) = D_{\varphi}^{2}(y)$ for $D_{\varphi}(z)$-almost-every $y \in \varphi^{-1}(z)$.

Now let $f \in L^{\infty}(\psi^{-1}(z),D_{\psi}(z))$ be arbitrary and observe that
\begin{align*}
D_{\psi}^{j}(z)(f) &= \int_{\psi^{-1}(z)} f(x,y)~dD_{\psi}^{j}(z)(x,y) \\
&= \int_{\varphi^{-1}(z)} \int_{p_{Y}^{-1}(y)} f(x,y)~dD_{p_{Y}}^{j}(y)(x)~dD_{\varphi}(z)(y).
\end{align*}
Since $D_{p_{Y}}^{1}(y) = D_{p_{Y}}^{2}(y)$ for $D_{\varphi}(z)$-almost every $y$,
\[
D_{\psi}^{1}(z)(f) = D_{\psi}^{2}(z)(f).
\]
This holds for all $f \in L^{\infty}(\psi^{-1}(z),D_{\psi}(z))$ and so $D_{\psi}^{1}(z) = D_{\psi}^{2}(z)$.

Since the above holds for all $z$ in a measure one set,
\[
\rho_{1} = \int_{Z} D_{\psi}^{1}(z)~d\zeta(z) = \int_{Z} D_{\psi}^{2}(z)~d\zeta(z) = \rho_{2}.
\]
\end{proof}

\begin{corollary}
Let $(X,\nu)$ and $(Y,\eta)$ be $G$-spaces with a common $G$-quotient $(Z,\zeta)$ such that $\varphi : (Y,\eta) \to (Z,\zeta)$ is relatively contractive and $\pi : (X,\nu) \to (Z,\zeta)$ is relatively measure-preserving.  Then the only relative joining of $(X,\nu)$ and $(Y,\eta)$ over $(Z,\zeta)$ such that the projection to $(Y,\eta)$ is relatively measure-preserving is the independent relative joining.
\end{corollary}
\begin{proof}
By the previous theorem, we need only show that the independent relative joining $\rho = \int D_{\pi} \times D_{\varphi}~d\zeta$ is a relative joining such that the projection to $(Y,\eta)$ is relatively measure-preserving.  Let $D_{p_{Y}}$ be the disintegration of $\rho$ over $\eta$.  Observe that $p_{Y}^{-1}(y) = \pi^{-1}(\varphi(y)) \times \{ y \}$ and that the support of $D_{\pi}(\varphi(y)) \times \delta_{y}$ is the same.  Now
\begin{align*}
\int_{Y} D_{\pi}(\varphi(y)) \times \delta_{y}~d\eta(y) &= \int_{Z} \int_{Y} D_{\pi}(z) \times \delta_{y}~dD_{\varphi}(z)(y)~d\eta(y) \\
&= \int_{Z} D_{\pi}(z) \times D_{\varphi}(z)~d\zeta(z) = \rho
\end{align*}
so by uniqueness, $D_{p_{Y}}(y) = D_{\pi}(\varphi(y)) \times \delta_{y}$ almost everywhere.  
Then, using that $\pi$ is relatively measure-preserving,
\[
D_{p_{Y}}(gy) = D_{\pi}(\varphi(gy)) \times \delta_{gy} = gD_{\pi}(\varphi(y)) \times g \delta_{y} = g D_{p_{Y}}(y)
\]
so $p_{Y}$ is relatively measure-preserving.  By the previous theorem, $\rho$ is then the unique relative joining.
\end{proof}

\begin{corollary}\label{C:reljoinunique}
Let $G$ be a locally compact second countable group and let $(X,\nu), (Y,\eta), (Z,\zeta)$ and $(W,\rho)$ be $G$-spaces such that the following diagram of $G$-maps commutes:
\begin{diagram}
(W,\rho)		&\rTo^{\psi}	&(X,\nu)\\
\dTo^{\tau}	&			&\dTo^{\pi}\\
(Y,\eta)		&\rTo^{\varphi}	&(Z,\zeta)
\end{diagram}
If $\tau$ and $\pi$ are relatively measure-preserving and $\psi$ and $\varphi$ are relatively contractive then $(W,\rho)$ is $G$-isomorphic to the independent relative joining of $(X,\nu)$ and $(Y,\eta)$ over $(Z,\zeta)$.
\end{corollary}
\begin{proof}
Consider the map $p : W \to X \times Y$ by $p(w) = (\psi(w), \tau(w))$.  Then $p_{*}\rho$ is a relative joining of $(X,\nu)$ and $(Y,\eta)$ over $(Z,\zeta)$.  Let $p_{X} : X \times Y \to X$ and $p_{Y} : X \times Y \to Y$ be the natural projections and observe that the following diagram commutes:
\begin{diagram}
(W,\rho)		&\rTo^{p}	&(X \times Y,p_{*}\rho)	&\rTo^{p_{X}}	&(X,\nu)\\
			&		&\dTo^{p_{Y}}			&			&\dTo^{\pi}\\
			&		&(Y,\eta)				&\rTo^{\varphi}	&(Z,\zeta)
\end{diagram}
since $p_{X} \circ p = \psi$ and $p_{Y} \circ p = \tau$.

Now $\psi$ is relatively contractive so $p$ and $p_{X}$ are relatively contractive (Theorem \ref{T:relcontcomp}) and likewise $\tau$ being relatively measure-preserving implies $p$ and $p_{Y}$ are relatively measure-preserving.  Therefore $p$ is an isomorphism (Theorem \ref{T:relmprelcon}).  Since $\varphi$ is relatively contractive and $p_{Y}$ is relatively measure-preserving and $\pi$ is relatively measure-preserving, the previous corollary says that $p_{*}\rho$ is the independent relative joining.
\end{proof}

\section{The Intermediate Contractive Factor Theorem for Products}

We are now ready to prove a strengthening of the Bader-Shalom Intermediate Factor Theorem \cite{BS04}, our key improvement being the removal of the requirement that the $G$-space be irreducible (that is, ergodic for each $G_{j}$):
\begin{theorem}\label{T:IFT}
Let $G = G_{1} \times G_{2}$ be a product of two locally compact second countable groups and let $\mu_{j} \in P(G_{j})$ be admissible probability measures for $j=1,2$.  Set $\mu = \mu_{1} \times \mu_{2}$.

Let $(B,\beta)$ be the Poisson boundary for $(G,\mu)$ and let $(X,\nu)$ be a measure-preserving $G$-space.  Let $(W,\rho)$ be a $G$-space such that there exist $G$-maps $\pi : (B \times X, \beta\times\nu) \to (W,\rho)$ and $\varphi : (W,\rho) \to (X,\nu)$ with $\varphi \circ \pi$ being the natural projection to $X$.

Let $(W_{1},\rho_{1})$ be the space of $G_{2}$-ergodic components of $(W,\rho)$ and let $(W_{2},\rho_{2})$ be the space of $G_{1}$-ergodic components.  Likewise, let $(X_{1},\nu_{1})$ and $(X_{2},\nu_{2})$ be the ergodic components of $(X,\nu)$ for $G_{2}$ and $G_{1}$, respectively.

Then $(W,\rho)$ is $G$-isomorphic to the independent relative joining of $(W_{1},\rho_{1}) \times (W_{2},\rho_{2})$ and $(X,\nu)$ over $(X_{1},\nu_{1}) \times (X_{2},\nu_{2})$.
\end{theorem}
\begin{proof}
Let $(X_{1},\nu_{1})$ and $(X_{2},\nu_{2})$ be the spaces of $G_{1}$- and $G_{2}$-ergodic components of $(X,\nu)$, respectively.  Then $F^{G}(X,\nu) = (X_{1},\nu_{1}) \times (X_{2},\nu_{2})$.  Also, $F^{G}(W,\rho) = (W_{1},\rho_{1}) \times (W_{2},\rho_{2})$.  Consider now $F^{G}(B \times X, \beta \times \nu)$.

First, observe that since $(B,\beta)$ is a Poisson boundary, it is a contractive $G$-space (Theorem \ref{T:PBcont}).  Since $F^{G}(B,\beta) = (B_{1},\beta_{1}) \times (B_{2},\beta_{2})$ is a relatively measure-preserving quotient of a contractive space, by Corollary \ref{C:mpquotrelcon}, $(B,\beta)$ is $G$-isomorphic to $(B_{1},\beta_{1}) \times (B_{2},\beta_{2})$.  

Observe that $(B_{1},\beta_{1})$ is a $G_{1}$-quotient of the $(G_{1},\mu_{1})$ Poisson boundary since for $(\mu_{1}\times\mu_{2})^{\mathbb{N}}$-almost every sequence $\omega \in (G_{1} \times G_{2})^{\mathbb{N}}$, it holds that $\lim \omega_{1}\cdots\omega_{n}\beta$ is a point mass, and writing $\omega_{j} = (u_{j},v_{j})$, $\omega_{1}\cdots\omega_{n}\beta = u_{1}\cdots u_{n}\beta_{1} \times v_{1}\cdots v_{n}\beta_{2}$ shows that $\lim u_{1}\cdots u_{n}\beta_{1}$ is a point mass for $\mu_{1}^{\mathbb{N}}$-almost every sequence $u \in G_{1}^{\mathbb{N}}$ (in fact, $(B_{1},\beta_{1}$) is the Poisson boundary as shown in \cite{BS04} but we will not need that).

Now $G_{1}$ acts trivially on $(B_{2},\beta_{2})$ so by Proposition \ref{P:ergdecomptrivial},
\[
\rpf{(B \times X)}{G_{1}} = \rpf{(B_{1} \times B_{2} \times X)}{G_{1}} = \rpf{(B_{1} \times X)}{G_{1}} \times B_{2}.
\]
Since $(B_{1},\beta_{1})$ is a quotient of the Poisson boundary of $(G_{1},\mu_{1})$ and $(X,\nu)$ is a measure-preserving $G_{1}$-space, by Proposition \ref{P:ergdecomperg}, $\rpf{(B_{1} \times X)}{G_{1}} = \rpf{X}{G_{1}}$ so
\[
\rpf{(B \times X)}{G_{1}} = \rpf{X}{G_{1}} \times B_{2}.
\]
Likewise, $\rpf{(B \times X)}{G_{2}} = \rpf{X}{G_{2}} \times B_{1}$.  Therefore
\[
F^{G}(B \times X, \beta\times\nu) = (B_{1},\beta_{1}) \times (B_{2},\beta_{2}) \times (X_{1},\nu_{1}) \times (X_{2},\nu_{2})
\]
with the diagonal action.

Therefore, applying the functor $F^{G}$ to the given maps $B \times X \to W \to X$, we obtain the following commuting diagram of $G$-maps (the measures are omitted for clarity):
\begin{diagram}
B \times X								&\rTo^{\pi}			&W				&\rTo^{\varphi}			&X\\
\dTo									&				&\dTo			&					&\dTo\\
B_{1} \times B_{2} \times X_{1} \times X_{2}	&\rTo^{F^{G}(\pi)}	&W_{1}\times W_{2}	&\rTo^{F^{G}(\varphi)}	&X_{1}\times X_{2}
\end{diagram}
The vertical maps are all relatively measure-preserving by Proposition \ref{P:1.10} (Proposition 1.10 in \cite{BS04}).  Since $(B,\beta)$ is a contractive $G$-space and $(X,\nu)$ is a measure-preserving $G$-space, the natural projection $B \times X \to X$ is a relatively contractive $G$-map by Theorem \ref{T:relcontB}.  Therefore $\pi$ and $\varphi$ are relatively contractive by Theorem \ref{T:relcontcomp}.  Likewise, $(X_{1} \times X_{2},\nu_{1}\times\nu_{2})$ is a measure-preserving $G$-space and the composition $F^{G}(\varphi) \circ F^{G}(\pi) = F^{G}(\varphi \circ \pi)$ is the natural projection to $X_{1} \times X_{2}$.  Therefore, since $B_{1} \times B_{2}$ is contractive, $F^{G}(\varphi\circ\pi)$ is relatively contractive.  Hence $F^{G}(\pi)$ and $F^{G}(\varphi)$ are both relatively contractive.

Isolating the right-hand side of the diagram:
\begin{diagram}
(W,\rho)							&\rTo	&(X,\nu)\\
\dTo								&		&\dTo\\
(W_{1},\rho_{1}) \times (W_{2},\rho_{2})	&\rTo	&(X_{1},\nu_{1})\times (X_{2},\nu_{2})
\end{diagram}
is a commuting diagram of $G$-maps such that the vertical arrows are relatively measure-preserving, the horizontal arrows are relatively contractive and $(X,\nu)$ is measure-preserving.  By Corollary \ref{C:reljoinunique}, $(W,\rho)$ is $G$-isomorphic to the independent relative joining of $(W_{1},\rho_{1}) \times (W_{2},\rho_{2})$ and $(X,\nu)$ over $(X_{1},\nu_{1}) \times (X_{2},\nu_{2})$ as claimed.
\end{proof}

\begin{corollary}[Bader-Shalom Intermediate Factor Theorem \cite{BS04}]
Let $G = G_{1} \times G_{2}$ be a product of two locally compact second countable groups and let $\mu_{j} \in P(G_{j})$ be admissible probability measures for $j=1,2$.  Set $\mu = \mu_{1} \times \mu_{2}$.

Let $(B,\beta)$ be the Poisson boundary for $(G,\mu)$ and let $(X,\nu)$ be a measure-preserving $G$-space that is ergodic for each $G_{j}$.  Let $(W,\rho)$ be a $G$-space such that there exist $G$-maps $\pi : (B \times X, \beta\times\nu) \to (W,\rho)$ and $\varphi : (W,\rho) \to (X,\nu)$ with $\varphi \circ \pi$ being the natural projection to $X$.

Then $(W,\rho)$ is $G$-isomorphic to $(W_{1},\rho_{1}) \times (W_{2},\rho_{2}) \times (X,\nu)$ where $(W_{1},\rho_{1})$ is a $(G_{1},\mu_{1})$-boundary and $(W_{2},\rho_{2})$ is a $(G_{2},\mu_{2})$-boundary.
\end{corollary}
\begin{proof}
Since the action of each $G_{j}$ is ergodic on $(X,\nu)$, the ergodic components spaces $(X_{1},\nu_{1})$ and $(X_{2},\nu_{2})$ are both trivial.  The previous theorem then implies that $(W,\rho)$ is $G$-isomorphic to the independent relative joining of $(W_{1},\rho_{1}) \times (W_{2},\rho_{2})$ and $(X,\nu)$ over the trivial system, that is $(W,\rho)$ is $G$-isomorphic to $(W_{1},\rho_{1}) \times (W_{2},\rho_{2}) \times (X,\nu)$.  Since $F^{G}$ is a functor and the $X_{j}$ are trivial, applying $F^{G}$ to the maps $B \times X \to W \to X$ gives $G$-maps $B \to W_{1} \times W_{2} \to 0$.  Therefore $(W_{1},\rho_{1})$ is a $(G,\mu)$-boundary on which the $G_{2}$-action is trivial, hence it is a $(G_{1},\mu_{1})$-boundary.  Likewise for $(W_{2},\rho_{2})$.
\end{proof}

\section{Actions of Products of Groups}\label{sec:stuff}

We now are ready to consider the stabilizers of actions of products of locally compact second countable groups.  For clarity, we present first the results for the products of two groups in this section and then later handle the general case.  Many of the results in this section hold for arbitrary groups, but some require the hypothesis that the factors be simple, hence the need to handle the case of products of more than two groups separately.

We begin by showing that the weak amenability of the action of a product of groups is equivalent to the weak amenability of the action on the product random subgroups functor space corresponding to the action.  From there, we deduce that if certain conditions hold on the spaces of ergodic components that the action is necessarily weakly amenable.  Combining this with property $(T)$ (and the relative version in the form of resolutions), we conclude with a classification of actions of products of such groups.
The study of actions of irreducible lattices in products of groups will be the subject of the next section.

\subsection{Weak Amenability and the Product Random Subgroups Functor}

\begin{theorem}\label{T:PRGweakamen}
Let $G_{1}$ and $G_{2}$ be locally compact second countable groups.  Let $G = G_{1} \times G_{2}$ and let $(X,\nu)$ be a measure-preserving $G$-space.  Assume that the $G$-action on $PRG(X,\nu)$ is weakly amenable.  Then the $G$-action on $(X,\nu)$ is weakly amenable.
\end{theorem}
\begin{proof}
Let $A$ be an affine orbital $G$-space over $(X,\nu)$.  Let $(C,\zeta)$ be the Poisson boundary of $G = G_{1} \times G_{2}$ for the measure $\mu = \mu_{1} \times \mu_{2}$ where $\mu_{j}$ are admissible probability measures on $G_{j}$, $j=1,2$.  By Theorem \ref{T:PBamen} and Proposition \ref{P:weakamenmaps}, there are then $G$-maps
\[
(C \times X, \zeta\times\nu) \to (A,\alpha_{0}) \to (X,\nu)
\]
with composition being the natural projection to $X$ and $\alpha_{0}$ being the push-forward of $\zeta\times\nu$ to $A$.  By Theorem \ref{T:PBcont}, $G \actson (C,\zeta)$ is contractive hence by Theorem \ref{T:relcontB}, the projection $(C \times X, \zeta \times \nu) \to (X,\nu)$ is relatively contractive.  Therefore by Theorem \ref{T:relcontcomp}, the maps $(C \times X, \zeta \times \nu) \to (A,\alpha_{0})$ and $(A,\alpha_{0}) \to (X,\nu)$ are both relatively contractive.

Let $(A_{1},\alpha_{1})$ and $(A_{2},\alpha_{2})$ be the ergodic decompositions of $(A,\alpha_{0})$ for $G_{2}$ and $G_{1}$, respectively.  Likewise, let $(X_{j},\nu_{j})$ and $(C_{j},\nu_{j})$ be the decompositions of $(X,\nu)$ and $(C,\zeta)$.  By Propositon \ref{P:ergdecomperg}, $\rpf{(C \times X)}{G_{j}} = C_{j} \times X_{j}$.  Since the ergodic decomposition is a functor (being a special case of quotienting by a random subgroup), there exist $G_{j}$-maps
\[
(C_{j} \times X_{j}, \zeta_{j} \times \nu_{j}) \to (A_{j},\alpha_{j}) \to (X_{j},\nu_{j})
\]
and therefore, as $(C_{j},\zeta_{j})$ is a contractive $G_{j}$-space, the maps $(A_{j},\alpha_{j}) \to (X_{j},\nu_{j})$ are relatively contractive.

Consider the diagram of $G$-maps:
\begin{diagram}
A_{1} \times A_{2} \times X		&\rTo^{p_{X}}						&X\\
\dTo^{p_{A}}	&								&\dTo^{\pi}\\
A_{1} \times A_{2}				&\rTo^{\varphi}		&X_{1} \times X_{2}
\end{diagram}
where $p_{X}$ is the projection to the $X$ coordinate, $p_{A} = p_{A_{1}} \times p_{A_{2}}$ is the diagonal product of the projections to $A_{1}$ and $A_{2}$, $\pi = \pi_{1} \times \pi_{2}$ is the diagonal product of the ergodic decomposition maps of $X$ and $\varphi = \varphi_{1} \times \varphi_{2}$ is the product of the natural maps $A_{j} \to X_{j}$ obtained by the inclusion at the level of $\sigma$-algebras.

By the Intermediate Contractive Factor Theorem (Theorem\ \ref{T:IFT}), since $p_{X}$ and $\varphi$ are relatively contractive and $p_{A}$ and $\pi$ are relatively measure-preserving, $(A,\alpha_{0})$ is $G$-isomorphic to the independent relative joining of $(A_{1},\alpha_{1}) \times (A_{2},\alpha_{2})$ and $(X,\nu)$ over $(X_{1},\nu_{1}) \times (X_{2},\nu_{2})$.  That is, $(A,\alpha_{0})$ is $G$-isomorphic to $(A_{1} \times A_{2} \times X, \alpha)$ where
\[
\alpha = \int_{X_{1} \times X_{2}} D_{\varphi_{1}}(x_{1}) \times D_{\varphi_{2}}(x_{2}) \times D_{\pi}(x_{1},x_{2})~d\nu_{1}\times\nu_{2}(x_{1},x_{2})
\]
as this is the independent relative joining.

Apply the product random subgroups functor to $(X,\nu)$ and obtain $G$-maps
\begin{diagram}
(X,\nu) &\rTo^{q} &PRG(X,\nu) &\rTo^{r} &(X_{1} \times X_{2}, \nu_{1} \times \nu_{2})
\end{diagram}
such that $r \circ q = \pi$ (by the universal property (Theorem \ref{T:irsfactor}) such a map $r$ exists).

Let $B$ be the affine orbital space over $PRG(X,\nu)$ that is the quotient of $A$ by $\stab \circ q$ (constructed in Subsection \ref{ss:affinequot}) and let $z : A_{1} \times A_{2} \times X \to B$ be the corresponding map.  Endow $B$ with the pushforward measure $\beta_{0}$.  Then $(B,\beta_{0})$ is the quotient of $(A,\alpha)$ by the map $a \mapsto \stab(\pi(p_{X}(a)))$.  Then, by the universal property of the quotient spaces (Theorem \ref{T:irsfactor}), the diagram above extends to:
\begin{diagram}
A_{1} \times A_{2} \times X		&\rTo^{p_{X}}			&X\\
\dTo^{q_{A}}					&					&\dTo^{q}\\
PRG(A)						&\rTo^{\psi}			&PRG(X) \\
\dTo							&					&\dTo^{\simeq} \\
B							&\rTo				&PRG(X) \\			
\dTo^{r_{A}}					&					&\dTo^{r}\\
A_{1} \times A_{2}				&\rTo^{\varphi}			&X_{1} \times X_{2}
\end{diagram}
More precisely, the existence of say, the map $r_{A}$ follows from the fact that $(B,\beta_{0})$ is the quotient of $(A,\alpha_{0})$ by $\stab \circ q$ and it holds that
\begin{align*}
\overline{\mathrm{proj}_{1}~\stab(p(a))} \times \overline{\mathrm{proj}_{2}~\stab(p(a))} &\subseteq G_{1} \times \overline{\mathrm{proj}_{2}~\stab(p(a))} \\
&= \overline{G_{1} \times \{ e \} \cdot \stab(p(a))} = \overline{G_{1} \cdot \stab(a)}
\end{align*}
by the orbitality of $A$ over $X$ and therefore the universal property (treating $A_{2}$ as the quotient of $A$ by $a \mapsto \overline{G_{1} \cdot \stab(a)}$) there exists a map $B \to A_{2}$ (and likewise for $A_{1}$).
Since $A$ is orbital over $X$,  $B$ is orbital over $PRG(X)$ by construction (see section \ref{ss:affinequot}).  Since the $G$-action on $PRG(X,\nu)$ is weakly amenable there then exists an invariant section $\tau : PRG(X,\nu) \to B$.  That is, $\tau(g q(x)) = g \cdot \tau(q(x))$.

Define the map $\psi : X \to A_{1} \times A_{2} \times X$ by $\psi(x) = (r_{A}(\tau(q(x))), x)$.  Then
\begin{align*}
\psi(gx) &= (r_{A}(\tau(q(gx))),gx) = (r_{A}(\tau(g q(x))),gx) \\
&= (r_{A}(\beta(g,q(x))\tau(q(x))), gx) = (g r_{A}(\tau(q(x))), gx) = g (r_{A}(\tau(q(x))), x) = g \psi(x)
\end{align*}
which is then, over the isomorphism $A \to A_{1} \times A_{2} \times X$, an invariant section $X \to A$.  As this holds for all affine orbital spaces of $(X,\nu)$, the $G$-action on $(X,\nu)$ is weakly amenable.
\end{proof}

\begin{theorem}\label{T:ergdecweakamen}
Let $G_{1}$ and $G_{2}$ be locally compact second countable groups.  Let $G = G_{1} \times G_{2}$ and let $(X,\nu)$ be an ergodic measure-preserving $G$-space.  Let $(X,\nu) \to (X_{j},\nu_{j})$ denote the spaces of $G_{3-j}$-ergodic components.  Assume that $G_{j} \actson (X_{j},\nu_{j})$ weakly amenably for both $j=1,2$ and that $\stab_{*}\nu_{j}$ are simple invariant random subgroups for $j=1,2$.  
Then one of the following holds:
\begin{itemize}
\item $G \actson (X,\nu)$ essentially free;
\item $G \actson (X,\nu)$ weakly amenably;
\item $\stab_{*}\nu = \delta_{\{ e \}} \times \stab_{*}\nu_{2}$; or
\item $\stab_{*}\nu = \stab_{*}\nu_{1} \times \delta_{\{ e \}}$.
\end{itemize}
\end{theorem}
\begin{proof}
Let $\pi_{1} : (X,\nu) \to (X_{1},\nu_{1})$ be the decomposition into $G_{2}$-ergodic components.  For $\nu_{1}$-almost every $x_{1}$, the $G_{2}$-action on $(\pi_{1}^{-1}(x_{1}), D_{\pi_{1}}(x_{1}))$ is $G_{2}$-ergodic.  Since $\mathrm{proj}_{G_{1}}~\mathrm{stab}_{G}(x)$ is $G_{2}$-invariant, by ergodicity it is constant almost everywhere on almost every ergodic component, that is, for $\nu_{1}$-almost every $x_{1}$ the subgroup $\mathrm{proj}_{G_{1}}~\mathrm{stab}_{G}(x)$ is constant $D_{\pi_{1}}(x_{1})$-almost everywhere.  Therefore the map $s_{1} : X \to S(G_{1})$ by $s_{1}(x) = \overline{\mathrm{proj}_{G_{1}}~\mathrm{stab}_{G}(x)}$ is constant on fibers over $X_{1}$.  By Theorem \ref{T:crazy}, then $(s_{1})_{*}\nu \normal \stab_{*}\nu_{1}$.  

Since $\stab_{*}\nu_{1}$ is simple, for $\nu$-almost every $x \in X$ it holds that $s_{1}(x) = \{ e \}$ or $s_{1}(x) = \stab(\pi_{1}(x))$.  Since the set $\{ x \in X : s_{1}(x) = \{ e \} \}$ is $G$-invariant (because $s_{1}(gx) = (\mathrm{proj}_{G_{1}}~g)s_{1}(x)(\mathrm{proj}_{G_{1}}~g)^{-1}$), by the ergodicity of $G \actson (X,\nu)$ it is either measure zero or measure one.  Therefore $(s_{1})_{*}\nu = \delta_{\{e\}}$ or $(s_{1})_{*}\nu = \stab_{*}\nu_{1}$.  Likewise, $(s_{2})_{*}\nu = \delta_{\{e\}}$ or $(s_{2})_{*}\nu = \stab_{*}\nu_{2}$.

Consider first the case when $(s_{1})_{*}\nu = \delta_{\{ e \}}$.  Then $\stab(x) \subseteq \{ e \} \times s_{2}(x)$ almost everywhere.  If in addition, $(s_{2})_{*}\nu = \delta_{\{ e \}}$ then $\stab(x) = \{ e \} \times \{ e \}$ almost everywhere so $G$ acts essentially freely.  So instead suppose $(s_{2})_{*}\nu = \stab_{*}\nu_{2}$.  Then $\stab(x) = \{ e \} \times H_{x}$ for some $H_{x} < G_{2}$ and $s_{2}(x) = \overline{\mathrm{proj}_{G_{2}}~\stab(x)} = \overline{H_{x}} = H_{x}$ since $H_{x}$ is necessarily closed (as $\stab(x)$ is always closed).  Therefore $H_{x} = \stab(\pi_{2}(x))$ almost everywhere (since $(s_{2})_{*}\nu = \stab_{*}\nu_{2}$) meaning that $\stab(x) = \{ e \} \times \stab(\pi_{2}(x))$ almost everywhere and so $\stab_{*}\nu = \delta_{\{e\}} \times \stab_{*}\nu_{2}$.  The symmetric case follows the same way.

Consider now the case when $(s_{1})_{*}\nu = \stab_{*}\nu_{1}$ and $(s_{2})_{*}\nu = \stab_{*}\nu_{2}$ and consider  the $G$-maps $q : (X,\nu) \to PRG(X,\nu)$ and $r :  PRG(X,\nu) \to (X_{1} \times X_{2}, \nu_{1}\times\nu_{2})$ such that $r \circ q = \pi_{1} \times \pi_{2}$.  By construction, for almost every $x \in X$ it holds that $s_{1}(x) \times s_{2}(x) \subseteq \stab(q(x))$.  Since $(s_{1})_{*}\nu = \stab_{*}\nu_{1}$ and $(s_{2})_{*}\nu = \stab_{*}\nu_{2}$, then for almost every $x \in X$, we have that $\stab(\pi_{1}(x)) \times \stab(\pi_{2}(x)) \subseteq \stab(q(x))$.  But as $PRG(X,\nu)$ is an extension of $(X_{1} \times X_{2}, \nu_{1} \times \nu_{2})$, this means that $PRG(X,\nu)$ is orbital over $(X_{1} \times X_{2}, \nu_{1} \times \nu_{2})$.  Since each $G_{j} \actson (X_{j},\nu_{j})$ weakly amenably, $G_{1} \times G_{2} \actson (X_{1} \times X_{2}, \nu_{1} \times \nu_{2})$ weakly amenably by Proposition \ref{P:weakamenprod}.  Then by Proposition \ref{P:weakamen}, $G \actson PRG(X,\nu)$ weakly amenably so by Theorem \ref{T:PRGweakamen}, $G \actson (X,\nu)$ weakly amenably.
\end{proof}

\begin{theorem}\label{T:weakamendense}
Let $G_{1}$ and $G_{2}$ be locally compact second countable groups.  Let $G = G_{1} \times G_{2}$ and let $(X,\nu)$ be an ergodic measure-preserving $G$-space.  If $\mathrm{proj}_{G_{2}}~\stab(x)$ is dense in $G_{j}$ almost everywhere for both $j=1,2$ then $G \actson (X,\nu)$ weakly amenably.
\end{theorem}
\begin{proof}
When both projections are dense almost everywhere, $PRG(X,\nu)$ is the quotient by $G_{1} \times G_{2}$ hence $G$ acts trivially on $PRG(X,\nu)$.  As $(X,\nu)$ is ergodic, so is $PRG(X,\nu)$ and therefore $PRG(X,\nu)$ is the trivial (one-point) space.  Clearly every group acts weakly amenably on the trivial space, so the conclusion follows by Theorem \ref{T:PRGweakamen}.
\end{proof}

\subsection{Irreducible Actions}

%
%

\begin{theorem}\label{T:bsnew}
Let $G_{1}$ and $G_{2}$ be locally compact second countable groups.  Set $G = G_{1} \times G_{2}$ and let $(X,\nu)$ be a measure-preserving $G$-space such that each $G_{j} \actson (X,\nu)$ ergodically for both $j=1,2$.  Then there exist normal subgroups $N_{1} \normal G_{1}$ and $N_{2} \normal G_{2}$ such that, setting $N = N_{1} \times N_{2}$, it holds that $G / N$ acts essentially freely on the space of $N$-ergodic components $\rpf{(X,\nu)}{N}$ and $N$ acts weakly amenably on almost every $N$-ergodic component.
\end{theorem}
\begin{proof}
Consider the functions $s_{j}(x) = \overline{\mathrm{proj}_{G_{j}}~\stab(x)}$ for each $j=1,2$.  Each $s_{j}$ is $G_{3-j}$-invariant so by ergodicity is constant almost surely.  Set $N_{j} = s_{j}(x)$.  Since $N_{j} = s_{j}(g_{j}x) = g_{j}s_{j}(x)g_{j}^{-1} = g_{j}N_{j}g_{j}^{-1}$ for any $g_{j} \in G_{j}$, we have that $N_{j} \normal G_{j}$.

Let $(Y,\eta)$ be the space of $N_{1} \times N_{2}$-ergodic components and let $\pi : (X,\nu) \to (Y,\eta)$ be the $G$-map.  Then $\stab(\pi(x)) = \stab(x) \cdot N_{1} \times N_{2} = N_{1} \times N_{2}$ almost everywhere so $G / (N_{1} \times N_{2})$ acts essentially freely on $(Y,\eta)$.

Now for almost every $x \in X$, we have that $\stab(x) < N_{1} \times N_{2}$ and that $\mathrm{proj}_{N_{j}}~\stab(x)$ is dense in $N_{j}$ for both $j=1,2$.  Since $N_{1} \times N_{2}$ acts ergodically on almost every ergodic component $y \in Y$, by Theorem \ref{T:weakamendense}, we have that $N_{1} \times N_{2}$ acts weakly amenably on $(\pi^{-1}(y),D_{\pi}(y))$.
\end{proof}

\begin{corollary}\label{C:bs1}
Let $G_{1}$ and $G_{2}$ be simple locally compact second countable groups.  Set $G = G_{1} \times G_{2}$ and let $(X,\nu)$ be a measure-preserving $G$-space such that each $G_{j} \actson (X,\nu)$ ergodically for both $j=1,2$.  Then either $G \actson (X,\nu)$ is essentially free or $G \actson (X,\nu)$ weakly amenably.
\end{corollary}
\begin{proof}
Assume the action is not essentially free.
By Theorem \ref{T:bsnew}, as $G_{1}$ and $G_{2}$ are simple, there exists $N \normal G$ of the form $N = \{ e \}$, $N = G_{1} \times \{ e \}$, $N = \{ e \} \times G_{2}$ or $N = G$ such that $G / N$ acts essentially freely on the space of $N$-ergodic components and $N$ acts weakly amenably on almost every ergodic component.  Since the action is not essentially free, $N$ is not the trivial group.  Since both $G_{j}$ act ergodically, $G / N$ acts essentially freely on the trivial space meaning that $N = G$.  Therefore $G$ acts weakly amenably on the only $N$-ergodic component which is $(X,\nu)$ itself.
\end{proof}

\subsection[Weakly Amenable Actions and Property (T)]{Weakly Amenable Actions and Property $(T)$}

We now show how the presence of property $(T)$ in only one of the two groups in the product is enough to rule out weakly amenable actions that are not essentially transitive.  We begin with some basic facts about such actions and then employ resolutions to rule them out.

\begin{proposition}\label{P:aivecs}
Let $G$ be a locally compact second countable group and $(X,\nu)$ an ergodic measure-preserving $G$-space such that $G \actson (X,\nu)$ weakly amenably and not essentially transitively.  Then there exists a sequence of almost invariant (but not invariant) vectors in $L_{0}^{2}(X,\nu)$, the subspace of $L^{2}(X,\nu)$ orthogonal to the constants.
\end{proposition}
\begin{proof}
Since $G \actson (X,\nu)$ is weakly amenable, by the Connes-Feldman-Weiss theorem (Theorem \ref{T:cfw}), it is orbit equivalent to an action of $\mathbb{Z}$ or $\mathbb{R}$ on a probability space $(Y,\eta)$.  Krasa \cite{krasa} has shown that if a group $H$ is amenable as a discrete group (which both $\mathbb{Z}$ and $\mathbb{R}$ are) and there is a unique invariant mean on $L^{\infty}(Y,\eta)$ then there exists a positive measure orbit (when $H$ is countable, this is due to del Junco and Rosenblatt \cite{djr}).  Clearly the uniqueness of an invariant mean and the existence of a positive measure orbit are characteristics of the equivalence relation, so we conclude that if $G \actson (X,\nu)$ has a unique invariant mean then the action is essentially transitive (using ergodicity, the positive measure orbit is of full measure).

Rosenblatt \cite{rosenblatt} (Theorem 1.4) showed that if $G \actson (X,\nu)$ admits more than one invariant mean then there exists a positive measure set $E \subseteq X$ and an approximately invariant net $(A_{\alpha})$ of measurable sets such that $A_{\alpha} \subseteq X \setminus E$ for all $\alpha$.  Approximately invariant means that for all $g \in G$ it holds that
$\lim_{\alpha} (\nu(A_{\alpha}))^{-1} \nu(g^{-1}A_{\alpha} \symdiff A_{\alpha}) = 0$.

Define the functions $f_{\alpha} = \bbone_{A_{\alpha}} - \nu(A_{\alpha}) \in L_{0}^{2}(X,\nu)$.  Then $\| f_{\alpha} \|_{2}^{2} = \nu(A_{\alpha})(1 - \nu(A_{\alpha}))$.  For $g \in G$,
$\int \big{|}g\cdot f_{\alpha} - f_{\alpha}\big{|}^{2}~d\nu = \nu(g^{-1}A_{\alpha} \symdiff A_{\alpha})$.
Let $q_{\alpha} = \| f_{\alpha} \|_{2}^{-1}f_{\alpha}$.  Then for $g \in G$,
\[
\| g \cdot q_{\alpha} - q_{\alpha} \|_{2}^{2} = \frac{\nu(g^{-1}A_{\alpha} \symdiff A_{\alpha})}{\nu(A_{\alpha})(1 - \nu(A_{\alpha}))} \leq \frac{\nu(g^{-1}A_{\alpha} \symdiff A_{\alpha})}{\nu(A_{\alpha})} \frac{1}{\nu(E)} \to 0
\]
since $\nu(E) > 0$ and $A_{\alpha} \subseteq X \setminus E$.  So the $\{ q_{\alpha} \}$ are almost invariant (norm one) vectors.

The reader is referred to Hjorth and Kechris \cite{hjorthkechris} Appendix A for a detailed account of the theory of nonuniqueness of invariant means for equivalence relations arising from group actions.
\end{proof}

\begin{proposition}\label{P:resolutionsprod}
Let $G = G_{1} \times G_{2}$ be a product of two locally compact second countable groups such that $G_{2}$ has property $(T)$.
Let $\pi : G \to \mathcal{U}(\mathcal{H})$ be a unitary representation of $G$ on a Hilbert space that has almost invariant vectors that are not invariant.  Then $\pi$ restricted to the space of $G_{2}$-invariant, but not $G_{1}$-invariant, vectors has almost invariant vectors (as $G$-, hence as a $G_{1}$-) representation.
\end{proposition}
\begin{proof}
Without loss of generality, we may assume that $\mathcal{H}$ has no $G$-invariant vectors by simply restricting $\pi$ to the complement of the invariant vectors.
By Proposition \ref{P:resolutions}, the projection map $\mathrm{proj}_{1} : G \to G_{1}$ is a resolution since $G_{2}$ has property $(T)$.
The space of $(G_{1},\mathrm{proj}_{1})$-points in $\mathcal{H}$, denoted $\mathcal{H}^{G_{1}}$, is a closed $G$-invariant subspace (Proposition \ref{P:closedQpoints}).  
Since $\pi$ has almost invariant vectors and $\mathrm{proj}_{1}$ is a resolution, $\pi^{G_{1}} : G_{1} \to \mathcal{H}^{G_{1}}$ also has almost invariant vectors.  

Observe that if $v \in \mathcal{H}$ is $G_{2}$-invariant and if $\{ g_{n} \}$ is any sequence in $G$ such that $\mathrm{proj}_{1}~g_{n}$ converges to some $g_{\infty} \in G_{1}$ then $\pi(g_{n})v = \pi(\mathrm{proj}_{1}~g_{n})v \to \pi(g_{\infty})v$ since $\pi$ is a continuous.  Therefore the space of $G_{2}$-invariant vectors is contained in $\mathcal{H}^{G_{1}}$.
Suppose now that for some $v \in \mathcal{H}$ there exists $h \in G_{2}$ such that $\pi(h)v \ne v$.    Consider the sequence $\{ g_{n} \}$ in $G$ given by $g_{n} = e$ for $n$ even and $g_{n} = h$ for $n$ odd.  Then $\mathrm{proj}_{1}~g_{n} = e$ for all $n$ which converges in $G_{1}$ but $\pi(g_{n})v = v$ for $n$ even and $\pi(g_{n})v = \pi(h)v \ne v$ for $n$ odd.  Therefore $v \notin \mathcal{H}^{G_{1}}$.  So we conclude that the space of $G_{1}$-points is precisely the space of $G_{2}$-invariant vectors.
\end{proof}

\begin{theorem}\label{T:resolutions}
Let $G = G_{1} \times G_{2}$ be a product of two locally compact second countable groups such that $G_{2}$ has property $(T)$.  Let $(X,\nu)$ be an ergodic measure-preserving $G$-space such that $G \actson (X,\nu)$ weakly amenably and not essentially transitively.  Let $\mathcal{H}$ be the subspace of $L^{2}(X,\nu)$ consisting of the $G_{2}$-invariant functions that are not $G$-invariant.  Then there exists a sequence of almost invariant vectors in $\mathcal{H}$.
\end{theorem}
\begin{proof}
By Proposition \ref{P:aivecs} there is a sequence of almost invariant vectors in $L^{2}(X,\nu)$ that are not invariant.  Since $G_{2}$ has property $(T)$,
by Proposition \ref{P:resolutionsprod}, there is a sequence of almost invariant vectors in the space of $G_{2}$-invariant but not $G_{1}$-invariant functions.
\end{proof}

\subsection[Actions of Products of Groups, at least one with Property (T)]{Actions of Products of Groups, at least one with Property $(T)$}


\begin{corollary}\label{C:bsnew2}
Let $G_{1}$ be a simple locally compact second countable group with property $(T)$ and let $G_{2}$ be any locally compact second countable group.  Set $G = G_{1} \times G_{2}$ and let $(X,\nu)$ be a measure-preserving $G$-space such that each $G_{j} \actson (X,\nu)$ ergodically for both $j=1,2$.  Then either the kernel of the $G$-action is of the form $N = \{ e \} \times N_{2}$ for some $N_{2} \normal G_{2}$ and the $G/N$-action on $(X,\nu)$ is essentially free or else $G \actson (X,\nu)$ is essentially transitive.
\end{corollary}
\begin{proof}
By Theorem \ref{T:bsnew}, there exists $N_{1} \normal G_{1}$ and $N_{2} \normal G_{2}$ such that, setting $N = N_{1} \times N_{2}$, we have that $G / N$ acts essentially freely on the space of $N$-ergodic components and $N$ acts weakly amenably on almost every ergodic component.  Since $G_{1}$ is simple, either $N_{1}$ is trivial or $N_{1} = G_{1}$.

Consider the case when $N_{1}$ is trivial.  Then $G / N = G_{1} \times (G_{2} / N_{2})$ acts essentially freely on the space of $N$-ergodic components.  Let $\pi : (X,\nu) \to (Y,\eta)$ be the $G$-map to the space of $N$-ergodic components.  Since $G / N \actson (X,\nu)$ essentially freely, $\stab(x) < N = \{ e \} \times N_{2}$ almost surely.  Therefore $\{ e \} \times \overline{\mathrm{proj}_{G_{2}}~\stab(x)} = \stab(x)$.  But $\{ e \} \times \overline{\mathrm{proj}_{G_{2}}~\stab(x)}$ is $G_{1}$-invariant hence constant by ergodicity.  Therefore $\stab(x)$ is constant almost surely meaning that $\stab(x) = \ker(G \actson (X,\nu))$ which is of the form $\{ e \} \times N_{2}^{\prime}$ for some $N_{2}^{\prime} \normal G_{2}$.

Now consider the case when $N_{1} = G_{1}$.  Since $G_{1}$ acts ergodically, the space of $N$-ergodic components is trivial and $(X,\nu)$ is the only $N$-ergodic component.  As $G / N$ acts essentially freely on the space of $N$-ergodic components, which is trivial, then $N = G$ so $N_{2} = G_{2}$.
So $N = G$ acts weakly amenably on $(X,\nu)$.  Suppose that $G$ does not act essentially transitively on $(X,\nu)$.
Then, by Theorem \ref{T:resolutions}, the space of $G_{1}$-invariant but not $G_{2}$-invariant functions in $L^{2}(X,\nu)$ has almost invariant vectors.  But $G_{1}$ acts ergodically on $(X,\nu)$ so the only $G_{1}$-invariant functions are the constants which are themselves $G_{2}$-invariant.  This contradiction means that $G$ acts essentially transitively.
\end{proof}

\begin{corollary}\label{C:bsnew4}
Let $G$ be a product of at least two simple noncompact locally compact second countable groups, at least one with property $(T)$ and let $(X,\nu)$ be a measure-preserving $G$-space that is ergodic for each simple factor of $G$.  Then $G \actson (X,\nu)$ is either essentially free or essentially transitive.
\end{corollary}
\begin{proof}
Write $G = G_{1} \times H$ where $G_{1}$ is simple and has property $(T)$ and $H = \prod_{j} H_{j}$ is a product of simple groups.
By Corollary \ref{C:bsnew2}, either the kernel of the $G$-action is of the form $\{ e \} \times N$ for some $N \normal H$ and the $G/N$-action on $(X,\nu)$ is essentially free or else $G \actson (X,\nu)$ is essentially transitive.  Therefore, the only thing to check is that if the kernel is of the form $\{ e \} \times N$ then $N$ is necessarily trivial (making the $G$-action on $(X,\nu)$ essentially free).

Suppose then that the kernel is $\{ e \} \times N$ for $N$ nontrivial.
First note that if $H$ is a single simple group then $N = H$ in which case $H$ acts ergodically and trivially on $(X,\nu)$ making it trivial.  So we may assume that $H$ has at least two factors.  Since $N$ is nontrivial, it has nontrivial projection to some simple factor of $H$.  Without loss of generality, we assume that $N$ projects nontrivially to $H_{1}$.

Let $n \in N$ such that $\mathrm{proj}_{H_{1}}~n \ne e$ and let $h_{1} \in H_{1}$.  Write $n = (n_{1},n_{2})$ where $n_{1} \in H_{1}$ and $n_{2} \in \prod_{j \ne 1} H_{j}$.  Then $h_{1}nh_{1}^{-1}n^{-1} = (h_{1}n_{1}h_{1}^{-1}n_{1}^{-1},e) \in H_{1}$ and also, as $N$ is normal, $h_{1}nh_{1}^{-1}n^{-1} \in N$.  Now, if $n_{1} \notin Z(H_{1})$ then there exists $h_{1} \in H_{1}$ such that $h_{1}n_{1}h_{1}^{-1}n_{1}^{-1}$ is nontrivial.  Since $H_{1}$ is simple and noncompact, $Z(H_{1})$ is trivial.  Therefore, there exists $h_{1}nh_{1}^{-1}n^{-1} \in N \cap H_{1}$ that is nontrivial.  Since $N \cap H_{1} \normal H_{1}$, then $N \cap H_{1} = H_{1}$.  Hence $H_{1}$ acts trivially on $(X,\nu)$ and ergodically meaning that $(X,\nu)$ is trivial.
\end{proof}

\subsection[Actions of Products of Property (T) Groups]{Actions of Products of Property $(T)$ Groups}

\begin{corollary}\label{C:actprodT}
Let $G_{1}$ and $G_{2}$ be locally compact second countable groups with property $(T)$.  Set $G = G_{1} \times G_{2}$ and let $(X,\nu)$ be an ergodic measure-preserving $G$-space.  Assume that there exists simple closed subgroups $H_{j} < G_{j}$ such that the space of $G_{3-j}$-ergodic components is isomorphic to $(G_{j}/H_{j},\mathrm{Haar})$ for $j=1,2$  and such that any nontrivial normal subgroup of $H_{j}$ has finite index in $H_{j}$.  Then either at least one $G_{j} \actson (X,\nu)$ essentially free or $G \actson (X,\nu)$ is essentially transitive.
\end{corollary}
\begin{proof}
First observe that since $(X,\nu)$ is measure-preserving, so is $(G_{j}/H_{j},\mathrm{Haar})$ and therefore $H_{j}$ has finite covolume in $G_{j}$.  Recall that the map $s_{1}(x) = \overline{\mathrm{proj}_{G_{1}}~\stab(x)}$ has the property that $s_{1}(x) \normal \stab(\pi_{1}(x)) = g_{1}H_{1}g_{1}^{-1}$ where $\pi_{1}$ is the ergodic decomposition map.  Since $H_{1}$ has the property that every normal subgroup is either trivial or of finite index, $s_{1}(x)$ is either trivial or has finite index in some conjugate of $H_{1}$ almost everywhere.  Since $G_{1}$ acts ergodically on $(G_{1}/H_{1},\mathrm{Haar})$, either $s_{1}(x) = \{ e \}$ almost everywhere or $s_{1}(x)$ has finite index in a conjugate of $H_{1}$ almost everywhere.  The case when $s_{1}(x) = \{ e \}$ corresponds to $G_{1}$ acting essentially freely on $(X,\nu)$.  As the same reasoning holds for $s_{2}(x) = \overline{\mathrm{proj}_{G_{2}}~\stab(x)}$, we may assume from here on that $\varphi(x) = s_{1}(x) \times s_{2}(x)$ is of the form $\varphi(x) = g K g^{-1}$ where $K = K_{1} \times K_{2}$ with $K_{j}$ of finite index in $H_{j}$.

Since $K_{j}$ has finite index in $H_{j}$, $K_{j}$ also has finite covolume in $G_{j}$.  Let $\psi : (X,\nu) \to PRG(X,\nu)$ be the product random subgroups functor map.  Then $\stab(\psi(x))$ has finite covolume in $G$ since $\varphi(x) \subseteq \stab(\psi(x))$.  We may therefore define an invariant mean $m_{x}$ on $PRG(X,\nu)$ using the Haar measures on $G / \stab(\psi(x))$.  We conclude that $PRG(X,\nu)$ is then weakly amenable.  Hence $(X,\nu)$ is weakly amenable by Theorem \ref{T:PRGweakamen}.
Since $G$ has property $(T)$ then the action is essentially transitive.
\end{proof}

\section{Actions of Lattices in Product Groups}

Having completed our study of the actions of products of two groups, we now turn to the study of actions of irreducible lattices in such products.  As with the previous section, we first state and prove the results for the product of two simple groups and in the following section generalize to the case of arbitrary products.  Unlike in the case of actions of the products of groups, a full classification of the stabilizers of actions of lattices is only possible under the additional assumption that the groups have the Howe-Moore property.

\subsection{The Projected Action}\label{S:projact}

In addition to the product random subgroups functor (which we apply to the induced action), we need a similar object that can be obtained directly from the action of a lattice.  The projected action, which we define presently, is in essence the same as the product random subgroups functor, with the caveat that it is not, strictly speaking, a quotient of the action of the lattice.  The idea in the proofs in this section is to make use of information gained from studying the projected action to conclude facts about the product random subgroups functor applied to the induced action.

\begin{theorem}\label{T:81}
Let $G_{1}$ and $G_{2}$ be noncompact locally compact second countable groups and set $G = G_{1} \times G_{2}$.  Let $\Gamma < G$ be an irreducible lattice and let $(X,\nu)$ be an ergodic measure-preserving $\Gamma$-space.  Consider the maps $s_{j} : X \to S(G_{j})$ given by $s_{j}(x) = \overline{\mathrm{proj}_{G_{j}}~\stab_{\Gamma}(x)}$ for $j = 1,2$.  Then each $(s_{j})_{*}\nu$ is an invariant random subgroup of $G_{j}$.
\end{theorem}
\begin{proof}
Clearly $(s_{j})_{*}\nu \in P(S(G_{j}))$ so the only thing to check is that it is conjugation-invariant.  For $\gamma \in \Gamma$,
\begin{align*}
(\mathrm{proj}_{G_{j}}~\gamma) s_{j}(x) (\mathrm{proj}_{G_{j}}~\gamma)^{-1}
&= (\mathrm{proj}_{G_{j}}~\gamma) \overline{\mathrm{proj}_{G_{j}}~\stab(x)} (\mathrm{proj}_{G_{j}}~\gamma)^{-1} \\
&= \overline{\mathrm{proj}_{G_{j}}~\gamma\stab(x)\gamma^{-1}}
= \overline{\mathrm{proj}_{G_{j}}~\stab(\gamma x)}
\end{align*}
so $(s_{j})_{*}\nu$ is invariant under conjugation by $\mathrm{proj}_{G_{j}}~\Gamma$ since $\nu$ is $\Gamma$-invariant.  Since $\Gamma$ is irreducible, $\mathrm{proj}_{G_{j}}~\Gamma$ is dense in $G_{j}$ and as the action by conjugation is continuous this means that $(s_{j})_{*}\nu$ is conjugation invariant.
\end{proof}

\begin{definition}
Let $G_{1}$ and $G_{2}$ be locally compact second countable groups and set $G = G_{1} \times G_{2}$.  Let $\Gamma < G$ be an irreducible lattice and let $(X,\nu)$ be a measure-preserving $\Gamma$-space.  The \textbf{projected action} of $\Gamma \actson (X,\nu)$ is the spaces $G_{j} \actson (Y_{j},\eta_{j})$ for $j=1,2$ that the ergodic $(s_{j})_{*}$-nonfree actions of $G_{j}$.
\end{definition}

\subsection{Actions of Lattices in Products of Howe-Moore Groups}

\begin{theorem}\label{T:stuff}
Let $G_{1}$ and $G_{2}$ be simple nondiscrete noncompact locally compact second countable groups with the Howe-Moore property and set $G = G_{1} \times G_{2}$.  Let $\Gamma < G$ be an irreducible lattice and let $(X,\nu)$ be an ergodic measure-preserving $\Gamma$-space.  Then one of the following holds:
\begin{itemize}
\item $\Gamma \actson (X,\nu)$ is essentially free;
\item $\Gamma \actson (X,\nu)$ is weakly amenable;
\item $\stab_{*}\nu$ is supported on the finite index subgroups of $\Gamma$;
\item $\stab_{*}\nu$ is supported on the torsion elements of $\Gamma$; or
\item one $G_{j}$ is totally disconnected and acts ergodically and essentially freely on the induced space $G \times_{\Gamma} X$ and the other $G_{3-j}$ does not act ergodically on the induced space.
\end{itemize}
\end{theorem}
\begin{proof}
Assume that $\Gamma \actson (X,\nu)$ is not essentially free.
For $\gamma \in \Gamma$, let $E_{\gamma} = \{ x \in X : \gamma x = x \}$.  Let $(Y_{j},\eta_{j})$ be the projected action to $G_{1}$ and $G_{2}$.  Since the $\Gamma$-action is not essentially free, there exists $\gamma \in \Gamma \setminus \{ e \}$ such that $\nu(E_{\gamma}) > 0$.  Let $L = \{ \gamma \in \Gamma \setminus \{ e \} : \nu(E_{\gamma}) > 0 \}$.  So $L \ne \emptyset$.

Consider first the case when there exists $\gamma \in L$ such that $\langle \mathrm{proj}_{G_{1}}~\gamma \rangle$ is unbounded in $G_{1}$ (that is, $\overline{\langle \mathrm{proj}_{G_{1}}~\gamma \rangle}$ is noncompact).  Since $\gamma \in \stab(x)$ on the positive measure set $E_{\gamma}$, there exists $F_{\gamma} \subseteq Y_{1}$ with $\eta_{1}(F_{1}) > 0$ such that $\mathrm{proj}_{G_{1}}~\gamma \in \stab(y)$ for all $y \in F_{\gamma}$ (because $\stab_{*}\eta_{1} = (s_{1})_{*}\nu$).  Then $\overline{\langle \mathrm{proj}_{G_{1}}~\gamma \rangle} \subseteq \stab(y)$ for all $y \in F_{\gamma}$.  Since $G_{1}$ has the Howe-Moore property and $(Y_{1},\eta_{1})$ is an ergodic measure-preserving $G_{1}$-space, by Theorem \ref{T:HMmixing}, $G_{1} \actson (Y_{1},\eta_{1})$ is mixing.  As $\langle \mathrm{proj}_{G_{1}}~\gamma \rangle$ is unbounded, $\lim_{n} \eta_{1}((\mathrm{proj}_{G_{1}}~\gamma)^{n} F_{\gamma} \cap F_{\gamma}) = (\eta_{1}(F_{\gamma}))^{2}$.  But $\mathrm{proj}_{G_{1}}~\gamma$ fixes $F_{\gamma}$ so $\eta_{1}((\mathrm{proj}_{G_{1}}~\gamma)^{n} F_{\gamma} \cap F_{\gamma}) = \eta_{1}(F_{\gamma})$.  Therefore $\eta_{1}(F_{\gamma}) = (\eta_{1}(F_{\gamma}))^{2}$ meaning that $\eta_{1}(F_{\gamma}) = 1$.  Then $\mathrm{proj}_{G_{1}}~\gamma \in \ker(G_{1} \actson Y_{1})$ and as $G_{1}$ is simple then $(Y_{1},\eta_{1})$ is the trivial space since it is $G_{1}$-ergodic.  So $\overline{\mathrm{proj}_{G_{1}}~\stab_{\Gamma}(x)} = G_{1}$ almost everywhere.

So we conclude that if there exists $\gamma_{j} \in \Gamma$ for both $j=1,2$ such that $\nu(E_{\gamma_{j}}) > 0$ and $\langle \mathrm{proj}_{G_{j}}~\gamma_{j} \rangle$ is unbounded in $G_{j}$ then $\overline{\mathrm{proj}_{G_{j}}~\stab_{\Gamma}(x)} = G_{j}$ almost everywhere.  Let $G \times_{\Gamma} X = (F \times X, m \times \nu)$ be the induced action to $G$ from $\Gamma \actson (X,\nu)$ (here $F$ is a fundamental domain for $\Gamma$ with cocycle $\alpha : G \times F \to \Gamma$ such that $gf\alpha(g,f) \in F$ and the action is given by $g \cdot (f,x) = (gf\alpha(g,f), \alpha(g,f)^{-1}x)$).  Let $(Z,\zeta) = PRG(G \times_{\Gamma} X)$ be the product random subgroups space for the induced action and let $\psi : (F \times X, m \times \nu) \to (Z,\zeta)$ be the defining map.  Then
\[
\stab_{G}(f,x) = f \stab_{\Gamma}(x) f^{-1}
\]
for almost every $(f,x) \in F \times X$.  So
\[
\mathrm{proj}_{G_{j}}~\stab_{G}(f,x) = (\mathrm{proj}_{G_{j}}~f) \mathrm{proj}_{G_{j}}~\stab_{\Gamma}(x) (\mathrm{proj}_{G_{j}}~f)^{-1}
\]
is dense almost everywhere for both $j = 1,2$.  Then by Theorem \ref{T:weakamendense}, $G \actson G \times_{\Gamma} X$ weakly amenably so by Proposition \ref{P:weakamenlattice}, $\Gamma \actson (X,\nu)$ weakly amenably.

Consider now the case when for every $\gamma \in \Gamma$ such that $\nu(E_{\gamma}) > 0$, it holds that $\langle \mathrm{proj}_{G_{j}}~\gamma \rangle$ is bounded in $G_{j}$ for both $j = 1,2$.  Then $\langle \gamma \rangle \subseteq \overline{\langle \mathrm{proj}_{G_{1}}~\gamma \rangle} \times \overline{\langle \mathrm{proj}_{G_{2}}~\gamma \rangle}$ which is a compact subgroup.  But $\langle \gamma \rangle$ is discrete and therefore finite meaning that $\gamma$ is torsion.
Therefore, in this case, $\stab_{*}\nu$ is supported on the torsion elements since every non-torsion $\gamma \in \Gamma$ has $\nu(E_{\gamma}) = 0$.

So we are left with the case when there exists $\gamma_{1} \in \Gamma$ with $\nu(E_{\gamma_{1}}) > 0$ and $\langle \mathrm{proj}_{G_{1}}~\gamma \rangle$ unbounded in $G_{1}$ but that for every $\gamma \in \Gamma$ with $\nu(E_{\gamma}) > 0$, it holds that $\langle \mathrm{proj}_{G_{2}}~\gamma \rangle$ is bounded in $G_{2}$ (or the reverse situation, which is the same by symmetry).  Then, as above, $\mathrm{proj}_{G_{1}}~\stab_{\Gamma}(x)$ is dense almost everywhere.  Let $\psi : G \times_{\Gamma} X \to PRG(G \times_{\Gamma} X)$ be the map defining the product random subgroups space of the induced action.  Let $(Z_{j},\zeta_{j})$ be the space of $G_{3-j}$-ergodic components of $G \times_{\Gamma} X$ for $j = 1,2$.  Let $\tau : PRG(G \times_{\Gamma} X) \to (Z_{1} \times Z_{2}, \zeta_{1} \times \zeta_{2})$ be the map from Proposition \ref{P:something} such that $\tau \circ \psi$ is the ergodic decomposition.  Then for $m \times \nu$-almost every $(f,x) \in F \times X$,
\[
\stab(\tau \circ \psi(f,x)) \supseteq \stab(\psi(f,x)) \supseteq G_{1} \times (\mathrm{proj}_{G_{2}}~f)\overline{\mathrm{proj}_{G_{2}}~\stab_{\Gamma}(x)}(\mathrm{proj}_{G_{2}}~f)^{-1}.
\]
Therefore $\stab_{G_{1}}(z_{1}) = G_{1}$ for $\zeta_{1}$-almost every $z_{1} \in Z_{1}$.  Since $G_{1}$ acts ergodically on $(Z_{1},\zeta_{1})$, the space is trivial.  As $(Z_{1},\zeta_{1})$ is the space of $G_{2}$-ergodic components this means that $G_{2}$ acts ergodically on $G \times_{\Gamma} X$.  Furthermore, $\stab_{G_{2}}(f,x) = (\{ e \} \times G_{2}) \cap f \stab_{\Gamma}(x) f^{-1}$.  Since $\Gamma$ is irreducible, $\Gamma \cap (\{ e \} \times G_{2}) = \{ e \}$ by Proposition \ref{P:irrlattint}.  So if $g \in \stab_{G_{2}}(f,x)$ then $f^{-1}gf \in \{ e \} \times G_{2} \cap \stab_{\Gamma}(x) = \{ e \}$ so $g = e$.  Hence $G_{2} \actson G \times_{\Gamma} X$ ergodically and essentially freely.

We now show that we are in the fifth case.  Suppose that $G_{1}$ also acts ergodically on the induced space $G \times_{\Gamma} X$.  Then both $G_{j}$ act ergodically on $G \times_{\Gamma} X$ so by Corollary \ref{C:bs1}, $G \actson G \times_{\Gamma} X$ is either essentially free or weakly amenable.  Therefore $\Gamma \actson (X,\nu)$ is either essentially free or weakly amenable by Proposition \ref{P:weakamenlattice}.

Suppose now that $G_{2}$ is connected (since it is simple, if it is not connected then it is totally disconnected).  Then $G_{2}$ is a simple real Lie group since it has the Howe-Moore property (Theorem \ref{T:connHMLie}).  By Theorem \ref{T:countableessfree}, since $\mathrm{proj}_{G_{2}}~\Gamma$ is a countable subgroup of $G_{2}$, either $(Y_{2},\eta_{2})$ is the trivial space or else $\mathrm{proj}_{G_{2}}~\Gamma$ acts essentially freely on $(Y_{2},\eta_{2})$.

If $(Y_{2},\eta_{2})$ is trivial then $\stab_{\Gamma}(x)$ projects densely to $G_{2}$ almost everywhere.  As we already have that it projects densely almost everywhere to $G_{1}$ this means that $G \actson G \times_{\Gamma} X$ weakly amenably by Theorem \ref{T:weakamendense} and so by Proposition \ref{P:weakamenlattice}, $\Gamma \actson (X,\nu)$ weakly amenably.  So we are left with the case when $\mathrm{proj}_{G_{2}}~\Gamma$ acts essentially freely on $(Y_{2},\eta_{2})$.  But this means exactly that $\mathrm{proj}_{G_{2}}~\stab_{\Gamma}(x) = \{ e \}$ almost everywhere.  So if $\nu(E_{\gamma}) > 0$ then $\gamma \in (G_{1} \times \{ e \}) \cap \Gamma = \{ e \}$ since $\Gamma$ is irreducible (Proposition \ref{P:irrlattint}) which means that $\Gamma \actson (X,\nu)$ essentially freely.  So if the $\Gamma$-action is not essentially free then $G_{2}$ must be totally disconnected.
\end{proof}

\begin{remark}
The fifth possibility in the previous theorem, that one of the groups be totally disconnected with certain other properties, is exactly the case that the results of \cite{CP12} handle.  In this sense, our work here complements perfectly that of \cite{CP12}.
\end{remark}

\subsection[Actions of Lattices in Products of Howe-Moore Groups, at least one with property (T)]{Actions of Lattices in Products of Howe-Moore Groups, at least one with property $(T)$}

\begin{corollary}\label{C:lattices}
Let $G_{1}$ and $G_{2}$ be simple nondiscrete noncompact locally compact second countable groups with the Howe-Moore property such that at least one $G_{j}$ has property $(T)$.  Set $G = G_{1} \times G_{2}$ and let $\Gamma < G$ be an irreducible lattice.  Let $(X,\nu)$ be an ergodic measure-preserving $\Gamma$-space.  Then one of the following holds:
\begin{itemize}
\item $\Gamma \actson (X,\nu)$ is essentially free;
\item $\stab_{*}\nu$ is supported on the finite index subgroups of $\Gamma$;
\item $\stab_{*}\nu$ is supported on the torsion elements of $\Gamma$;
\item one $G_{j}$ is totally disconnected and acts ergodically and essentially freely on the induced space $G \times_{\Gamma} X$ and the other $G_{3-j}$ does not act ergodically on the induced space; or
\item one $G_{j}$ does not have property $(T)$, is totally disconnected, and there is a nontrivial ergodic $G_{j}$-space $(Y,\eta)$ that is a $\Gamma$-quotient of $(X,\nu)$ and such that $G_{j} \actson (Y,\eta)$ is not essentially transitive and $\Gamma \actson (Y,\eta)$ is weakly amenable.
\end{itemize}
\end{corollary}
\begin{proof}
Assume none of the first four possibilities hold.
By the previous theorem, $\Gamma \actson (X,\nu)$ weakly amenably.  Note that if $\Gamma \actson (X,\nu)$ essentially transitively then $(X,\nu)$ is necessarily a finite atomic space by ergodicity, in which case the stabilizers are finite index subgroups of $\Gamma$.  As this possibility is assumed not to hold, we have that $\Gamma$ does not act essentially transitively on $(X,\nu)$.

Now if both $G_{1}$ and $G_{2}$ have property $(T)$ then $\Gamma$ also has property $(T)$ (being a lattice) and therefore $\Gamma \actson (X,\nu)$ is essentially transitive since being weakly amenable, it is orbit equivalent to an action of $\mathbb{Z}$ by Theorem \ref{T:cfw} and the corresponding cocycle into $\mathbb{Z}$ must take values in a compact (finite) subgroup.

Without loss of generality, we may therefore assume that $G_{2}$ has property $(T)$ and that $G_{1}$ does not.  Then $\mathrm{proj}_{1} : G \to G_{1}$ is a resolution by Propositon \ref{P:reslatts}.  Since $\Gamma$ is a lattice in $G$, by Proposition \ref{P:reslatts}, the map $\mathrm{proj}_{1} : \Gamma \to \overline{\mathrm{proj}_{1}~\Gamma} = G_{1}$ is also a resolution.

Let $\pi : \Gamma \to \mathcal{U}(L^{2}(X,\nu))$ be the Koopman representation.
By Proposition \ref{P:aivecs}, since $\Gamma \actson (X,\nu)$ is weakly amenable but not essentially transitive, there exists a sequence $f_{n} \in L^{2}(X,\nu)$ that are not $\Gamma$-invariant but are almost invariant.

Let $\{ \gamma_{n} \}$ and $\{ \gamma_{n}^{\prime} \}$ be sequences in $\Gamma$ such that $\mathrm{proj}_{1}~\gamma_{n} \to g$ in $G_{1}$ and that $\mathrm{proj}_{1}~\gamma_{n}^{\prime} \to g$ in $G_{1}$.  Then the sequence $\{ a_{n} \}$ given by $a_{2n} = \gamma_{n}$ and $a_{2n+1} = \gamma_{n}^{\prime}$ also has the property that $\mathrm{proj}_{1}~a_{n} \to g$ in $G_{1}$.  Consider the space of $G_{1}$-points in $L^{2}(X,\nu)$:
\[
\mathcal{F} = \{ f \in L^{2}(X,\nu) : g \mapsto \pi(g)f \text{ factors through } \mathrm{proj}_{1} \}.
\]
By Proposition \ref{P:closedQpoints}, this is a closed $\Gamma$-invariant space.  For $f \in \mathcal{F}$, then $\pi(\gamma_{n})f \to q \in \mathcal{F}$ and $\pi(\gamma_{n}^{\prime}) \to q^{\prime} \in \mathcal{F}$ and $\pi(a_{n})f \to q^{\prime\prime} \in \mathcal{F}$.  By the construction of $\{ a_{n} \}$ then $q^{\prime\prime} = q = q^{\prime}$ so we can define an action of $G_{1}$ on $\mathcal{F}$ by
\[
g_{1} \cdot f = \lim \pi(\gamma_{n})f \text{ for any $\{ \gamma_{n} \}$ such that $\mathrm{proj}_{1}~\gamma_{n} \to g_{1}$ in $G_{1}$. }
\]
Since $\mathcal{F}$ is a closed $\Gamma$-invariant subalgebra, the point realization $(Y,\eta)$ corresponding to it is a $G_{1}$-space that is a $\Gamma$-quotient of $(X,\nu)$.  Now $\Gamma \actson (X,\nu)$ ergodically hence $G_{1} \actson (Y,\eta)$ ergodically and there exists a $\Gamma$-map $\psi : (X,\nu) \to (Y,\eta)$.

Since $\mathrm{proj}_{1}$ is a resolution, there exists a sequence $\{ q_{n} \}$ of $G_{1}$-almost invariant, but not invariant, functions in $L^{2}(Y,\eta)$ and in particular, $(Y,\eta)$ is nontrivial.  Since $\Gamma \actson (X,\nu)$ is weakly amenable, the same holds for $\Gamma \actson (Y,\eta)$.

Note that if $G_{1}$ is connected then $\mathrm{proj}_{1}~\Gamma$ acts essentially freely on $(Y,\eta)$ by Theorem \ref{T:countableessfree}.  Therefore $\mathrm{proj}_{1}~\stab_{\Gamma}(x) = \{ e \}$ almost everywhere.  So $\stab_{\Gamma}(x) \subseteq \{ e \} \times G_{2}$ almost surely.  But $\Gamma$ is irreducible so $\Gamma \cap \{ e \} \times G_{2} = \{ e \}$ by Proposition \ref{P:irrlattint}.  Then $\stab_{\Gamma}(x) = \{ e \}$ almost surely so $\Gamma \actson (X,\nu)$ is essentially free which we have assumed is not the case.

Clearly $\mathrm{proj}_{1}~\stab_{\Gamma}(x) \subseteq \stab_{G_{1}}(\psi(x))$ and therefore $\overline{\mathrm{proj}_{1}~\stab_{\Gamma}(x)} \subseteq \stab_{G_{1}}(\psi(x))$ for all $x \in X$.  Let $(Z,\zeta)$ be the quotient space of $(X,\nu)$ by the map $\varphi(x) = \overline{\mathrm{proj}_{1}~\stab_{\Gamma}(x)}$.  Then by the universal property of quotient spaces there exist $\Gamma$-maps
\[
(X,\nu) \to PRG(X,\nu) \to (Z,\zeta) \to (Y,\eta).
\]
Suppose that $G_{1} \actson (Y,\eta)$ is essentially transitive.  Then, taking a continuous compact model for the $G$-action on $Y$, there exists $y_{0} \in Y$ such that $G_{1} \cdot y_{0}$ is homeomorphic to $G_{1} / \stab_{G_{1}}(y_{0})$ and $G_{1} \cdot y_{0}$ has a finite $G_{1}$-invariant measure (since $\eta(G_{1} \cdot y_{0}) = 1$) meaning that $\stab_{G_{1}}(y_{0}) = \Lambda$ is a lattice in $G_{1}$ or else that $(Y,\eta)$ is purely atomic (since the action is ergodic).

Consider first the case when $(Y,\eta)$ is purely atomic.  Since $G_{1}$ acts continuously, then $(Y,\eta)$ is trivial (as $G_{1}$ is nondiscrete and acts continuously) by ergodicity.  But then the sequence of $G_{1}$-almost invariant vectors are in fact invariant, a contradiction.  So $(Y,\eta)$ is not atomic.

Therefore we are left with the case when $\overline{\mathrm{proj}_{1}~\stab_{\Gamma}(x)}$ is contained in a $G_{1}$-conjugate of a fixed lattice $\Lambda < G_{1}$ almost surely.  Then $\mathrm{proj}_{1}~\stab_{\Gamma}(x)$ is discrete almost surely so $\overline{\mathrm{proj}_{1}~\stab_{\Gamma}(x)} = \mathrm{proj}_{1}~\stab_{\Gamma}(x)$.  

Let $\pi : (X,\nu) \to (Y,\eta)$ be the $\Gamma$-map.  For $\gamma \in \Gamma$, let $E_{\gamma} = \{ x \ in X : \gamma x = x \}$.  Then $\nu(E_{\gamma}) > 0$ for infinitely many $\gamma$ (since we have assumed the stabilizers are infinite almost surely).  Then for all $x \in E_{\gamma}$ it holds that $\mathrm{proj}_{1}~\gamma \in \stab_{G_{1}}(\pi(x))$.  Let $F = \pi(E_{\gamma})$.  Then $\eta(F) > 0$.  So there exists a positive Haar measure set $Q \subseteq G$ such that for $g \in Q$, it holds that $\mathrm{proj}_{1}~\gamma \in \stab(g y_{0}) = g \Lambda g^{-1}$.  Then $g^{-1}\mathrm{proj}_{1}~\gamma g \in \Lambda$ for all $g \in Q$.  But $\Lambda$ is discrete and $Q$ has positive Haar measure meaning that $G_{1}$ is then discrete, a contradiction.
\end{proof}

Combining our work with the results in \cite{CP12}, we obtain:
\begin{corollary}\label{C:lattices2}
Let $G_{1}$ and $G_{2}$ be simple nondiscrete noncompact locally compact second countable groups with the Howe-Moore property such that at least one $G_{j}$ has property $(T)$.  Set $G = G_{1} \times G_{2}$ and let $\Gamma < G$ be an irreducible lattice.  Let $(X,\nu)$ be an ergodic measure-preserving $\Gamma$-space.  Then one of the following holds:
\begin{itemize}
\item $\Gamma \actson (X,\nu)$ is essentially free;
\item $\stab_{*}\nu$ is supported on the finite index subgroups of $\Gamma$; or
\item $\stab_{*}\nu$ is supported on the torsion elements of $\Gamma$; or
\item one $G_{j}$ is totally disconnected, has property $(T)$ and acts ergodically and essentially freely on the induced space $G \times_{\Gamma} X$ and the other $G_{3-j}$ is connected, does not have property $(T)$ and does not act ergodically on the induced space.
\end{itemize}
\end{corollary}
\begin{proof}
By the previous corollary, if none of the first three possibilities occur then one of:
\begin{itemize}
\item one $G_{j}$ is totally disconnected and acts ergodically and essentially freely on the induced space $G \times_{\Gamma} X$ and the other $G_{3-j}$ does not act ergodically on the induced space; or
\item one $G_{j}$ does not have property $(T)$, is totally disconnected, and there is a nontrivial ergodic $G_{j}$-space $(Y,\eta)$ that is a $\Gamma$-quotient of $(X,\nu)$ and such that $G_{j} \actson (Y,\eta)$ is weakly amenable but not essentially transitive.
\end{itemize}
The result in \cite{CP12} (Theorem \ref{T:CP12}) rules out the second case since in that case either both groups are totally disconnected or the connected group has property $(T)$.  Likewise, in the first case, the only possibility not covered by \cite{CP12} is that there is a connected group in the product that does not have property $(T)$.
\end{proof}

\section{Higher-Order Product Groups}

We now generalize the results of the previous two sections to products of arbitrarily (finitely) many groups and irreducible lattices in such products.

\subsection{The Higher-Order Product Random Subgroups Functor}

\begin{definition}
Let $G_{j}$ be locally compact second countable groups for $j = 1, \ldots, k$.  Set $G = G_{1} \times \cdots \times G_{k}$.  Given a $G$-space $(X,\nu)$, define $PRG_{G_{1},G_{2},\ldots,G_{k}}^{(k)}(X,\nu)$ to be the quotient space of $(X,\nu)$ by the map 
\[
\varphi(x) = \overline{\mathrm{proj}_{G_{1}}~\stab(x)} \times \cdots \times \overline{\mathrm{proj}_{G_{k}}~\stab(x)}.
\]
\end{definition}

Note that $PRG_{G_{1}}^{(1)}(X,\nu) = (X,\nu)$ in the case when there is a single group.

\begin{proposition}\label{P:PRGchain}
Let $G_{j}$ be locally compact second countable groups for $j = 1, \ldots, k$ where $k \geq 2$.  Set $G = G_{1} \times \cdots \times G_{k}$ and let $(X,\nu)$ be a $G$-space.  Then 
\[
PRG_{G_{1},G_{2} \times \cdots \times G_{k}}^{(2)}(PRG_{G_{1} \times G_{2}, G_{3},\ldots,G_{k}}^{(k-1)}(X,\nu)) = PRG_{G_{1},\ldots,G_{k}}^{(k)}(X,\nu).
\]
\end{proposition}
\begin{proof}
This will follow from the universal property of the quotient space by a random subgroup (Theorem \ref{T:irsfactor}).  Let
\begin{diagram}
(X,\nu) &\rTo^{\pi} &PRG_{G_{1}\times G_{2},G_{3},\ldots,G_{k}}^{(k-1)}(X,\nu) &\rTo^{\psi} &PRG_{G_{1}, G_{2}\times \cdots \times G_{k}}^{(2)}(PRG_{G_{1}\times G_{2},G_{3},\ldots,G_{k}}^{(k-1)}(X,\nu))
\end{diagram}
be the $G$-maps defining the quotient spaces.  Observe that since
\[
\overline{\mathrm{proj}_{G_{1}\times G_{2}}~\stab(x)} \times \overline{\mathrm{proj}_{G_{3}}~\stab(x)} \times \cdots \times \overline{\mathrm{proj}_{G_{k}}~\stab(x)} \subseteq \stab(\pi(x))
\]
it holds that
\[
\overline{\mathrm{proj}_{G_{1}}~\stab(\pi(x))} \supseteq \overline{\mathrm{proj}_{G_{1}}~\overline{\mathrm{proj}_{G_{1}\times G_{2}}~\stab(x)}} \supseteq \overline{\mathrm{proj}_{G_{1}}~\stab(x)}
\]
and likewise that
\[
\overline{\mathrm{proj}_{G_{2}\times\cdots\times G_{k}}~\stab(\pi(x))} \supseteq \overline{\mathrm{proj}_{G_{2}}~\stab(x)} \times \cdots \times \overline{\mathrm{proj}_{G_{k}}~\stab(x)}.
\]
Hence by the universal property of $PRG^{(k)}$ there exist $G$-maps
\[
(X,\nu) \to PRG_{G_{1},\ldots,G_{k}}^{(k)}(X,\nu) \to PRG_{G_{1}, G_{2}\times \cdots \times G_{k}}^{(2)}(PRG_{G_{1}\times G_{2},G_{3},\ldots,G_{k}}^{(k-1)}(X,\nu)).
\]
On the other hand, since
\begin{align*}
\overline{\mathrm{proj}_{G_{1}}~\stab(x)} \times &\cdots \times \overline{\mathrm{proj}_{G_{k}}~\stab(x)} \\ 
&\supseteq~\overline{\mathrm{proj}_{G_{1} \times G_{2}}~\stab(x)} \times \overline{\mathrm{proj}_{G_{3}}~\stab(x)} \times \cdots \times \overline{\mathrm{proj}_{G_{k}}~\stab(x)}
\end{align*}
by the universal property of $PRG^{(k-1)}$ there are $G$-maps
\[
(X,\nu) \to PRG_{G_{1}\times G_{2},G_{3},\ldots,G_{k}}^{(k-1)}(X,\nu) \to PRG_{G_{1},\ldots,G_{k}}^{(k)}(X,\nu).
\]
Then by the universal property of $PRG^{(2)}$ and the obvious inclusion of the stabilizers there exist $G$-maps
\begin{align*}
PRG_{G_{1}\times G_{2},G_{3},\ldots,G_{k}}^{(k-1)}(X,\nu) &\to PRG_{G_{1}, G_{2}\times \cdots \times G_{k}}^{(2)}(PRG_{G_{1}\times G_{2},G_{3},\ldots,G_{k}}^{(k-1)}(X,\nu)) \\
&\to PRG_{G_{1},\ldots,G_{k}}^{(k)}(X,\nu)
\end{align*}
and we therefore conclude that
\[
PRG_{G_{1},\ldots,G_{k}}^{(k)}(X,\nu) = PRG_{G_{1}, G_{2}\times \cdots \times G_{k}}^{(2)}(PRG_{G_{1}\times G_{2},G_{3},\ldots,G_{k}}^{(k-1)}(X,\nu)).
\]
\end{proof}

\subsection{Actions of Higher-Order Product Groups}

\begin{theorem}\label{T:71prime}
Let $G_{j}$ be locally compact second countable groups for $j = 1, \ldots, k$.  Set $G = G_{1} \times \cdots \times G_{k}$ and let $(X,\nu)$ be a measure-preserving $G$-space.  If $G \actson PRG_{G_{1},\ldots,G_{k}}^{(k)}(X,\nu)$ weakly amenably then $G \actson (X,\nu)$ weakly amenably.
\end{theorem}
\begin{proof}
By Proposition \ref{P:PRGchain} since $G \actson PRG_{G_{1},\ldots,G_{k}}^{(k)}(X,\nu)$ weakly amenably, it then holds that $G \actson PRG_{G_{1},G_{2} \times \cdots \times G_{k}}^{(2)}(PRG_{G_{1} \times G_{2}, G_{3},\ldots,G_{k}}^{(k-1)}(X,\nu))$ weakly amenably.  Then, by Theorem \ref{T:PRGweakamen}, we also have $G \actson PRG_{G_{1} \times G_{2}, G_{3},\ldots,G_{k}}^{(k-1)}(X,\nu)$ weakly amenably.  Proceeding inductively, we then have that $G \actson PRG_{G_{1} \times \cdots \times G_{j}, G_{j+1}, \ldots, G_{k}}^{(k-j+1)}(X,\nu)$ weakly amenably for all $1 \leq j \leq k$.  Hence in particular, it holds that $G \actson PRG_{G_{1} \times \cdots \times G_{k}}^{(1)}(X,\nu) = (X,\nu)$ weakly amenably.
\end{proof}

\begin{theorem}\label{T:72prime}
Let $G_{j}$ be locally compact second countable groups for $j = 1, \ldots, k$ with $k \geq 2$.  Set $G = G_{1} \times \cdots \times G_{k}$ and let $(X,\nu)$ be an ergodic measure-preserving $G$-space.  Let $(X,\nu) \to (X_{j},\nu_{j})$ be the ergodic decomposition into $\widetilde{G}_{j} = G_{1} \times \cdots \times G_{j-1} \times \{ e \} \times G_{j+1} \times \cdots \times G_{k}$-ergodic components.  Assume that $G_{j} \actson (X_{j},\nu_{j})$ weakly amenably and that $\stab_{*}\nu_{j}$ is a simple invariant random subgroup for all $j=1,\ldots,k$.  Then either there exists at least one $G_{j}$ such that $\mathrm{proj}_{G_{j}}~\stab(x) = \{ e \}$ almost everywhere or else $G \actson (X,\nu)$ weakly amenably.
\end{theorem}
\begin{proof}
Consider the maps $s_{j}(x) = \overline{\mathrm{proj}_{G_{j}}~\stab(x)}$.  Since $s_{j}(x)$ is a $\widetilde{G}_{j}$-invariant function, it descends to a function on $(X_{j},\nu_{j})$.  Therefore by Theorem \ref{T:crazy}, $(s_{j})_{*}\nu \normal \stab_{*}\nu_{j}$.  Since the set $\{ x \in X : s_{j}(x) = \{ e \} \}$ is $G$-invariant, by ergodicity either $(s_{j})_{*}\nu = \delta_{\{e\}}$ or else $(s_{j})_{*}\nu = \stab_{*}\nu_{j}$ for each $j$.  If $(s_{j})_{*}\nu = \delta_{\{e\}}$ then the conclusion follows.  So we may assume that $(s_{j})_{*}\nu = \stab_{*}\nu_{j}$ for all $j = 1,\ldots,k$.  This says precisely that $PRG_{G_{1},\ldots,G_{k}}^{(k)}(X,\nu)$ is orbital over $(X_{1} \times \cdots \times X_{k}, \nu_{1} \times \cdots \times \nu_{k})$ (which is a quotient of the product random subgroups functor by the universal property since the $\widetilde{G}_{j}$-ergodic components are a quotient by an invariant random subgroup with larger stabilizers).  Then the fact that each $G_{j}$ acts weakly amenably on $(X_{j},\nu_{j})$ says that $G \actson (X_{1} \times \cdots \times X_{k},\nu_{1} \times \cdots \times \nu_{k})$ weakly amenably which in turn means, by Proposition \ref{P:weakamen}, that $G \actson PRG_{G_{1},\ldots,G_{k}}^{(k)}(X,\nu)$ weakly amenably.  Then, by Theorem \ref{T:71prime}, $G \actson (X,\nu)$ weakly amenably.
\end{proof}

\begin{theorem}\label{T:73prime}
Let $G_{j}$ be locally compact second countable groups for $j = 1, \ldots, k$ with $k \geq 2$.  Set $G = G_{1} \times \cdots \times G_{k}$ and let $(X,\nu)$ be an ergodic measure-preserving $G$-space.  If $\mathrm{proj}_{G_{j}}~\stab(x)$ is dense in $G_{j}$ almost everywhere for each $j = 1,\ldots,k$ then $G \actson (X,\nu)$ weakly amenably.
\end{theorem}
\begin{proof}
When the projections are all dense, $G \actson PRG_{G_{1},\ldots,G_{k}}^{(k)}(X,\nu)$ trivially.  By ergodicity, it is then the trivial one-point space.  As every group acts weakly amenably on a point, the conclusion follows from Theorem \ref{T:71prime}.
\end{proof}

\begin{corollary}\label{C:actprodTreal}
Let $G_{j}$ be locally compact second countable groups for $j = 1, \ldots, k$ with $k \geq 2$ each with property $(T)$.  Set $G = G_{1} \times \cdots \times G_{k}$ and let $(X,\nu)$ be an ergodic measure-preserving $G$-space.
Assume that there exist simple closed subgroups $H_{j} < G_{j}$ such that the spaces of $\prod_{\ell \ne j}G_{\ell}$-ergodic components is isomorphic to $(G_{j}/H_{j},\mathrm{Haar})$ for each $j$  and such that any nontrivial normal subgroup of $H_{j}$ has finite index in $H_{j}$.  Then either at least one $G_{j} \actson (X,\nu)$ essentially free or $G \actson (X,\nu)$ is essentially transitive.
\end{corollary}
\begin{proof}
The same reasoning as in Corollary \ref{C:actprodT} gives that if none of the $G_{j}$ act essentially freely then the $G$-action on $PRG_{G_{1},\ldots,G_{k}}^{(k)}(X,\nu)$ is weakly amenable, hence $G \actson (X,\nu)$ is weakly amenable.  Since $G$ has property $(T)$, the action is then essentially transitive.
\end{proof}

\subsection{Actions of Lattices in Higher-Order Product Groups}

\begin{theorem}\label{T:83prime}
Let $G = G_{1}, \times \cdots \times G_{k}$ be a product of at least two simple nondiscrete noncompact locally compact second countable groups with the Howe-Moore property.  Let $\Gamma < G$ be an irreducible lattice and let $(X,\nu)$ be an ergodic measure-preserving $\Gamma$-space.  Then one of the following holds:
\begin{itemize}
\item $\Gamma \actson (X,\nu)$ is essentially free;
\item $\stab_{*}\nu$ is supported on the torsion elements of $\Gamma$;
\item $\Gamma \actson (X,\nu)$ is weakly amenable; or
\item at least one $G_{j}$ is totally disconnected.
\end{itemize}
\end{theorem}
\begin{proof}
Assume the action is not essentially free.  For $\gamma \in \Gamma$, let $E_{\gamma} = \{ x \in X : \gamma x = x \}$.  Let $L = \{ \gamma \in \Gamma \setminus \{ e \} : \nu(E_{\gamma}) > 0 \}$.  Then $L$ is nonempty since the action is not essentially free.
For each $j$, let $s_{j}(x) = \overline{\mathrm{proj}_{G_{j}}~\stab(x)}$.  Then $(s_{j})_{*}\nu$ is an invariant random subgroup of $G_{j}$ by Theorem \ref{T:81}.  Let $(Y_{j},\eta_{j}$ be the ergodic $(s_{j})_{*}\nu$-nonfree action of $G_{j}$.  If there exists $\gamma \in L$ such that $\overline{\langle \mathrm{proj}_{G_{j}}~\gamma \rangle}$ is noncompact then, as in the proof of Theorem \ref{T:stuff}, since $G_{j}$ has the Howe-Moore property, $s_{j}(x) = G_{j}$ almost everywhere.
Define the set
\[
S = \{ j \in \{ 1, \ldots, k \} : \exists~\gamma \in L \text{ such that } \overline{\langle \mathrm{proj}_{G_{j}}~\gamma \rangle} \text{ is noncompact} \}.
\]
Then for every $j \in S$, it holds that $s_{j}(x) = G_{j}$ almost everywhere.

Consider now the set
\[
T = \{ j \in \{ 1, \ldots, k \} : G_{j} \text{ is connected} \}.
\]
Let $j \in T$.  Then $G_{j}$ is a simple real Lie group since it has the Howe-Moore property (Theorem \ref{T:connHMLie}).  Since $\mathrm{proj}_{G_{j}}~\Gamma$ is a countable subgroup of $G_{j}$, by Theorem \ref{T:countableessfree}, either $\mathrm{proj}_{G_{j}}~\Gamma$ acts essentially freely on $(Y_{j},\eta_{j})$ or else $(Y_{j},\eta_{j})$ is trivial.  For $\gamma \in L$, since $\nu(E_{\gamma}) > 0$, on a $(s_{j})_{*}\nu$-positive measure set $\mathrm{proj}_{G_{j}}~\gamma$ is in the stabilizer of $y \in Y_{j}$.  But $\mathrm{proj}_{G_{j}}~\Gamma$ acts essentially freely so this is a contradiction.  Therefore $(Y_{j},\eta_{j})$ is trivial.  This means that $s_{j}(x) = G_{j}$ almost everywhere.

If $S$ is empty then every $\gamma \in L$ has the property that $\langle \gamma \rangle$ is contained in a compact group.  Since $\Gamma$ is discrete, this means that $\gamma$ has a finite orbit hence is a torsion element.  So if $S$ is empty then $\stab_{*}\nu$ is supported on the torsion elements of $\Gamma$.

If $|S \cup T| = k$ then $PRG_{G_{1},\ldots,G_{k}}^{(k)}(G \times_{\Gamma} X)$ has the property that almost every stabilizer projects densely into each of the $G_{j}$.  Then, by Theorem \ref{T:73prime}, $G \actson G \times_{\Gamma} X$ is weakly amenable and so, by Proposition \ref{P:weakamenlattice}, $\Gamma \actson (X,\nu)$ is weakly amenable.

Therefore we are left with the case when there exists some $j \notin S \cup T$ and $S$ is nonempty.  Then $G_{j}$ is totally disconnected.
\end{proof}

\begin{corollary}\label{C:85prime}
Let $G = G_{1}, \times \cdots \times G_{k}$ be a product of at least two simple nondiscrete noncompact locally compact second countable groups with the Howe-Moore property, at least one with property $(T)$.  Let $\Gamma < G$ be an irreducible lattice and let $(X,\nu)$ be an ergodic measure-preserving $\Gamma$-space.  Then one of the following holds:
\begin{itemize}
\item $\Gamma \actson (X,\nu)$ is essentially free;
\item $\stab_{*}\nu$ is supported on the torsion elements of $\Gamma$;
\item $\stab_{*}\nu$ is supported on the finite index subgroups of $\Gamma$; or
\item at least one $G_{j}$ is totally disconnected and at least one $G_{j}$ is connected and none of the connected $G_{j}$ have property $(T)$.
\end{itemize}
\end{corollary}
\begin{proof}
By Theorem \ref{T:83prime}, if neither of the first two possibilities occur then either $\Gamma \actson (X,\nu)$ is weakly amenable or at least one $G_{j}$ is totally disconnected.  Consider first when $\Gamma \actson (X,\nu)$ is weakly amenable.  Suppose that $\Gamma \actson (X,\nu)$ is not essentially transitive (that is, $\stab_{*}\nu$ is not supported on the finite index subgroups of $\Gamma$).  Then, as in the proof of Corollary \ref{C:lattices}, there exists a totally disconnected $G_{j}$ and a $\Gamma$-quotient of $(X,\nu)$ that is a $G_{j}$-space on which $\Gamma$ acts weakly amenably.  So we are left with the case when there is a totally disconnected $G_{j}$.

Then, by \cite{CP12} (Theorem \ref{T:CP12}), the only case left is when there is a connected simple factor that does not have property $(T)$.
Moreover, in \cite{CP12} Corollary 5.2, the requirement that all the connected factors have property $(T)$ is used in the following way: one gets an irreducible lattice $\Gamma_{0}$ in the product of the connected factors acting on a space $(X_{0},\nu_{0})$ weakly amenably and then uses property $(T)$ to conclude the action has either finite orbits or finite stabilizers.  Applying our work to $\Gamma_{0}$ in the product of the connected factors, the proof of Corollary \ref{C:lattices} then gives that one connected factor having property $(T)$ is enough to conclude $\Gamma_{0}$ acts with either finite stabilizers or finite index stabilizers.  Then the proof of \cite{CP12} Corollary 5.2 goes through when only one connected factor has property $(T)$ and so $\Gamma \actson (X,\nu)$ weakly amenably necessarily implies the action is essentially transitive.
\end{proof}

\subsection{Actions of Semisimple Groups and Lattices}

We conclude with  a strengthening of the results on actions of semisimple real Lie groups and irreducible lattices in them due to Nevo-Stuck-Zimmer \cite{SZ94},\cite{nevozimmer}.  We remark that our methods give a more general statement than theirs, except in the case of a lattice in a simple higher-rank Lie group, in which case the only known proof is the algebraic proof they give (as opposed to the more geometric methods we employ):

\begin{corollary}\label{C:nzbetterambient}
Let $G$ be a semisimple group with trivial center and no compact factors, at least one simple factor having property $(T)$.  Let $(X,\nu)$ be an ergodic $G$-space such that each simple factor of $G$ acts ergodically on $(X,\nu)$.  Then $G \actson (X,\nu)$ is essentially free or essentially transitive.
\end{corollary}
\begin{proof}
The case when $G$ is simple is covered by Nevo-Stuck-Zimmer \cite{SZ94},\cite{nevozimmer}.  The case when $G$ has at least two simple factors follows from Corollary \ref{C:bsnew4}.
\end{proof}

\begin{corollary}\label{C:latticesfinalstate}
Let $G$ be a semisimple group with trivial center and no compact factors with at least one simple factor being a connected (real) Lie group with property $(T)$.  Let $\Gamma < G$ be an irreducible lattice and $(X,\nu)$ be a nonatomic ergodic measure-preserving $\Gamma$-space.  Then $\Gamma \actson (X,\nu)$ is essentially free.
\end{corollary}
\begin{proof}
When $G$ has a single simple factor, by hypothesis then $G$ is a simple real Lie group with property $(T)$ hence the results of Nevo-Stuck-Zimmer \cite{SZ94},\cite{nevozimmer} give the conclusion.  When $G$ has at least two simple factors, Corollary \ref{C:85prime} states that if the action is not essentially free then either $\stab_{*}\nu$ is supported on the torsion elements of $\Gamma$ or $\stab_{*}\nu$ is supported on the finite index subgroups of $\Gamma$ (the final possibility in that Corollary is ruled out by our hypothesis that there is a simple connected factor with property $(T)$).
If $\stab_{*}\nu$ is supported on the finite index subgroups of $\Gamma$ then by ergodicity, $(X,\nu)$ is finite and atomic and we have assumed $(X,\nu)$ is nonatomic.  Suppose $\gamma \in \Gamma$ is torsion, so $\gamma^{m} = e$ for some $m \in \mathbb{N}$.  Then $\mathrm{proj}_{G_{j}}~\gamma$ is torsion in $G_{j}$ but $G_{j}$ is simple and connected hence torsion-free.  Therefore the action is essentially free.
\end{proof}

\dbibliography{references}

\end{document}